\documentclass[12pt]{amsart}
\usepackage{amssymb}
\usepackage{multirow}
\usepackage{here}
\usepackage{cases}
\usepackage {empheq}
\usepackage{url}

\usepackage{amsthm}

\newtheorem{thm}{Theorem}
\newtheorem{lem}{Lemma}[section]
\newtheorem{prop}{Proposition}[section]

\newtheorem{conj}{Conjecture}
\newtheorem{definition}{Definition}[section]
\newtheorem{rem}{Remark}

\newtheorem*{theorem}{Theorem}

\newcommand{\N}{\mathbb{N}}
\newcommand{\Z}{\mathbb{Z}}
\newcommand{\m}{\mathcal{M}}
\newcommand{\h}{\mathcal{H}}

\numberwithin{equation}{section}
\if0
\if0
\fi
\fi
\if0
\fi
\begin{document}
\title
{Number of solutions to a special type of \\ unit equations in two variables}

\author{Takafumi Miyazaki}
\address{Takafumi Miyazaki
\hfill\break\indent Gunma University, Division of Pure and Applied Science,
\hfill\break\indent Faculty of Science and Technology
\hfill\break\indent Tenjin-cho 1-5-1, Kiryu 376-8515.
\hfill\break\indent Japan}
\email{tmiyazaki@gunma-u.ac.jp}

\author{Istv\'an Pink}
\address{Istv\'an Pink
\hfill\break\indent University of Debrecen, Institute of Mathematics
\hfill\break\indent H-4002 Debrecen, P.O. Box 400.
\hfill\break\indent Hungary}
\email{pinki@science.unideb.hu}

\thanks{The first author was supported by JSPS KAKENHI (No.s 16K17557, 20K03553).
The second author was supported in part by the NKFIH grants ANN130909, K115479, K128088 and by the project EFOP3.6.1-16-2016-000.}

%\date{\today}

\maketitle

\markleft{Takafumi MIYAZAKI \& Istv\'an PINK}
\markright{%Number of solutions to a
Special type of unit equations in two variables}

\vspace{-0.3cm}\centerline
{\sl\footnotesize Dedicated to Professor K\'alm\'an Gy{\H o}ry on the occasion of his 80th birthday}
\centerline{\sl\footnotesize and to Professor Hirofumi Tsumura on the occasion of his 60th birthday}

\vspace{-0.1cm}\begin{abstract}
For any fixed coprime positive integers $a,b$ and $c$ with $\min\{a,b,c\}>1$, we prove that the equation $a^x+b^y=c^z$ has at most two solutions in positive integers $x,y$ and $z$, except for one specific case which exactly gives three solutions.
Our result is essentially sharp in the sense that there are infinitely many examples allowing the equation to have two solutions in positive integers.
From the viewpoint of a well-known generalization of Fermat's equation, it is also regarded as a 3-variable generalization of the celebrated theorem of Bennett [M.A.Bennett, On some exponential equations of S.S.Pillai, Canad. J. Math. 53(2001), no.2, 897--922] which asserts that Pillai's type equation $a^x-b^y=c$ has at most two solutions in positive integers $x$ and $y$ for any fixed positive integers $a,b$ and $c$ with $\min\{a,b\}>1$.
\end{abstract}

%\tableofcontents

%%%%%%%%%%%%%%%
\section{Introduction}
%%%%%%%%%%%%%%%

The history of the $S$-unit equations related to Diophantine equations is very rich (cf.~\cite{Gy3,EvGyStTi,EvScSch,EvGy2}).
Indeed, many diophantine problems can be reduced to $S$-unit equations over number fields.
Especially, the simplest one among those is the $S$-unit equation in two unknowns over rational number field, which is written as follows:
\begin{equation}\label{sunit}
\alpha X+\beta Y+\gamma Z=0,
\end{equation}
where $\alpha,\beta,\gamma$ are given non-zero integers, and $X,Y,Z$ are unknown integers composed of finitely many given primes.
The set of the predetermined primes for unknowns is as usual denoted by $S$.
From the theory on Diophantine approximations we know that there are only finitely many solutions to equation \eqref{sunit}, and effective upper bounds for their sizes can be obtained by means of Baker's theory of linear forms in logarithms (cf.~\cite{Gy,Gy2,BuGy,GyYu}).
Since any unknown in \eqref{sunit} can be expressed as a product of powers of given primes, equation \eqref{sunit} is an exponential Diophantine equation.
Based upon this, one of the simplest examples of equation \eqref{sunit} is
\begin{equation}\label{pillai}
a^x-b^y=c,
\end{equation}
where $a,b,c$ are fixed positive integers with $\min\{a,b\}>1$, and $x,y$ are unknown positive integers.
This equation is a special case of Pillai's equation, and Pillai's famous conjecture says that there are only finitely many pairs of distinct powers with their difference fixed.
It is also worth noting that the case where $c=1$ corresponds to a special one of Catalan's equation (cf.~\cite{Mi}).

In a series of papers in the 1930's, Pillai \cite{Pi,Pi2} actively studied equation \eqref{pillai} and obtained some finiteness results on the number of solutions (for more detail see \cite{Be_Canad2001, Be_JNT2003}).
Early 1990's, Scott \cite{Sc} extensively investigated equation \eqref{pillai} in the case where $a$ is a prime with a motivation to a classical problem listed in R.K.Guy's book, and he used strictly elementary methods in quadratic fields to obtain very sharp upper bounds for the number of solutions in several cases.
For more details on these topics or some other related ones, see for example \cite{Be_JNT2003,BerHa,ShTi}.

In the direction on the number of solutions to equation \eqref{pillai}, Bennett \cite{Be_Canad2001} established the following definitive result.

\begin{theorem}[M.A.~Bennett]\label{pillai_atmost2}
For any fixed positive integers $a,b$ and $c$ with $\min\{a,b\}>1,$ equation \eqref{pillai} has at most two solutions in positive integers $x$ and $y.$
\end{theorem}

This result is essentially sharp in the sense that there are a number of examples where there are two solutions to equation \eqref{pillai} (cf.~\cite[(1.2)]{Be_Canad2001}).
The proof of Bennett uses lower bounds for linear forms in logarithms of two algebraic numbers together with a `gap principle', based upon an arithmetic of exponential congruences, which gives rise to a large gap among three hypothetical solutions.
It should be also remarked that the non-coprimality case, i.e., $\gcd(a,b)>1$ is handled just by a short elementary observation.
Several other works to estimate the number of solutions of more general equations of type \eqref{pillai} were made available in the literature (cf.~\cite{ScSt_jnt2006,ScSt_jtnb2013,ScSt_jtnb2015}).

There is another example of equation \eqref{sunit}, which is not only regarded as a 3-variable generalization of equation \eqref{pillai}, but also closely related to the so-called generalized Fermat equation, that is, a Diophantine equation arising from the quest to seek for all the relation that the sum of two powers of `coprime' positive integers is equal to another power (cf.~\cite{BeMiSi}, \cite[Ch.14]{Co}).
It is the main subject in this article, given as follows:
\begin{equation}\label{abc}
a^x+b^y=c^z,
\end{equation}
where $a,b,c$ are fixed coprime positive integers with $\min\{a,b,c\}>1$, and $x,y,z$ are unknown positive integers.
Equation \eqref{abc} is also the simplest one among purely exponential Diophantine equations.
It seems that the earliest published work on solving equation \eqref{abc} is due to Sierpi\'nski \cite{Si}, who solved the equation for $(a,b,c)=(3,4,5)$.
Just after this work, Je\'smanowicz \cite{Je} (who was a student of Sierpi\'nski) considered equation \eqref{abc} for other primitive Pythagorean triples, and he posed the problem to ask for determining the solutions to the equation for any Pythagorean triple.
His problem is the most known unsolved problem concerning equation \eqref{abc}.
In a series of papers in the 1990's, Terai gave some pioneer works on equation \eqref{abc} for some families of $(a,b,c)$ including Pythagorean triples, and later he posed several problems including a generalization of the mentioned one of Sierpi\'nski-Je\'smanowicz, now called Terai's conjecture (cf.~\cite{Te_aa99}).
For more details on these topics or some other related ones, see some recent papers \cite{CiMi,Lu2,Miy_aa18} and the references therein.

As mentioned before, from the theory of $S$-unit equations, we know that equation \eqref{abc} has at most finitely many solutions, moreover, the number of solutions can be bounded by an absolute large number (see the excellent survey of Gy\H{o}ry \cite{Gy4}).
However, apparently a kind of such estimates is far from the actual number, indeed, many existing works in the literature suggest that equation \eqref{abc} has few solutions in general.
Regarding to this topic, in the last few years, some important progresses have been made by some researchers.
In a series of papers, Hu and Le \cite{HuLe_aa15,HuLe_jnt18,HuLe_deb19} discussed equation \eqref{abc} over the irreducible residue class groups modulo powers of the base numbers of the equation, with various other elementary number theory methods including continued fractions.
As a result, they found a large gap among three hypothetical solutions, so-called `gap principle', and the combination of their gap principle and Baker's method implies that equation \eqref{abc} has at most two solutions whenever the maximal value of $a,b,c$ is sufficiently large (see Proposition \ref{Hu-Le} below).
Here remark that there are infinitely many examples where there are two solutions to equation \eqref{abc} (see \eqref{atmost1:ex} in the final section).
Since the exponential unknowns $x,y$ and $z$ are bounded by an explicit constant depending only on $a,b$ and $c$ by Baker's method, just a finite search remains to be done in order to obtain the definitive result on the number of solutions to equation \eqref{abc} corresponding to Bennett's mentioned theorem.
However, a kind of brute force computations is never enough to settle that finite search.
Related to this study, the work of Scott and Styer \cite{ScSt_deb16} should be referred.
They considerably improved the argument over quadratic fields in \cite{Sc} dealing with the case where $c$ is a prime, to obtain the same conclusion as that of the mentioned work of Hu and Le, whenever $c$ is odd (see Proposition \ref{Scott-Styer} below).
Actually, the main content of this article is to completely handle the remaining finite search mentioned before, as follows:

\begin{thm}\label{atmost2}
For any fixed coprime positive integers $a,b$ and $c$ with $\min\{a,b,c\}>1,$ equation \eqref{abc} has at most two solutions in positive integers $x,y$ and $z,$ except when $(a,b,c)$ is $(3,5,2)$ or $(5,3,2)$ which exactly gives three solutions.
\end{thm}

This result is essentially sharp as indicated in the mentioned work of Hu and Le, and its exceptional case comes from the identities $3^{}+5^{}=2^{3}, 3^{3}+5^{}=2^{5}$ and $3^{}+5^{3}=2^{7}$ (cf.~\eqref{atmost1:ex}).

The proof of our theorem proceeds under the assumption that all $a,b,c$ are explicitly finite and $c$ is even, and there are three main important steps.
The first step is to improve the gap principle of Hu and Le.
Their gap principle is derived by examining equation \eqref{abc} modulo each powers of $a,b$ and $c$, which is expressed as some divisibility relation or some inequality.
The main idea for the improving is to consider two congruences `simultaneously' in their treatment using modulus of each powers of the base numbers.
Roughly saying, this replaces a factor in the inequality from Hu and Le's gap principle as a common factor of two of exponential unknowns of the equation, where the value of an appearing factor is strictly restricted from the viewpoint of generalized Fermat equations.
In this stage, by combining the improved gap principle together with Baker's method in 2-adic case, the required bounds for $\max\{a,b,c\}$ for which theorem holds true is substantially reduced.
The second step is to find very sharp upper bounds for all the exponential unknowns of at least two of three equations, which in what follows are called the first two equations, coming from equation \eqref{abc} with the existence of three hypothetical solutions.
This is done elementarily by comparing the 2-adic valuations of both sides of each of those three equations, a procedure that works only in the case where $c$ is even which can be assumed from the mentioned work of Scott and Styer.
Working out these two steps together with several elementary number theory methods yields at most finitely many possible values of all letters concerning the first two equations.
Finally, in each of those cases, we check whether those two equations hold or not.
At this point it is worth noting that although the derived general bounds for all letters in those  equations are relatively sharp, a direct enumeration of all possible solutions of the system formed by the first two equations is impossible.
Therefore, we worked very carefully and found efficient methods for solving that system in reasonable computational time.

The organization of this article is as follows.
In the next section, we prepare some useful conditions which are consequences of previous existing results related to equations \eqref{pillai} and \eqref{abc}.
Section \ref{Sec-Uz} is devoted to find a sharp upper bound for $z$ using a result of Bugeaud \cite{bug-book} on the 2-adic estimate of the difference between two powers of algebraic numbers, where this choice is fitted for the second main idea in the proof.
On the other hand, we find some 2-adic properties on $z$ in Section \ref{Sec-z-2adic}, where one of those yields an exact information in a certain case.
We summarize the contents of these sections together with the second main idea in the forthcoming section, in particular, we deduce relatively small upper estimates of all exponential unknowns of the first two equations.
In Section \ref{Sec-improve-HuLe}, we improve the gap principle of Hu and Le, and give some of its applications in Section \ref{Sec-ineq-z1<z2}.
In Section \ref{Sec-gfe}, we quote several existing works on the generalized Fermat equation, and give their applications to the improved gap principle.
Section \ref{Sec-AB} is devoted to study a certain Diophantine equation related to equation \eqref{pillai}.
In Section \ref{Sec-z1=z2}, we use the consequences in the previous section together with the preparations in Sections \ref{Sec-improve-HuLe}, \ref{Sec-ineq-z1<z2} and \ref{Sec-gfe} to settle the case where the values of two $z$ of the first two equations coincide.
Sections \ref{Sec-bounds-z1<z2_c=max} and \ref{Sec-bounds-z1<z2_a=max} are devoted to exactly find all possible values of letters in the first two equations, and we sieve those completely in Sections \ref{Sec-sieve-z1<z2_c=max} and \ref{Sec-sieve-z1<z2_a=max}, and the proof is completed.
In the final section, we make a few remarks concerning an extension of Theorem \ref{atmost2}.

All computations in this paper were performed using the computer package MAGMA \cite{BoCaPl}.
The total computational time through this article did not exceed 25 hours.

%%%%%%%%%%%%%%%%%%%%%%%%%
\section{Previous results and their consequences}%
%%%%%%%%%%%%%%%%%%%%%%%%%

We begin with quoting the following results on equations \eqref{pillai} or \eqref{abc}, where the first two of them play important roles in the proof of our theorem.

\begin{prop}[\cite{HuLe_deb19}] \label{Hu-Le}
If $\max\{a,b,c\}>10^{62},$ then Theorem \ref{atmost2} is true.
\end{prop}

\begin{prop}[\cite{ScSt_deb16}] \label{Scott-Styer}
If $c$ is odd, then Theorem \ref{atmost2} is true.
\end{prop}

The following is a direct consequence of \cite[Theorem 6]{Sc}.

\begin{prop}\label{Scott}
If $c=2,$ then Theorem \ref{atmost2} is true.
\end{prop}

The following is a direct consequence of \cite[Theorem 1]{ScSt_jnt2006}, and it is a relevant analogue to Proposition \ref{pillai_atmost2}.

\begin{prop}\label{pillai_atmost2_relevant}
For any fixed positive integers $a,b$ and $c$ with $\min\{a,b\}$ $>1,$ the equation
\[a^x+b^y=c\] has at most two solutions in positive integers $x$ and $y$.
\end{prop}

Using Propositions \ref{Scott-Styer}, \ref{Scott} and \ref{pillai_atmost2_relevant}, we show a technical lemma.

\begin{lem}\label{technical}
Theorem \ref{atmost2} is true in each of the following cases$:$
\begin{itemize}
\item $a \equiv 1 \pmod{4}, \ b \equiv 1 \pmod{4};$
\item $\max\{a,b\}<11, \ c \equiv 0 \pmod{2}.$
\end{itemize}
\end{lem}

\begin{proof}
In the first case, we take equation \eqref{abc} modulo 4 to see that $c^z \equiv 2 \pmod{4}$, implying, $z=1$.
Proposition \ref{pillai_atmost2_relevant} completes the proof.

In the second case, suppose that equation \eqref{abc} has three solutions.
By Proposition \ref{pillai_atmost2_relevant}, there exists at least one solution with $z>1$ among them.
Since both $a,b$ are composed of only primes in $\{3,5,7\}$, we have $c=2$ by \cite[Theorem 7.2]{BeBi}.
Proposition \ref{Scott} shows the lemma.
\end{proof}

In order to prove our theorem, it suffices to consider the case where none of $a,b,c$ is a power, and $a,b,c$ are pairwise coprime.
Moreover, in view of Propositions \ref{Hu-Le}, \ref{Scott-Styer} and Lemma \ref{technical}, and since Theorem \ref{atmost2} holds true for $11 \le \max\{a,b,c\}<18$ by the combination of \cite{Ha} and \cite{Uc}, we may assume any of the conditions in $(\ast)$ below.
\begin{equation}\label{cond}
\tag{$\ast$}
\begin{cases}
\,\text{none of $a,b,c$ is a power, \ $a,b,c$ are pairwise coprime};\\
\,a \equiv -1 \pmod{4} \quad \text{or} \quad b \equiv -1 \pmod{4};\\
\,\max\{a,b\} \ge 11, \quad \,18 \le \max\{a,b,c\} \le 10^{62};\\
\,2 \mid c, \ c>2.
\end{cases}
\end{equation}
In particular, in the sequel, we always assume that $a,b,c$ are pairwise coprime, both $a,b$ are odd and $c$ is even.

%%%%%%%%%%%%%%%%%%%%
\section{General upper bound for $z$} \label{Sec-Uz}%
%%%%%%%%%%%%%%%%%%%%

Here we find a relatively sharp upper bound for $z$ in equation \eqref{abc}.
For this we prepare some lemmas.

The following is a slight improvement of a special case of \cite[Lemma 2.2]{PedW}.

\begin{lem} \label{pdw-lemma}
Let $v$ be a number with $v>{\rm e}^2/4.$
Assume that $\frac{t}{\log^2{t}}=v$ for some positive number $t$ with $t>{\rm e}^2.$
Then
\[
t<\left(1+\frac{\log\log{v_0}}{\log{v_0}-1}\right)^2v\log^2(4v),
\]
where $v_0$ is any number with ${\rm e}<v_0<2v^{1/2}.$
\end{lem}

\begin{proof}
Define $w$ and $Y$ as $w=\frac{2t^{1/2}}{\log t}$ and $(1+Y)w\log{w}=t^{1/2}$.
It is easy to see that $Y>0$ as $t>{\rm e}^2$. Observe that
\[
 (1+Y)w\log{w}=t^{1/2}=\frac{w}{2}\log{t}  =w\log \big( (1+Y)w\log{w} \big).
\]
Thus
\[
Y\log{w}=\log(1+Y)+\log\log{w}.
\]
Since $\log(1+Y)<Y$, we have $Y<\frac{\log\log{w}}{\log{w}-1}$, that is,
\[
t^{1/2}<\left(1+ \frac{\log\log{w}}{\log{w}-1} \right)w\log{w}.
\]
Since $w=2v^{1/2}>{\rm e}$ by assumption, and the function $\frac{\log\log{W}}{\log{W}-1}$ in $W$ is decreasing for $W>{\rm e}$, the above displayed inequality leads to the assertion.
\end{proof}

For a rational prime $p$ and non-zero integer $A$, as usual let $\nu_p(A)$ denote the $p$-adic valuation of $A$, that is, the exponent of $p$ in the prime factorization of $A$.

The next lemma is well-known and gives a precise information on the $2$-adic valuations of integers in certain forms.

\begin{lem} \label{2adic}
Let $s$ be an odd integer with $s \neq \pm 1.$
For any positive integer $n,$ the following hold.
\begin{alignat*}{2}
&\nu_2(s^n-1)=\nu_2(s^2-1)-1+\nu_2(n), & \quad & \text{if $s \equiv 1 \pmod{4}$ or $2 \mid n;$}\\
&\nu_2(s^n-1)=1, & & \text{if $s \equiv 3 \pmod{4}$ and $2 \nmid n;$}\\
&\nu_2(s^n+1)=\nu_2(s+1), && \text{if $2 \nmid n;$}\\
&\nu_2(s^n+1)=1, && \text{if $2 \mid n$}.
\end{alignat*}
\end{lem}

Define the function $\log_{\ast}$ as follows:
\[
\log_{\ast}(x):=\log \max\{x,{\rm e}\} \quad (x>0).
\]
Note that this is an increasing function.

The following proposition is a special case of \cite[Theorem 2.13]{bug-book}.

\begin{prop} \label{bu}
Let $u_1,u_2$ be coprime odd integers with $u_1,u_2 \neq \pm 1.$
Assume that a positive integer $g$ satisfies
\[
\nu_{2}({u_j}^g-1) \ge E \quad (j=1,2)
\]
for some number $E$ with $E>2.$
Let $H_1, H_2$ be real numbers satisfying
\[
H_j \ge \log \max\{|u_j|, 2^E\} \quad (j=1,2).
\]
Put
\[
\Lambda=u_1^{b_1}-u_2^{b_2} \quad (\neq 0),
\]
where $b_1, b_2$ are any positive integers.
If $\nu_2(u_2-1) \ge 2,$ then
\[
\nu_2(\Lambda) \le \frac{36.1g\mathcal B^2}{(\log 2)^4E^3}\,H_1H_2,
\]
where
\[
\mathcal B= \max \biggl\{ \log\Big( \frac{b_1}{H_2}+\frac{b_2}{H_1} \Big)+ \log(E\log 2)+0.4, \, 6E\log 2\biggl\}.
\]
\end{prop}

We give an application of Proposition \ref{bu} as follows.

\begin{lem} \label{2adic-bounds}
Assume that $\max\{a,b\} \ge 9.$ Put
\[
\alpha=\min\bigr\{ \nu_2(a^2-1) -1, \, \nu_2(b^2-1) -1 \bigr\}, \quad \beta=\nu_2(c).
\]
Let $(x,y,z)$ be a solution of equation \eqref{abc} with $z>1.$
Then
\[
z<\max \!\left\{ c_1,\, c_2 \log_{\ast}^2( c_3\log c ) \right\} \,(\log a)\,\log{b},
\]
where
\[
(c_1,c_2,c_3)=\begin{cases}
\, \left( \dfrac{1803.3m_2}{\beta}, \
\dfrac{23.865m_2}{\beta}, \
\dfrac{143.75(m_2+1)}{\beta} \right),
& \text{if $\alpha=2,$}\\
\, \left( \dfrac{2705m_3}{\alpha \beta}, \ \dfrac{156.39m_3\bigr(1+\frac{\log v_\alpha}{v_\alpha-1}\bigr)^2 }{\alpha^{3}\beta}, \ \dfrac{646.9(m_3+1)}{\alpha^{2}\beta} \right),
& \text{if $\alpha \ge 3$}
\end{cases}
\]
with $v_\alpha=3\alpha \log 2 - \log(3\alpha \log 2),$ and
\[
m_2=\begin{cases}
\, \frac{\log{8}}{\log \min\{a,b\}}, & \text{if $\min\{a,b\} \le 7,$}  \\
\, 1, & \text{if $\min\{a,b\}>7,$}
\end{cases}
\qquad m_3=\frac{\log{2^{\alpha}}}{\log (2^{\alpha}-1)}.
\]
\end{lem}

\begin{rem}\label{c1c2c3}\rm
$c_1,c_2,c_3$ are explicit constants depending only on $\alpha,\beta$,  and on $m_2$ only if $\alpha=2$.
Also, these numbers are decreasing on $\alpha \,(\ge 3)$ and on $\beta$.
\end{rem}

\begin{proof}[Proof of Lemma \ref{2adic-bounds}]
Since $c$ is even and $z>1$, it follows from equation \eqref{abc} that $a^x+b^y \equiv 0 \pmod{4}$.
Therefore, one of the following cases holds.
\begin{equation}\label{noneasy}
\begin{cases}
\ a \equiv 1, \ b \equiv -1 \pmod{4}, &2 \nmid y; \\
\ a \equiv -1, \ b \equiv 1 \pmod{4}, &2 \nmid x; \\
\ a \equiv b \equiv -1 \pmod{4}, &x \not\equiv y \pmod{2}.
\end{cases}
\end{equation}
Put $\Lambda =a^x+b^y$.
Since $\Lambda=c^z$, we have
\begin{equation}\label{z-low}
z=\frac{1}{\beta} \cdot \nu_2(\Lambda).
\end{equation}
In order to find an upper bound for $\nu_2(\Lambda)$, let us apply Proposition \ref{bu}.
For this, set $u_1,u_2,b_1,b_2$ as follows:
\[
(u_1,u_2;b_1,b_2)=
\begin{cases}
\, (a,-b;x,y), & \text{if $a \equiv 1, b \equiv -1 \mod{4},$} \\
\, (-a,b;x,y), & \text{if $a \equiv -1, b \equiv 1 \mod{4},$} \\
\, (-a,-b;x,y), & \text{if $a \equiv b \equiv -1 \mod{4}.$}
\end{cases}
\]
Then, by \eqref{noneasy},
\begin{gather*}
\pm \Lambda =u_1^{b_1}-u_2^{b_2},\\
u_1=\pm a, \ \ u_2=\pm b, \quad \min\bigr\{\nu_2 (u_1-1), \, \nu_2 (u_2-1) \bigr\} = \alpha.
\end{gather*}
For any positive integer $g$ and $i \in \{1,2\}$, observe from Lemma \ref{2adic} that
\[
\nu_2 ({u_i}^g-1)= \nu_2 ({u_i}^2-1)-1+\nu_2 (g)=\nu_2 (u_i-1)+\nu_2 (g) \ge \alpha +\nu_2(g).
\]
Thus we may set
\[
(g,E):= \begin{cases}
\, (2,3), & \text{if $\alpha=2,$} \\
\, (1,\alpha), & \text{if $\alpha \ge 3.$}
\end{cases}
\]
In what follows, let us separately consider the cases where $\alpha \ge 3$ and $\alpha=2$.
By symmetry of $a$ and $b$, we may assume that $a>b$ without loss of generality.

First, consider the case where $\alpha \ge 3$.
Observe that
\[
|u_1|=a>b=|u_2| \ge 2^\alpha-1.
\]
Thus we may set $H_1:=\log a$ and $H_2:=m_3\log{b}$.
Proposition \ref{bu} together with \eqref{z-low} gives
\begin{equation} \label{z-upp-alp3}
z \leq \frac{36.1m_3\mathcal B^2}{\,\beta (\log 2)^4 \alpha^3\,}\,(\log a)\log b,
\end{equation}
where
\[
\mathcal B= \max\!\left\{\log \Big( \frac{x}{m_3\log{b}}+\frac{y}{\log a}\Big) + \log( \alpha\log 2 )+0.4,\, 6\alpha\log 2  \right\}.
\]
Since $x<\frac{\log{c}}{\log a}z, \, y<\frac{ \log{c}}{\log{b}}z$, we have
\[\hspace{-2cm}
(6\alpha \log 2 \le) \ \ \mathcal B \le \log \max\left\{ z L', 2^{6\alpha}  \right\}
\]
with
\[
L'=\frac{\big(1+\frac{1}{m_3}\big)\alpha (\log 2){\rm e}^{0.4}\log c}{(\log a)\log b}.
\]
If $zL' \le 2^{6\alpha}$, then inequality \eqref{z-upp-alp3} yields
\[
z \le \frac{36.1m_3(6\alpha\log 2)^2}{\beta (\log 2)^4\,\alpha^3}\,(\log a)\log{b} <\frac{2705m_3}{\alpha \beta}\,(\log a)\log{b}.
\]
Suppose that $z L'>2^{6\alpha}$. Then
\[
z \le \frac{36.1m_3\log^2 (zL')}{\beta (\log 2 )^4\alpha^3}\,(\log a)\log{b},
\]
that is,
\begin{align*}
\frac{z L'}{\log^2(z L')}
&\le \frac{36.1 m_3L'}{ (\log 2 )^4\alpha^3\beta}\,(\log a)\log{b}
%\\&\le \frac{36.1 \cdot \frac{2\alpha {\rm e}^{0.4}(\log 2)\log{c}}{\log a\log{b}}}{ (\log 2 )^4\beta\alpha^3}(\log a)\log{b}\\&
=\frac{p(m_3+1)\log{c}}{\alpha^2\beta}
\end{align*}
with $p=\frac{36.1\cdot {\rm e}^{0.4}}{(\log 2 )^3}$.
Let us apply Lemma \ref{pdw-lemma} to the above inequality with
\[
v:=\frac{p(m_3+1)\log{c}}{\alpha^2\beta}, \quad v_0:=\frac{2^{3\alpha}}{3\alpha \log 2}, \quad t:=zL'.
\]
Indeed, as $zL'>2^{6\alpha}$,% \ge 2^{18}$,
\[
2v^{1/2} \ge 2 \sqrt{\frac{zL'}{\log^2(zL')}}>v_0.
%2 \sqrt{\frac{2^{18}}{\log^2(2^{18})}}>82.073.
\]
Putting $V=\bigl(1+\frac{\log\log{v_0}}{\log{v_0}-1}\bigl)^2$, we see that
\begin{align*}
z &< \frac{1}{L'}\,V \,v \log^2(4v)\\
%&<\frac{(\log a)\log b}{2\alpha {\rm e}^{0.4}(\log 2)\log c}\left(1+\frac{\log\log{v_0}}{\log{v_0}-1}\right)^2
%\,\frac{p\log{c}}{\beta\,\alpha^2} \log^2(4v)\\
&=\frac{V(\log a)\log b}{\big(1+\frac{1}{m_3}\big)\alpha (\log 2){\rm e}^{0.4}\log c} \,\frac{\frac{   36.1\cdot {\rm e}^{0.4}             (m_3+1)}{(\log 2 )^3}\log{c}}{\alpha^2\beta} \log^2(4v)\\
&=\frac{36.1V m_3}{(\log 2)^4\alpha^3\beta} \,\log^2(4v)\,(\log a)\log b\\
&< \frac{156.39V m_3}{\alpha^3\beta} \log_{\ast}^2 \biggl( \frac{646.9(m_3+1)\log{c}}{\alpha^2\beta}  \biggl) (\log a)\log{b}.
\end{align*}
To sum up, the assertion is shown in this case.

Next, consider the case where $\alpha=2$.
We proceed almost similarly to the previous case.
Observe that $2^E=8$, and $|u_1|=a=\max\{a,b\}>8$ by assumption.
Thus we may set $H_1:=\log a$ and $H_2:=m_2\log{b}$.
Proposition \ref{bu} gives
\begin{equation} \label{z-upp-alp2}
z \leq \frac{36.1gm_2\mathcal B^2}{\beta (\log 2 )^4(\alpha+1)^3}\,(\log a)\log{b},
\end{equation}
where
\[
\mathcal B= \max\!\left\{\log \Big( \frac{x}{m_2\log{b}}+\frac{y}{\log a}\Big) +  \log\big((\alpha+1)\log 2\big)+0.4,\, 6(\alpha+1)\log 2  \right\}.
\]
Observe that
\[\hspace{-2cm}
(6(\alpha+1) \log 2 \le) \ \ \mathcal B \le \log \max\left\{ zL', 2^{6(\alpha+1)}  \right\}
\]
with
\[
L'=\frac{\big(1+\frac{1}{m_2}\big)(\alpha+1)(\log 2){\rm e}^{0.4}\log c}{(\log a)\log b}.
\]
If $zL' \le 2^{6(\alpha+1)}$, then inequality \eqref{z-upp-alp2} yields
\[
z \le \frac{36.1gm_2\big(6(\alpha+1)\log 2\big)^2}{\beta (\log 2)^4(\alpha+1)^3}\,(\log a)\log{b}<\frac{1803.3\,m_2}{\beta}\,(\log a)\log{b}.
\]
Suppose that $zL'>2^{6(\alpha+1)}$. Then
\[
\frac{zL'}{\log^2(zL')} \le \frac{36.1gm_2 \cdot L'}{\beta(\alpha+1)^3 (\log 2)^4}\,(\log a)\log{b}=p(m_2+1)\frac{\log{c}}{(\alpha+1)^2\beta}
\]
with $p=\frac{36.1g{\rm e}^{0.4}}{(\log 2)^3}$.
Let us apply Lemma \ref{pdw-lemma} to the above inequality with
\[
v:=p(m_2+1)\frac{\log{c}}{(\alpha+1)^2\beta}, \quad v_0:=82.073, \quad t:=zL'.
\]
Indeed, as $zL'>2^{6(\alpha+1)}=2^{18}$,% \ge 2^{18}$,
\[
2v^{1/2} \ge 2 \sqrt{\frac{zL'}{\log^2(zL')}}>2 \sqrt{\frac{2^{18}}{\log^2(2^{18})}}>v_0.
\]
Putting $V=\bigl(1+\frac{\log\log{v_0}}{\log{v_0}-1}\bigl)^2$, we see that
\begin{align*}
z &< \frac{1}{L'}\,V \,v \log^2(4v)\\
&=\frac{V(\log a)\log b}{\big(1+\frac{1}{m_2}\big)(\alpha+1)(\log 2){\rm e}^{0.4}\log{c}}\,\frac{p(m_2+1)\log c}{(\alpha+1)^2\beta} \log^2(4v)\\
%&=\frac{pm_2}{(\log 2){\rm e}^{0.4}}
%\left(1+\frac{\log\log{v_0}}{\log{v_0}-1}\right)^2
%\,\log^2(4v)\,\frac{(\log a)\log{b}}{\beta(\alpha+1)^3} \\
&=\frac{pm_2V}{(\log 2){\rm e}^{0.4}(\alpha+1)^3\beta}\,\log_{\ast}^2\left(\frac{4p(m_2+1)\log{c}}{(\alpha+1)^2\beta} \right) \,(\log a)\log{b} \\
&=\frac{36.1gm_2V}{(\log 2)^4(\alpha+1)^3\beta}\,\log_{\ast}^2\left(\frac{\frac{4 \cdot 36.1g{\rm e}^{0.4}}{ (\log 2)^3}(m_2+1)\log{c}}{(\alpha+1)^2\beta} \right)\,(\log a)\log{b} \\
&< \frac{23.865\,m_2}{\beta}\log_{\ast}^2\left(\frac{143.75(m_2+1)\log c}{\beta} \right) (\log a)\log{b}.
\end{align*}
To sum up, the lemma is proved.
\end{proof}

%%%%%%%%%%%%%%%%%%%%%%%%%%%%%%%%%%%%%%
\section{2-adic investigation of $z$} \label{Sec-z-2adic}
%%%%%%%%%%%%%%%%%%%%%%%%%%%%%%%%%%%%%%

In this section, we show some results related to the 2-adic properties of $z$ in equation \eqref{abc}.
For this we prepare some notation as follows:
\begin{gather*}
\alpha_a=\nu_2(a^2-1) -1, \quad \alpha_b=\nu_2(b^2-1) -1,\\
\alpha=\min\{\alpha_a, \alpha_b \}, \quad \beta=\nu_2(c).
\end{gather*}
Note that both $a,b$ are congruent to $\pm1$ modulo $2^{\alpha}$ and $\alpha \ge 2$.

\begin{lem}\label{2adic_lower}
Let $(x,y,z)$ be a solution of equation \eqref{abc}.
Then either $(\beta,z)=(1,1)$ or $\beta z \ge \alpha.$
\end{lem}

\begin{proof}
We take equation \eqref{abc} modulo $2^{\alpha}$ to see that
\[
c^z \equiv \pm 2 \pmod{2^{\alpha}} \quad \text{or}  \quad c^z \equiv 0 \pmod{2^{\alpha}}.
\]
The first congruence implies that $2 \parallel c^z$ since $\alpha \ge 2$, and the second one means that $\nu_{2}(c^z) \ge \alpha$.
\end{proof}

\begin{lem}\label{2adic_upper}
Let $(x,y,z)=(X,Y,Z),(X',Y',Z')$ be two solutions of equation \eqref{abc}.
Then $XY' \ne X'Y,$ and
\[
\beta \cdot \min\{Z,Z'\} \le \alpha +\nu_2 (XY'-X'Y).
\]
\end{lem}

\begin{proof}
Let $\mathcal A \in \{a,b\}$.
Since $c$ is even, we use Lemma \ref{two} (see below) for $(A,B,C;\lambda)=(a,b,c;1)$ to see that
\begin{align}\label{mod}
\mathcal A^{\mathcal E} \equiv \delta \mod{2^{\beta \min\{Z,Z'\}}}
\end{align}
for some $\delta \in \{1,-1\}$, where $\mathcal E=|XY'-X'Y|$ is a positive integer.
Lemma \ref{2adic} gives
\[
\beta \min\{Z,Z'\} \le \nu_{2}(\mathcal A^{\mathcal E}-\delta)
\le \nu_{2}(\mathcal A^2-1)-1+\nu_{2}(\mathcal E).
\]
This shows the lemma as $\alpha=\min_{\mathcal A \in \{a,b\}}\nu_{2}(\mathcal A^2-1)-1$.
\end{proof}

\if0 \begin{lem} \label{2adic-cong}
Let $s$ be an odd integer with $s \neq \pm1$.
Define $\delta \in \{1,-1\}$ and a positive integer $\alpha_s$ by the following relations:
\[
s \equiv \delta \pmod{4}, \quad 2^{\alpha_s} \parallel (s-\delta).
\]
For any positive integer $n$,
\[
s^n \equiv \delta^n+2^{u} \mod{2^{u+1}},
\]
where $u=\alpha_s+\nu_2(n)$.
\end{lem}
\fi

Finally, we show a simple sufficient condition to give an exact information on $z$ in a certain case.

\begin{lem}\label{2adic_exact}
Let $(x,y,z)=(X,Y,Z),(X',Y',Z')$ be two solutions of equation \eqref{abc} with $\beta Z \ne 1$ and $\beta Z' \ne 1.$
If $X \not\equiv X' \pmod{2},$ then $Z=\alpha/\beta$ or $Z'=\alpha/\beta.$
\end{lem}

\begin{proof}
Without loss of generality, we may assume that $X$ is odd and $X'$ is even.
If $Y'$ is even, then $a^{X'}+b^{Y'}$ is a sum of two squares of odd integers, so $a^{X'}+b^{Y'} \equiv 2 \pmod{4}$, thereby $\beta Z'=\nu_2(c^{Z'})=\nu_2(a^{X'}+b^{Y'})=1$.
Thus, $Y'$ is odd by assumption.
To sum up,
\begin{equation}\label{parity}
2 \nmid X, \ \ 2\mid X', \ \ 2 \nmid Y'.
\end{equation}

Take $\delta_a,\delta_b \in \{1,-1\}$ such that $a \equiv \delta_a \pmod{4}$ and $b \equiv \delta_b \pmod{4}$.
By the definition of $\alpha_a$ and $\alpha_b$,
\[
2^{\alpha_a} \parallel (a-\delta_a), \quad 2^{\alpha_b} \parallel (b-\delta_b).
\]
Recall that $\min\{\alpha_a,\alpha_b\}=\alpha$.
Then, for any solution $(x,y,z)$ of equation \eqref{abc}, %by Lemma \ref{2adic-cong},
\[
a^{x} \equiv {\delta_a}^x+2^u \mod{2^{u+1}}, \quad b^{y} \equiv {\delta_b}^y+2^v \mod{2^{v+1}},
\]
where $u=\alpha_a+\nu_2(x)$ and $v=\alpha_b+\nu_2(y)$.
Replacing the modulus of each above congruences by $2^{\min\{u,v\}+1}$ and adding the resulting relations yields
\[
a^x+b^y \equiv \delta + 2^u+2^v \mod{2^{\min\{u,v\}+1}}
\]
with $\delta:={\delta_a}^x+{\delta_b}^y \in \{-2,0,2\}$.
This congruence implies that
\[
a^x+b^y \equiv
\begin{cases}
\, 2 \pmod{4}, &\text{if $\delta=\pm2$},\\
\, 2^{\min\{u,v\}} \pmod{2^{\min\{u,v\}+1}}, &\text{if $\delta=0, u \ne v$},\\
\, 0 \pmod{2^{\min\{u,v\}+1}}, &\text{if $\delta=0, u=v$}.
\end{cases}
\]
Since $\nu_2(a^x+b^y)=\nu_2(c^z)=\beta z$, we have
\[ \beta z= \begin{cases}
\, 1, &\text{if $\delta=\pm2$},\\
\, \min\{u,v\}, &\text{if $\delta=0, u \ne v$}.
\end{cases} \]

We apply the previous argument with $(x,y,z)=(X,Y,Z),(X',Y',Z')$, together with the consideration to \eqref{parity}, to show that
\[\begin{cases}
\, \beta Z= \min\{U,V\}, &\text{if $U \ne V$},\\
\, \beta Z'= \min\{U',V'\}, &\text{if $U' \ne V'$},
\end{cases}\]
where
\[
U=\alpha_a, \ \ V=\alpha_b+\nu_2(Y), \ \ U'=\alpha_a+\nu_2(X'), \ \ V'=\alpha_b.
\]
Observe that $U-V<U'-V'$.
To sum up, these relations together imply that
\[\begin{cases}
\, \alpha_a-\alpha_b \ge \nu_2(Y), \ \beta Z'=V'=\alpha_b, &\text{if $U \ge V$},\\
\, \alpha_b-\alpha_a \ge \nu_2(X'), \ \beta Z=U=\alpha_a, &\text{if $U' \le V'$},\\
\, \beta Z=U=\alpha_a, \ \beta Z'=V'=\alpha_b, &\text{if $U<V$ and $U'>V'$}.
\end{cases}\]
This shows the lemma.
\end{proof}

%%%%%%%%%%%%%%%%%%%%%%%%%%%%%%%%
\section{Preliminaries for Theorem \ref{atmost2}} \label{prelim}%
%%%%%%%%%%%%%%%%%%%%%%%%%%%%%%%%

From now on, let $(a,b,c)$ be any fixed triple of positive integers satisfying \eqref{cond}.
Positive integers $\alpha,\beta$ are defined as in the previous section.
From \eqref{cond},
\begin{equation} \label{la,lb,lc}
\max\{a,b\} \ge \max\{11,2^\alpha+1\}, \ \ \min\{a,b\} \ge 2^\alpha-1, \ \ c \ge 3 \cdot 2^\beta.
\end{equation}
Also, we suppose that equation \eqref{abc} has three solutions, say $(x,y,z)=(x_t,y_t,z_t)$ with $t \in \{1,2,3\}$, that is,
\[
a^{x_1}+b^{y_1}=c^{z_1},\ \ a^{x_2}+b^{y_2}=c^{z_2}, \ \ a^{x_3}+b^{y_3}=c^{z_3}.
\]
Without loss of generality, we may assume that
\begin{equation} \label{zorders}
z_1 \le z_2 \le z_3.
\end{equation}
For each $t \in \{1,2,3\}$, we often refer to the equation $a^{x_t}+b^{y_t}=c^{z_t}$ as `$t$-th equation', and also the pair of 1st and 2nd equations as `the first two equations'.

It is obvious that
\begin{equation} \label{basic}
x_t<\frac{\log{c}}{\log a}\,z_t, \quad y_t<\frac{\log{c}}{\log b}\,z_t \quad (t=1,2,3).
\end{equation}
Lemmas \ref{2adic_lower}, \ref{2adic_upper} and \ref{2adic-bounds} tell us that
\begin{gather}
(\beta,z_t)=(1,1) \ \ \text{or} \ \ z_t \ge \frac{\alpha}{\beta} \quad (t=1,2);\label{z1}\\
\beta z_t\le \alpha +\nu_2(x_ty_{t+1}-x_{t+1}y_t) \quad (t=1,2);\label{z2}\\
z_3<\mathcal H_{\alpha,\beta,m_2}(c;a,b), \label{z3}
\end{gather}
respectively, where
\[
\mathcal H_{\alpha,\beta,m_2}(u;v,w):=\max \bigr\{ c_1,\, c_2 \log_{\ast}^2(c_3\log u) \big\} \cdot \log v \cdot \log w.
\]
From Remark \ref{c1c2c3}, %$c_1,c_2,c_3$ are explicit constants depending only on $\alpha,\beta$, and on $m_2$ only when $\alpha=2$.
%Thus,
the lower index $m_2$ of $\mathcal H$ makes sense only when $\alpha=2$.
In what follows, we simply write $\mathcal H_{\alpha,\beta,m_2}(u;v,w)=\mathcal H_{\alpha,\beta,m_2}(u)$ when $u=v=w$, also, $\mathcal H_{\alpha,\beta,m_2}(u;v):=\mathcal H_{\alpha,\beta,m_2}(u;v,w)/\log{w}$.
Also, all indices $\alpha,\beta,m_2$ are often omitted.

Under this setting, we show some results as the first preliminaries in the proof of Theorem \ref{atmost2}.

\begin{lem}\label{z1z2_upper}
The following inequalities hold.
\begin{align*}
\tag{i}
&\beta z_1-\frac{\log (z_1z_2)}{\log 2} < \alpha +\frac{1}{\log 2}\log\left(\frac{\log^2 c}{(\log a) \log b} \right)-\frac{\log g}{\log 2};\\
&\tag{ii}
\beta z_2-\frac{\log z_2}{\log 2} < \alpha +\frac{\log \mathcal H_{\alpha,\beta,m_2}(c)}{\log 2}-\frac{\log g}{\log 2},
\end{align*}
where $g$ is the greatest odd divisor of $\gcd(x_2,y_2).$
\end{lem}

\begin{proof}
Let $t \in \{1,2\}$. From \eqref{basic},
\[
|x_ty_{t+1}-x_{t+1}y_t|<\max\{x_ty_{t+1},x_{t+1}y_t\}<\frac{\log^2 c}{(\log a) \log b}\,z_t z_{t+1}.
\]
In particular, by \eqref{z3},
\begin{equation}\label{x2y3x3y2}
|x_2y_3-x_3y_2|<\frac{\log^2 c}{(\log a) \log b}\,z_2 z_3<z_2 \cdot \mathcal H_{\alpha,\beta,m_2}(c).
\end{equation}
Also, by the definition of $g$,
\[
\nu_2( x_t y_{t+1}-x_{t+1}y_t )=\nu_2 \left( \frac{x_t y_{t+1}-x_{t+1}y_t}{g} \right).
\]
These relations together with \eqref{z2} yield
\begin{align*}
\beta z_t \le \alpha + \nu_2\left( x_t y_{t+1}-x_{t+1}y_t\right)
& \le \alpha +\frac{\log \bigr(| x_t y_{t+1}-x_{t+1}y_t |/ g\bigr)}{\log 2}\\
&< \alpha +\frac{1}{\log 2}\log\left(\frac{\log^2 c}{g(\log a) \log b}\,z_{i} z_{t+1} \right).
\end{align*}
Thus
\[
\beta z_t-\frac{\log z_t}{\log 2} <\alpha +\frac{1}{\log 2}\log\left(\frac{\log^2 c}{g(\log a) \log b}\,z_{t+1} \right).
\]
This inequality for $t=1$ immediately yields (i).
(ii) follows from the above inequality for $t=2$ together with \eqref{z3}.
\end{proof}

\begin{definition}\rm
For given $\alpha,\beta,m_2,a,b,c$ and $g$, define $\mathcal U_2=\mathcal U_2(\alpha,\beta,m_2,c,g)$ as the largest $z_2$ among $z_2$ satisfying inequality (ii) of Lemma \ref{z1z2_upper}, and define $\mathcal U_1=\mathcal U_1(\alpha,\beta,m_2,a,b,c,g)$ as the largest $z_1$ among $z_1$ satisfying inequality (i) of Lemma \ref{z1z2_upper} with $z_2$ replaced by $\mathcal U_2$.
\end{definition}

As remarked before, $m_2$ affects the definitions of $\mathcal U_1,\mathcal U_2$ only when $\alpha=2$.
Note that both $\mathcal U_1,\mathcal U_2$ are decreasing on $\alpha\,(\ge 3),\beta$ and $g$, and are increasing on $c$, and on $m_2$ if $\alpha=2$.

\begin{lem}\label{firstbounds}
$z_2 \le 230, \ \max\{x_1,y_1,x_2,y_2\}<4300.$
\end{lem}

\begin{proof}
On inequality (ii) of Lemma \ref{z1z2_upper}, put $(\beta,c,m_2)=(1,10^{62},\frac{\log 8}{\log 3})$, thereby
\[
z_2-\frac{\log z_2}{\log 2}< \alpha + \frac{\log \mathcal H_{\alpha,1,\log 8/\log 3}(10^{62})}{\log 2}
\]
For each $\alpha$ with $2 \le \alpha \le \frac{\log (\min\{a,b\}+1)}{\log 2} \le \frac{\log 10^{62}}{\log 2} \ (<206)$, the above inequality yields that $\mathcal U_2(\alpha,1,\frac{\log 8}{\log 3}, 10^{62},1) \le 230$.
The second asserted inequality follows from the inequality $\min\{a,b\}^{\max\{x_1,y_1,x_2,y_2\}}<c^{\,\mathcal U_2}$ with $\min\{a,b\} \ge 2^\alpha-1$.
\end{proof}

In what follows, let $(\dagger)$ denote the following case:
\[\tag{$\dagger$}
\begin{cases}
\,\text{$x_I \not\equiv x_J \pmod{2}$ \ or \ $y_I \not\equiv y_J \pmod{2}$}\\
\,\text{for some pair $\{I,J\} \subset \{1,2,3\}$.}
\end{cases}
\]

\begin{lem}\label{dagger}
In case $(\dagger),$ either $(\beta,z_1) =(1,1)$ or $z_1=\alpha/\beta.$
\end{lem}

\begin{proof}
Suppose that $(\beta,z_1) \ne (1,1)$.
Since $z_1 \le z_2 \le z_3$, we have $(\beta,z_t) \ne (1,1)$ for any $t$.
Then Lemma \ref{2adic_exact} tells that $z_t=\alpha/\beta$ for some $t$.
In particular, $z_1 \le z_t \le \alpha/\beta$.
On the other hand, $z_1 \ge \alpha/\beta$ by \eqref{z1}.
These inequalities together show that $z_1=\alpha/\beta$.
\end{proof}

%%%%%%%%%%%%%%%%%%%%%%
\section{Improving work of Hu and Le}\label{Sec-improve-HuLe}%
%%%%%%%%%%%%%%%%%%%%%%

In this section, we improve the gap principle of Hu and Le.
To follow their strategy, we start with several lemmas derived as consequences of \cite{HuLe_deb19}.

\begin{definition}
Let $r$ and $s$ be coprime integers with $s>2.$
Define $n(r,s)$ be the least positive integer among positive integer $n$'s for which $r^n$ is congruent to $\pm 1$ modulo $s.$
Moreover, define $\delta=\delta(r,s) \in \{1,-1\}$ and a positive integer $f=f(r,s)$ as follows{\rm :}
\[
\delta \equiv r^{n(r,s)} \pmod{s}, \quad f=\frac{r^{n(r,s)}-\delta}{s}.
\]
\end{definition}

In the following lemma, the first statement is elementary, and the second one easily follows from the proof of \cite[Lemma 4.4]{HuLe_aa15}.

\begin{lem}\label{ele1}
Let $r$ and $s$ be coprime integers with $s>2.$
\begin{itemize}
\item[\rm (i)]
Let $n'$ be any positive integer satisfying
\[
r^{n'} \equiv \delta' \mod{s}
\]
for some $\delta' \in \{1,-1\}.$ Then
\begin{gather*}
n' \equiv 0 \mod{n(r,s)},\\ r^{n'}-\delta' \equiv 0 \mod{(r^{n(r,s)}-\delta(r,s))}.
\end{gather*}
\item[\rm (ii)]
Let $t$ be any positive integer whose prime factors divide $s.$
Assume that $s \not\equiv 2 \pmod{4}.$
Let $n'$ be any positive integer satisfying
\[
r^{n'} \equiv \delta' \mod{st}
\]
for some $\delta' \in \{1,-1\}.$ Then
\[
n' \equiv 0 \mod{\frac{t \cdot n(r,s)}{\gcd \bigr(t,f(r,s)\bigr)}}.
\]
\end{itemize}
\end{lem}

Let $A,B$ and $C$ be any fixed pairwise coprime integers greater than 1.
For each $\lambda \in \{1,-1\}$, consider the following equation:
\begin{equation}\label{ABC}% \tag{ABC}
A^X+\lambda B^Y=C^Z,
\end{equation}
where $X,Y,Z$ are unknown positive integers.
For our purpose, it suffices to observe equation \eqref{ABC} under the following conditions (corresponding to \eqref{cond}):
\begin{equation}\label{COND} \tag{$\ast \ast$}
\begin{cases}
\,\text{none of $A,B,C$ is a power};\\
\,\text{$2 \mid C, \,C>2, \, \max\{A,B\} \ge 11$}, & \text{if $\lambda=1$};\\
\,\text{$2 \mid A, \,A>2, \, \max\{B,C\} \ge 11$}, & \text{if $\lambda=-1$}.
\end{cases}
\end{equation}

The following lemma easily follows from the proof of \cite[Lemma 3.3]{HuLe_aa15}.
It is worth noting that its first assertion is based upon primitive divisor theorem.

\begin{lem}\label{two}
Let $(X,Y,Z)$ and $(X',Y',Z')$ be two solutions of equation \eqref{ABC}.
Assume that $C^{\,\min\{Z,Z'\}}>2.$
Then $XY' \ne X'Y.$
Moreover, for each $\mathcal A \in \{A,B\},$
\[
\mathcal A^{\,|XY'-X'Y|} \equiv \pm 1 \mod{C^{\,\min\{Z,Z'\}}}.
\]
\end{lem}

The following lemma easily follows from the proofs of \cite[Lemma 4.6]{HuLe_jnt18} and \cite[Lemma 4.4]{HuLe_deb19}.

\begin{lem}\label{Y2}
Let $(X,Y,Z)$ and $(X',Y',Z')$ be two solutions of equation \eqref{ABC} such that $Z<Z'.$
Then the following hold.
\begin{itemize}
\item[\rm (i)]
If $C^{Z}>2,$ then
\[
\gcd\bigr(C,f(B,C^Z)\bigr) \mid X', \ \ \gcd\bigr(C,f(A,C^Z)\bigr) \mid Y'.
\]
\item[\rm (ii)]
If $C^{Z} \not\equiv 2 \pmod{4},$ then
\[
\gcd \bigr( C^{Z'-Z},\, f(B,C^Z)\bigr) \mid X', \ \ \gcd \bigr(C^{Z'-Z},\,f(A,C^Z)\bigr) \mid Y'.
\]
\end{itemize}
\end{lem}

Before going to state the improved gap principle, we show an elementary lemma.

\begin{lem}\label{g2_upper_a}
Let $(X,Y,Z)$ be a solution of equation \eqref{ABC} with $\lambda=-1.$
Put $G=\gcd(X,Y).$ If $G>1,$ then
\[
X<\frac{G}{G-1}\frac{\log C}{\log A}\,Z, \quad Y<\frac{G}{G-1}\frac{\log C}{\log B}\,Z.
\]
\end{lem}

\begin{proof}
By the definition of $G$,
\[
C^{Z}=(A^{X/G})^{G}-(B^{Y/G})^{G}=(A^{X/G}-B^{Y/G}) \bigr((A^{X/G})^{G-1}+ \cdots + (B^{Y/G})^{G-1}\bigr)
\]
with $A^{X/G}-B^{Y/G}>0$. In particular,
\[
A^{\frac{G-1}{G}X}<C^{Z}, \quad B^{\frac{G-1}{G}Y}<C^{Z}.
\]
These give the asserted inequalities.
\end{proof}

For any positive numbers $P$ and $Q$, we define $t_{P,Q}$ as follows:
\[
t_{P,Q}:=\frac{\log \min\{P,Q\}}{\log \max\{P,Q\}}.
\]

Now we state our improved gap principle.

\begin{prop}\label{improvedgap}
Suppose that equation \eqref{ABC} has three solutions $(X,Y,Z)=(X_r,Y_r,Z_r)$ with $r \in \{1,2,3\}$ such that $Z_1<Z_2 \le Z_3.$
Put $G_2=\gcd(X_2,Y_2),$ and
\[
\chi :=\begin{cases}
\ 2, & \text{if $Z_1>1$, and either $\lambda=1$ or $C>\max\{A,B\}$},\\
\ 1, & \text{otherwise}.
\end{cases}
\]
\begin{itemize}
\item[\rm (I)]
Suppose that $C^{Z_1} \equiv 2 \pmod{4}$ with $C^{Z_1}>2.$ Then
\begin{align*}
& C \mid G_2 \cdot (X_2Y_3-X_3Y_2);\\
& C \le \frac{\min\{ X_2 \log B, Y_2 \log A \}}{\log (\chi C^{Z_1}-1)} \cdot  |X_2Y_3-X_3Y_2|.
\end{align*}
Moreover, if either $\lambda=1,$ or $\lambda=-1$ with $G_2>1,$ then
\[
C<\mathcal K \cdot t_{A,B} \cdot \frac{Z_2}{Z_1}\cdot |X_2Y_3-X_3Y_2|,
\]
where
\begin{gather*}
\mathcal K=\begin{cases}
\, \frac{Z_1\log C}{\log (\chi C^{Z_1}-1)}, &\text{if $\lambda=1$},\\
\, \frac{Z_1\log C}{\log (\chi C^{Z_1}-1)}\cdot \frac{G_2}{G_2-1}, &\text{if $\lambda=-1,G_2>1$}.
\end{cases}
\end{gather*}
\item[\rm (II)]
Suppose that $C^{Z_1} \not\equiv 2 \pmod{4}.$
Then
\begin{align*}
& C^{Z_2-Z_{1}} \mid G_2 \cdot (X_2Y_3-X_3Y_2);\\
& C^{Z_2-Z_{1}} \le \frac{\min\{ X_2 \log B, Y_2 \log A \}}{\log (\chi C^{Z_1}-1)} \cdot  |X_2Y_3-X_3Y_2|.
\end{align*}
Moreover, if either $\lambda=1,$ or $\lambda=-1$ with $G_2>1,$ then
\[
C^{Z_2-Z_{1}}<\mathcal K \cdot t_{A,B} \cdot \frac{Z_2}{Z_1}\cdot |X_2Y_3-X_3Y_2|,
\]
where $\mathcal K$ is the same as in (I).
\end{itemize}
\end{prop}

\begin{proof}
We fix the value of $\mathcal A \in \{A,B\}$. % for a space.
Applying Lemma \ref{two} for $(X,Y,Z)=(X_2,Y_2,Z_2)$ and $(X',Y',Z')=(X_3,Y_3,Z_3)$ shows that
\begin{align}\label{cong_||}
\mathcal A^{\,|X_2Y_3-X_3Y_2|} \equiv \varepsilon \mod{C^{Z_2}}
\end{align}
with $X_2Y_3-X_3Y_2 \neq 0$ and some $\varepsilon \in \{1,-1\}$.
Since $Z_2 \ge Z_1$, Lemma \ref{ele1}\,(i) for $(r,s)=(\mathcal A,C^{Z_1})$ together with congruence \eqref{cong_||} tells us that $|X_2Y_3-X_3Y_2|$ is divisible by $n(\mathcal A,C^{Z_{1}})$.
Put
\[
n_{1}:=n(\mathcal A,C^{Z_{1}}).
\]
Then
\begin{equation}\label{n2}
|X_2Y_3-X_3Y_2|=n_1n_2
\end{equation}
for some positive integer $n_2$. On the other hand,
\[
\mathcal A^{n_1}=C^{Z_{1}}f+\delta,
\]
where $f=f(\mathcal A,C^{Z_1})$ and $\delta=\delta(\mathcal A,C^{Z_1})$.

Suppose that $f=1$. Then $\mathcal A^{n_1}=C^{Z_1} \pm 1$. Observe that
\[
\mathcal A^{n_1} \ge C^{Z_1}-1\, \begin{cases}
\,=A^{X_1}+B^{Y_1}-1>\mathcal A, & \text{if $\lambda=1$},\\
\,\ge C^2-1>C>\mathcal A, & \text{if $C>\mathcal A$ and $Z_1>1$}.
\end{cases}
\]
Thus, if $Z_1>1$, and either $\lambda=1$ or $C>\mathcal A$, then $\min\{n_1,Z_1\}>1$, and so the well-known theorem of \cite{Mi} on Catalan's equation tells us that $(\mathcal A^{n_1},C^{Z_1})=(8,9)$, which contradicts \eqref{COND}.
By these observations,
\begin{equation}\label{lowbound_n1}
n_{1}=n(\mathcal A,C^{Z_{1}}) \ge \frac{\log (\chi C^{Z_1}-1)}{\log \mathcal A}.
\end{equation}

(I) From \eqref{cong_||} and \eqref{n2},
\[
\mathcal A^{\,n_1n_2}=C^{Z_2}h+\varepsilon,
\]
for some positive integer $h$.
We substitute the mentioned expression of $\mathcal A^{\,n_1}$ into the above to see that
\begin{align*}
C^{Z_2}h+\varepsilon=(C^{Z_{1}}f+\delta)^{n_2}=C^{Z_{1}}\sum_{i=1}^{n_2} \binom{n_2}{i}(C^{Z_{1}})^{i-1}f^i\,\delta^{n_2-i}+\delta^{n_2}.
\end{align*}
It is easy to see that $\varepsilon=\delta^{n_2}$ as $\varepsilon \equiv \delta^{n_2} \pmod{C^{Z_1}}$ with $\delta, \varepsilon \in \{1,-1\}$ and $C^{Z_1}>2$ by assumption.
Therefore,
\begin{align*}
C^{Z_2-Z_{1}}h &= \sum_{i=1}^{n_2} \binom{n_2}{i}(C^{Z_{1}})^{i-1}f^i\, \delta^{n_2-i}\\
&=n_2f\, \delta^{n_2-1}+C^{Z_{1}}\sum_{i=2}^{n_2} \binom{n_2}{i}(C^{Z_{1}})^{i-2}f^i\, \delta^{n_2-i}.
\end{align*}
Since $Z_2>Z_1$ by assumption, we have $n_2f \equiv 0 \pmod{C}$, and so
\[
n_2\gcd (f,C) \equiv 0 \mod{C}.
\]
By \eqref{n2},
\begin{equation}\label{n1C_div}
n_1 C \mid \gcd (f,C) \cdot |X_2Y_3-X_3Y_2|.
\end{equation}
On the other hand, Lemma \ref{Y2}\,(i) for $(X,Y,Z)=(X_1,Y_1,Z_1),(X',Y',Z')=(X_2,Y_2,Z_2)$ and $n=n_1$ tells us that
\[
\begin{cases}
\, \gcd (f,C) \mid Y_2, & \text{if $\mathcal A=A$},\\
\, \gcd (f,C) \mid X_2, & \text{if $\mathcal A=B$}.
\end{cases}
\]
This together with \eqref{n1C_div} gives that
\begin{equation}\label{general}
\begin{cases}
\, n(A,C^{Z_{1}}) \cdot C \,\mid \, Y_2 \cdot |X_2Y_3-X_3Y_2|, & \text{if $\mathcal A=A$},\\
\, n(B,C^{Z_{1}}) \cdot C \, \mid \, X_2 \cdot |X_2Y_3-X_3Y_2|, & \text{if $\mathcal A=B$}.
\end{cases}
\end{equation}
This enables us to deduce the first asserted divisibility relation.

Moreover, from \eqref{lowbound_n1}, inequalities \eqref{general} imply
\begin{equation}\label{chi}
\begin{cases}
\, \dfrac{\log (\chi C^{Z_1}-1)}{\log A}\,C \le Y_2 \cdot |X_2Y_3-X_3Y_2|,\\
\, \dfrac{\log (\chi C^{Z_1}-1)}{\log B}\,C \le X_2 \cdot |X_2Y_3-X_3Y_2|.
\end{cases}
\end{equation}
This implies the first asserted upper bound for $C$.

Suppose that $\lambda=1$.
Since $X_2<\frac{\log C}{\log A}Z_2$ and $Y_2<\frac{\log C}{\log B}Z_2$, we use \eqref{chi} to see that
\begin{align*}
C<\frac{\log A}{\log (\chi C^{Z_1}-1)} \cdot \frac{\log C}{\log B}\,Z_2\cdot |X_2Y_3-X_3Y_2|,\\
C<\frac{\log B}{\log (\chi C^{Z_1}-1)} \cdot \frac{\log C}{\log A}\,Z_2\cdot |X_2Y_3-X_3Y_2|,
\end{align*}
thereby
\[
C<\frac{Z_1\log C}{\log (\chi C^{Z_1}-1)}\cdot  t_{A,B} \cdot \frac{Z_2}{Z_1}\cdot |X_2Y_3-X_3Y_2|.
\]

Suppose that $\lambda=-1$ and $G_{2}>1$.
Applying Lemma \ref{g2_upper_a} for $(X,Y,Z)=(X_2,Y_2,Z_2)$ gives
\[
X_2<\frac{G_2}{G_2-1}\frac{\log C}{\log A}\,Z_2, \quad Y_2<\frac{G_2}{G_2-1}\frac{\log C}{\log B}\,Z_2.
\]
These inequalities together with \eqref{chi} gives the remaining assertion.

(II) Apply Lemma \ref{ele1}\,(ii) for
\[
(r,s,t,n')=(\mathcal A,C^{Z_1},C^{Z_2-Z_1},|X_2Y_3-X_3Y_2|),
\]
together with congruence \eqref{cong_||}. Then
\[
n_1C^{Z_2-Z_1} \,\mid \, \gcd (C^{Z_2-Z_1},f) \cdot |X_2Y_3-X_3Y_2|.
\]
Using this divisibility relation together with Lemma \ref{Y2}\,(ii) for $(X,Y,Z)=(X_1,Y_1,Z_1),(X',Y',Z')=(X_2,Y_2,Z_2)$ and $n=n_1$, we can show the assertions almost similarly to case (I).
\end{proof}

In the remaining parts of this section, we apply Proposition \ref{improvedgap} to a special case concerning equation \eqref{ABC}.
For this we prepare two lemmas from the works of Hu and Le.

The following lemma directly follows from the proofs of \cite[Lemmas 3.2 and 3.4]{HuLe_deb19}.

\begin{lem}\label{cf-1}
Let $(X,Y,Z)$ be a solution of equation \eqref{ABC} for $\lambda=1.$
Then the following hold.
\begin{itemize}
\item[\rm (i)]
If $A^{2X}<C^{Z},$ then
\[
0<\frac{\log C}{\log B}-\frac{Y}{Z}<\frac{2}{ZC^{Z/2}\log B}.
\]
\item[\rm (ii)]
Let $(X',Y',Z')$ be another solution of equation \eqref{ABC} for $\lambda=1$.
If $X>X'$ and $Z \le  Z',$ then
\[
0<\frac{\log C}{\log B}-\frac{Y'}{Z'}<\frac{2}{Z'A^{X-X'}C^{Z'-Z}\log B}.
\]
\end{itemize}
\end{lem}

The following lemma directly follows from the argument in \cite[Section 5]{HuLe_deb19}.

\begin{lem}\label{cf-2}
Let $(X,Y,Z)$ and $(X',Y',Z')$ be two solutions of equation \eqref{ABC} for $\lambda=1.$
Assume that both $Y/Z$ and $Y'/Z'$ are convergents to $\frac{\log C}{\log B}.$
If $Y/Z<Y'/Z'<\frac{\log C}{\log B},$ then
\[
Z'>\frac{1}{Z \left(  \frac{\log C}{\log B}-\frac{Y}{Z} \right)}.
\]
\end{lem}

\begin{prop}\label{=<lam1}
Suppose that equation \eqref{ABC} for $\lambda=1$ has three solutions
$(X,Y,Z)=(X_r,Y_r,Z_r)$ with $r \in \{1,2,3\}$ such that $Z_1=Z_2<Z_3.$
Then one of the following inequalities holds.
\begin{gather*}
C^{Z_2/2}<\frac{2}{\log \min\{A,B\}}\,Z_3,\\
C^{Z_2/2}/Z_2 < \max_{t \in \{1,2\}}\big\{ |X_3Z_2-X_tZ_3|, \ |Y_3Z_2-Y_tZ_3| \big\}.
\end{gather*}
\end{prop}

\begin{proof}
Since $Z_1=Z_2$,
\begin{equation}\label{Z1=Z2}
A^{X_1}+B^{Y_1}=A^{X_2}+B^{Y_2}=C^{Z_2}.
\end{equation}
Note that $X_1 \ne X_2, Y_1 \ne Y_2$, and that $X_1<X_2$ if and only if $Y_2<Y_1$.
By symmetry of indices 1 and 2 in the assertion, we may assume that
\begin{equation}\label{X1<X2}
X_1<X_2, \quad Y_2<Y_1.
\end{equation}
Also, from the equations in \eqref{Z1=Z2}, observe that
\[
B^{Y_2} \mid (A^{X_2-X_1}-1), \ \
A^{X_1} \mid (B^{Y_1-Y_2}-1), \ \
A^{X_2-X_1} \cdot B^{Y_1-Y_2} \equiv 1 \mod{C^{Z_2}}.
\]
These imply that
\begin{align}
&A^{X_2-X_1}>B^{Y_2}, \quad B^{Y_1-Y_2}>A^{X_1},\label{gap_eq}\\
&\max\{A^{X_2-X_1},\,B^{Y_1-Y_2}\}> C^{Z_2/2}.\label{gap}
\end{align}
By symmetry of $A$ and $B$ in the assertion, we may assume that $A^{X_2-X_1}>B^{Y_1-Y_2}$.
From \eqref{gap},
\begin{equation}\label{gapA}
A^{X_2-X_1}> C^{Z_2/2}.
\end{equation}

Suppose that $X_2 \le  X_3$.
From \eqref{X1<X2} observe that the equation $C^{X'}-B^{Y'}=A^{Z'}$ has three solutions $(X',Y',Z')=(Z_r,Y_r,X_r)$ with $r \in \{1,2,3\}$ satisfying $X_1<X_2 \le X_3$.
Since $A^{X_1}$ is odd, Proposition \ref{improvedgap}\,(ii) yields that
\[
\gcd (Y_2,Z_2)\,|Z_2Y_3-Z_3Y_2| \ge A^{X_2-X_1}.
\]
It follows from \eqref{gapA} that
\begin{equation}\label{gap_Z1=Z2_no1}
\gcd (Y_2,Z_2)\,|Z_2Y_3-Z_3Y_2|> C^{Z_2/2}.
\end{equation}

Suppose that $X_3<X_2$.
From \eqref{Z1=Z2} and \eqref{gap_eq} observe that
\[
C^{Z_1}=C^{Z_2}>A^{X_2}=A^{X_1}A^{X_2-X_1}>A^{X_1}B^{Y_1-Y_2}>A^{2X_1}.
\]
Lemma \ref{cf-1}\,(i) for $(X,Y,Z)=(X_1,Y_1,Z_1)$ tells us that
\[
\hspace{-2cm}(0<) \quad \frac{\log C}{\log B}-\frac{Y_1}{Z_1}<\frac{2}{Z_1C^{Z_1/2}\log B}.
\]
If the RHS above is greater than $1/(2Z_1^2)$, then $\frac{C^{Z_1/2}}{Z_1}<\frac{4}{\log B}$.
It is easy to see that this leads to a contradiction to \eqref{COND}.
Thus,
\begin{equation}\label{cf_first}
\hspace{-1cm}\frac{\log C}{\log B}-\frac{Y_1}{Z_1} < \frac{2}{Z_1C^{Z_1/2}\log B} \le \frac{1}{2Z_1^2}.
\end{equation}

On the other hand, since $X_3<X_2$ and $Z_2<Z_3$, Lemma \ref{cf-1}\,(ii) for $(X,Y,Z)=(X_2,Y_2,Z_2), (X',Y',Z')=(X_3,Y_3,Z_3)$ gives
\[
(0<) \quad \frac{\log C}{\log B}-\frac{Y_3}{Z_3}<\frac{2}{Z_3A^{X_2-X_3}C^{Z_3-Z_2}\log B}.
\]
Suppose that the RHS above is greater than $1/(2Z_3^2)$. Then
\[
4Z_3>(\log B)\,A^{X_2-X_3}C^{Z_3-Z_2} \ge (A\log B)\,C^{Z_3-Z_2} \ge (A\log B) \cdot C.
\]
Put $\alpha:=\min\{ \nu_{2}(A^2-1),\, \nu_{2}(B^2-1)\}-1$.
Since $\max\{A,B\} \ge \max \{11,2^\alpha +1\}$ by \eqref{COND}, and $\min\{A,B\} \ge 2^\alpha-1$, we have $A\log B \ge c(\alpha) \ge 3 \log 11$, where $c(\alpha):=(2^\alpha -1)\log \max \{11,2^\alpha +1\}$.
Then $Z_3>\frac{1}{4}(A\log B)C >10$, and so
\[
Z_2> Z_3-\frac{\log (4Z_3)}{\log C}>\frac{3}{4}\,Z_3.
\]
This gives rise to a sharp lower bound for $Z_2$, that is, $Z_2 \ge \left\lceil \frac{3}{16}\,c(\alpha)C\right\rceil \,(\ge 3)$.
However, this is incompatible, for any $\alpha \ge 2$ and $C \ge 6$, with
\begin{align*}
\frac{2^{Z_2}}{Z_2^2}<\frac{2^\alpha \log^2 C}{\log(2^\alpha -1) \cdot \log \max \{11,2^\alpha +1\}},
\end{align*}
which is shown in the same way to show Lemma \ref{z1z2_upper}\,(i).
Therefore,
\begin{equation}\label{cf_second}
\frac{\log C}{\log B}-\frac{Y_3}{Z_3} < \frac{2}{Z_3A^{X_2-X_3}C^{Z_3-Z_2}\log B} \le \frac{1}{2Z_3^2}.
\end{equation}

To sum up, by a well-known criterion of Legendre on the continued fraction, we may conclude, from inequalities \eqref{cf_first} and \eqref{cf_second}, that both $\frac{Y_1}{Z_1},\frac{Y_3}{Z_3}$ are convergents to $\frac{\log C}{\log B}$.
The fact that $\frac{Y_1}{Z_1}\ne \frac{Y_3}{Z_3}$ follows from Lemma \ref{two} for $(X,Y,Z)=(X_1,Y_1,Z_1)$ and $(X',Y',Z')=(X_3,Y_3,Z_3)$.
Moreover, Lemma \ref{cf-2} together with \eqref{cf_first} and \eqref{cf_second} tells us that
\begin{equation}\label{gap_Z1=Z2_no2}
Z_3>\frac{\log B}{2}\,C^{Z_1/2}=\frac{\log B}{2}\,C^{Z_2/2},
\end{equation}
or $Z_1>\frac{\log B}{2}\,A^{X_2-X_3}C^{Z_3-Z_2}$.
It is easily observed that the latter inequality does not hold similarly to the observation to show \eqref{cf_second}.
The assertion follows from \eqref{gap_Z1=Z2_no1} and \eqref{gap_Z1=Z2_no2}, together with the consideration on symmetries of indices 1 and 2, and of $A$ and $B$.
\end{proof}

\begin{prop}\label{=<lam-1}
Consider the case where $A<C.$
Suppose that equation \eqref{ABC} for $\lambda=-1$ has three solutions $(X,Y,Z)=(X_r,Y_r,Z_r)$ with $r \in \{1,2,3\}$ such that $Z_1=Z_2<Z_3.$
Then the following hold.
\begin{itemize}
\item[\rm (i)]
If $X_3>\max\{X_1,X_2\},$ then one of the following inequalities holds.
\begin{gather*}
C^{Z_2}/Z_2<\max_{t \in \{1,2\}} |X_t Z_3 - X_3 Z_2|,\\
C^{Z_2/2}/Z_2 <\max_{t \in \{1,2\}} |Y_t Z_3- Y_3 Z_2|,\\
C^{Z_2/2}<\frac{2}{\log C}\,X_3.
\end{gather*}
\item[\rm (ii)]
If $X_3 \le \max\{X_1,X_2\},$ then one of the following inequalities holds.
\begin{gather*}
C^{Z_2}<\gcd (X_3,Z_3)\max_{t \in \{1,2\}}|  X_t Z_3-X_3 Z_2|,\\
C^{Z_2}/Z_2^2<|X_1-X_2|\gcd(Y_3,Z_3)\,\max_{t \in \{1,2\}}| Y_t Z_3 - Y_3 Z_2|.
\end{gather*}
\end{itemize}
\end{prop}

\begin{proof}
Since $Z_1=Z_2$,
\begin{equation}\label{Z1=Z2_-}
A^{X_1}-B^{Y_1}=A^{X_2}-B^{Y_2}=C^{Z_2}.
\end{equation}
Note that $X_1 \ne X_2$ and $Y_1 \ne Y_2$, and that $X_1<X_2$ if and only if $Y_1<Y_2$.
Also, $X_1>1$ as $A<C$ in equations \eqref{Z1=Z2_-}.
By symmetry of indices 1 and 2 in the assertion, we may assume that $X_1<X_2$ and $Y_1<Y_2$.
Thus,
\begin{equation}\label{X1<X2_-}
1<X_1<X_2.
\end{equation}
In particular, $A^{X_1} \equiv 0 \pmod{4}$.
Also, from \eqref{Z1=Z2_-},
\[
B^{Y_1} \mid (A^{X_2-X_1}-1), \quad A^{X_1} \mid (B^{Y_2-Y_1}-1).
\]
Thus
\begin{equation}\label{gap_eq-}
A^{X_2-X_1}>B^{Y_1}, \quad B^{Y_2-Y_1}>A^{X_1}.
\end{equation}
Let us consider several cases separately.

First, suppose that
\[
Y_2 \le  Y_3.
\]
Observe that the equation $A^{X'}-C^{Y'}=B^{Z'}$ has three solutions $(X',Y',Z')=(X_r,Z_r,Y_r)$ with $r \in \{1,2,3\}$ satisfying $Y_1<Y_2 \le Y_3$.
Since $B^{Y_1}$ is odd, Proposition \ref{improvedgap}\,(ii) yields
\[
\gcd (X_2,Z_2)\,|X_2Z_3-X_3Z_2| \ge B^{Y_2-Y_1}.
\]
As $B^{Y_2-Y_1}>A^{X_1}>C^{Z_2}$ by \eqref{Z1=Z2_-} and \eqref{gap_eq-}, we have
\begin{equation}\label{gap_Z1=Z2_-_no1}
\gcd (X_2,Z_2)\,|X_2Z_3-X_3Z_2| >C^{Z_2}.
\end{equation}

Second, suppose that
\[
Y_3<Y_2, \quad X_2<X_3.
\]
Then the equation $C^{X'}+B^{Y'}=A^{Z'}$ has three solutions $(X',Y',Z')=(Z_r,Y_r,X_r)$ with $r \in \{1,2,3\}$ satisfying $X_1<X_2<X_3$.
Proposition \ref{improvedgap}\,(ii) yields
\[
\gcd (Z_2,Y_2)\,|Z_2Y_3-Z_3Y_2| \ge A^{X_2-X_1}.
\]
If $B^{2Y_1} \ge A^{X_1}$, then, since $A^{X_2-X_1}>B^{Y_1}$ by \eqref{gap_eq-},
\begin{equation}\label{gap_Z1=Z2_-_no2}
\gcd (Y_2,Z_2)\,|Y_2Z_3-Y_3Z_2|>B^{Y_1} \ge A^{X_1/2}>C^{Z_2/2}.
\end{equation}
Suppose that $B^{2Y_1}<A^{X_1}$.
Lemma \ref{cf-1}\,(i) for $(X,Y,Z)=(X_1,Y_1,Z_1)$ gives
\[
\hspace{-1cm}
(0<) \quad \frac{\log A}{\log C}-\frac{Z_1}{X_1}<\frac{2}{X_1A^{X_1/2}\log C}.
\]
On the other hand, since $Y_3<Y_2$ and $X_2<X_3$, Lemma \ref{cf-1}\,(ii) for $(X,Y,Z)=(Y_2,Z_2,X_2), (X',Y',Z')=(Y_3,Z_3,X_3)$ gives
\[
\hspace{-1cm}
(0<) \quad \frac{\log A}{\log C}-\frac{Z_3}{X_3}<\frac{2}{X_3B^{Y_2-Y_3}A^{X_3-X_2}\log C}.
\]
Similarly to the arguments in the proof of Proposition \ref{=<lam1}, we use the above two displayed inequalities to show, by the criterion of Legendre, that $\frac{Z_1}{X_1},\frac{Z_3}{X_3}$ are distinct convergents to $\frac{\log A}{\log C}$.
Then Lemma \ref{cf-2} shows that
\begin{equation}\label{gap_Z1=Z2_-_no3}
X_3>\frac{\log C}{2}A^{X_1/2}>\frac{\log C}{2}\,C^{Z_1/2},
\end{equation}
or $X_1>\frac{\log C}{2}\,A^{X_3-X_2}B^{Y_2-Y_3}$.
It is shown that the latter inequality does not hold as observed in showing \eqref{cf_second}.

To sum up, assertion (i) follows from the combination of inequalities \eqref{gap_Z1=Z2_-_no1}, \eqref{gap_Z1=Z2_-_no2}, \eqref{gap_Z1=Z2_-_no3}, together with the consideration on symmetry of indices 1 and 2.

Finally, suppose that
\[
Y_3<Y_2, \quad X_3 \le  X_2.
\]
If $Y_1 \le  Y_3$ and $X_3 \le  X_1$, then $A^{X_1}=B^{Y_1}+C^{Z_1}<B^{Y_3}+C^{Z_3}=A^{X_3} \le A^{X_1}$, which is absurd.
Thus, $Y_3<Y_1$ or $X_1<X_3$.

Consider the case where $Y_3<Y_1\,(<Y_2)$ and $X_3 \le X_1$.
Observe that the equation $A^{X'}-C^{Y'}=B^{Z'}$ has three solutions $(X',Y',Z')=(X_r,Z_r,Y_r)$ with $r \in \{3,1,2\}$ satisfying $Y_3<Y_1<Y_2$.
Then Proposition \ref{=<lam1}\,(ii) yields
\[
\gcd (X_1,Z_1)\,| X_1 Z_2 - X_2 Z_1| \ge B^{Y_1-Y_3}.
\]
On the other hand,
\[
A^{X_3}-B^{Y_3}=C^{Z_3}=C^{Z_3-Z_1}(A^{X_1}-B^{Y_1}).%=C^{Z_3-Z_1}A^{X_1}-C^{Z_3-Z_1}B^{Y_1}.
\]
Since $X_3 \le X_1$, we take this relation modulo $A^{X_3}$ to see that
\[
C^{Z_3-Z_1}B^{Y_1-Y_3} \equiv 1 \mod{A^{X_3}}.
\]
This gives that $C^{Z_3-Z_1}B^{Y_1-Y_3}>A^{X_3}>C^{Z_3}$, and so
\[
B^{Y_1-Y_3}>C^{Z_1}.
\]
These obtained inequalities together yield
\begin{equation}\label{gap_Z1=Z2_-_no4}
\gcd (X_1,Z_1)\,| X_1 Z_2 - X_2 Z_1| >C^{Z_1}.
\end{equation}

Consider the case where $Y_3<Y_1\,(<Y_2)$ and $X_1<X_3 \le  X_2$.
Since the equation $C^{X'}+B^{Y'}=A^{Z'}$ has three solutions $(X',Y',Z')=(Z_r,Y_r,X_r)$ with $r \in \{1,3,2\}$ satisfying $X_1<X_3 \le  X_2$, Proposition \ref{=<lam1}\,(ii) yields
\[
\gcd(Z_3,Y_3)\,| Z_3 Y_2 - Z_2 Y_3| \ge A^{X_3-X_1}.
\]
On the other hand, taking the relation $B^{Y_3}+C^{Z_3}=A^{X_3}=A^{X_3-X_1}(B^{Y_1}+C^{Z_1})$ modulo $C^{Z_1}$ yields $A^{X_3-X_1}B^{Y_1-Y_3} \equiv 1 \pmod{C^{Z_1}}$, in particular,
\[
A^{X_3-X_1}B^{Y_1-Y_3}>C^{Z_1}.
\]
Since $Z_2\gcd (X_1,Z_1)\,|X_1-X_2| \ge B^{Y_1-Y_3}$ as seen in the previous case, these obtained inequalities together yield
\begin{equation}\label{gap_Z1=Z2_-_no5}
\gcd(Y_3,Z_3)\,| Y_2Z_3  - Y_3Z_2 | \cdot Z_2\gcd (X_1,Z_1)\,|X_1-X_2|>C^{Z_1}.
\end{equation}

Consider the case where $Y_1 \le  Y_3<Y_2$ and $X_1<X_3 \le  X_2$.
Taking the equation $B^{Y_3}+C^{Z_3}=A^{X_3-X_1}(B^{Y_1}+C^{Z_1})$ modulo $C^{Z_1}$ yields $A^{X_3-X_1} \equiv B^{Y_3-Y_1} \pmod{C^{Z_1}}$, in particular,
\[
\max\{A^{X_3-X_1},B^{Y_3-Y_1}\}>C^{Z_1}.
\]
Suppose that $A^{X_3-X_1}>C^{Z_1}$.
Since the equation $C^{X'}+B^{Y'}=A^{Z'}$ has three solutions $(X',Y',Z')=(Z_r,Y_r,X_r)$ with $r \in \{1,3,2\}$ satisfying $X_1<X_3 \le  X_2$, Proposition \ref{=<lam1}\,(ii) yields
\begin{equation}\label{gap_Z1=Z2_-_no6}
\gcd (Z_3,Y_3)\,| Z_3 Y_2 - Z_2 Y_3| \ge A^{X_3-X_1}>C^{Z_1}.
\end{equation}
Suppose that $B^{Y_3-Y_1}>C^{Z_1}$.
Since $Y_1<Y_3$, and the equation $A^{X'}-C^{Y'}=B^{Z'}$ has three solutions $(X',Y',Z')=(X_r,Z_r,Y_r)$ with $r \in \{1,3,2\}$ satisfying $Y_1<Y_3<Y_2$, Proposition \ref{=<lam1}\,(ii) yields
\begin{equation}\label{gap_Z1=Z2_-_no7}
\gcd (X_3,Z_3)\,| X_3 Z_2 - X_2 Z_3| \ge B^{Y_3-Y_1}>C^{Z_1}.
\end{equation}

To sum up, assertion (ii) follows from the combination of inequalities \eqref{gap_Z1=Z2_-_no1}, \eqref{gap_Z1=Z2_-_no4},
\eqref{gap_Z1=Z2_-_no5}, \eqref{gap_Z1=Z2_-_no6} and \eqref{gap_Z1=Z2_-_no7}, together with the consideration on symmetry of indices 1 and 2.
\end{proof}

%%%%%%%%%%%%%%%%%%%%%%%
\section{Applications}\label{Sec-ineq-z1<z2}%
%%%%%%%%%%%%%%%%%%%%%%%

Here we give three applications of Proposition \ref{improvedgap} to equation \eqref{abc} having three solutions.
For this we prepare some notation.

Based upon \eqref{zorders}, let $(i,j,k)$ and $(l,m,n)$ be permutations of $\{1,2,3\}$ such that
\[
x_i \le x_j \le x_k, \quad y_l \le y_m \le y_n.
\]
To ensure the uniqueness of these, we assume that $i<j$ if $x_i=x_j$, and $j<k$ if $x_j=x_k$, and that $l<m$ if $x_l=x_m$, and $m<n$ if $x_m=x_n$.
%Note that $z_1<z_3, x_i<x_k, y_l<y_n$ by Proposition \ref{pillai_atmost2}.
Also, define non-negative integers $d_z,d_x,d_y$ and positive integers $g_2,g_x,g_y$ as follows:
\begin{alignat*}{3}
& d_z:=z_2-z_1, &  &d_x:=x_j-x_i, & & d_y:=y_m-y_l,\\
& g_2:=\gcd(x_2,y_2), & \quad  &g_x:=\gcd(y_j,z_j), & \quad &g_y:=\gcd(x_m,z_m).
\end{alignat*}

\begin{lem}\label{Lem_ineq_z1<z2_c^{z1}=2(4)}
Suppose that $d_z>0$ with $c^{z_1} \equiv 2 \pmod{4}.$
Then
\[
c<\min\!\left\{ 2^{\alpha+1-z_2} \frac{({g_2}')^2}{g_2},\,  \frac{({g_2}')^2}{g_2},\, \frac{\log c}{\log (c-1)}\,t_{a,b}\,z_2\right\} \cdot z_2\,\mathcal H_{\alpha,1,m_2}(c)
\]
with ${g_2}'=\gcd(c,g_2).$
\end{lem}

\begin{proof}
Note that $(\beta,z_1)=(1,1)$.
We apply Proposition \ref{improvedgap}\,(I) for $(A,B,C;\lambda)=(a,b,c;1)$ and $(X_r,Y_r,Z_r)=(x_r,y_r,z_r)$ with $r \in \{1,2,3\}$. Then
\begin{gather}
\label{div1}
c \mid ({g_2}')^2 \cdot \frac{x_2y_3-x_3y_2}{g_2},\\
\label{ineq1}
c<\frac{\log c}{\log (c-1)} \cdot t_{a,b}\,z_2 \cdot |x_2y_3-x_3y_2|.
\end{gather}
By \eqref{x2y3x3y2}, it is easy to see that the assertion for the second part in $\min$ follows from \eqref{div1}, and the third one follows from \eqref{ineq1}.
It remains to consider the first part.

Since $2 \parallel c$, and $g_2$ is odd by Lemma \ref{g2}\,(i), we use divisibility relation \eqref{div1} to see that $(cg_2/2)/({g_2}')^2$ is an odd positive divisor of $x_2y_3-x_3y_2$.
Thus
\begin{align*}
\nu_{2}(x_2y_3-x_3y_2)&=
\nu_2 \left( \frac{x_2y_3-x_3y_2}{(cg_2/2)/({g_2}')^2} \right)\\
&\le \frac{1}{\log 2} \log \left( \frac{|x_2y_3-x_3y_2|}{(cg_2/2)/({g_2}')^2} \right)\\
&=1-\frac{\log c}{\log 2}+\frac{\log \Bigl(\frac{({g_2}')^2}{g_2}\cdot  |x_2y_3-x_3y_2|\Bigl) }{\log 2}.%+\frac{\log }{\log 2}.
\end{align*}
Since $z_2 \le \alpha +\nu_{2}(x_2y_3-x_3y_2)$ by \eqref{z2} for $t=2$, it follows that
\begin{align*}
\frac{\log c}{\log 2}\le -z_2+\alpha +1+\frac{\log (\frac{({g_2}')^2}{g_2})}{\log 2}+\frac{\log  |x_2y_3-x_3y_2|}{\log 2}.
\end{align*}
This together with \eqref{x2y3x3y2} implies the remaining assertion.
\end{proof}

\begin{lem}\label{Lem_ineq_z1<z2_c^z1=0(4)}
Suppose that $d_z>0$ with $c^{z_1} \equiv 0 \pmod{4}.$ Then
\[
c^{d_z}<\min \! \left \{ 2^{\alpha -\beta z_1} \frac{({g_2}')^2}{g_2}, \,
%\frac{({g_2}')^2}{g_2}, \,
\frac{z_1 \log c}{\log (\chi c^{z_1}-1)}\,t_{a,b}\,\frac{z_2}{z_1} \right \}\cdot \frac{\log^2 c}{(\log a) \log b}\,z_2z_3
\]
with ${g_2}'=\gcd (c^{d_z},g_2),$ where $\chi =2$ if $z_1>1,$ and $\chi=1$ if $z_1=1,$
\end{lem}

\begin{proof}
The proof proceeds along similar lines to that of Lemma \ref{Lem_ineq_z1<z2_c^{z1}=2(4)}.
We apply Proposition \ref{improvedgap}\,(II) for $(A,B,C;\lambda)=(a,b,c;1)$ and $(X_r,Y_r,Z_r)=(x_r,y_r,z_r)$ with $r \in \{1,2,3\}$. Then
\begin{gather}
\label{div} c^{z_2-z_1} \mid ({g_2}')^2 \cdot \frac{x_2y_3-x_3y_2}{g_2},\\
\label{ineq} c^{z_2-z_1}< \frac{z_1 \log c}{\log (\chi c^{z_1}-1)} \cdot t_{a,b} \cdot \frac{z_2}{z_1} \cdot |x_2y_3-x_3y_2|.
\end{gather}
By \eqref{x2y3x3y2}, the assertion for the second part in $\min$ follows from \eqref{ineq}.
It remains to consider the first part.

Since $2^\beta \parallel c$, and $g_2$ is odd by Lemma \ref{g2}\,(i), we use \eqref{div} to see that $(c/2^{\beta})^{z_2-z_1}g_2/({g_2}')^2$ is an odd positive divisor of $x_2y_3-x_3y_2$.
Thus
\begin{align*}
\nu_{2}(x_2y_3-x_3y_2)&=\nu_2 \left( \frac{x_2y_3-x_3y_2}{(c/2^{\beta})^{z_2-z_1}g_2/({g_2}')^2} \right)\\
&\le \frac{1}{\log 2} \log \left( \frac{|x_2y_3-x_3y_2|}{(c/2^{\beta})^{z_2-z_1}g_2/({g_2}')^2} \right)\\
&=\beta(z_2-z_1) -(z_2-z_1)\,\frac{\log c}{\log 2}+\frac{\Bigl(\frac{({g_2}')^2}{g_2}\cdot  |x_2y_3-x_3y_2|\Bigl)}{\log 2}.
\end{align*}
Since $\beta z_2 \le \alpha +\nu_{2}(x_2y_3-x_3y_2)$ by \eqref{z2} for $t=2$, it follows that
\[
(z_2-z_1)\,\frac{\log c}{\log 2} \le \alpha- \beta z_1 +\frac{\log (\frac{({g_2}')^2}{g_2})}{\log 2}+\frac{\log  |x_2y_3-x_3y_2|}{\log 2}.
\]
This together with \eqref{x2y3x3y2} implies the remaining assertion.
\end{proof}

\begin{lem}\label{Lem_ineq_xi<xj}
Suppose that $d_x>0.$ Then
\[ \tag{i}
a^{d_x}<\frac{({g_x}')^2}{g_x} \cdot \frac{\log c}{\log b} \cdot z_j z_k \le \frac{\log c}{\log b} \cdot {z_j}^2 z_k
\]
with ${g_x}'=\gcd (a^{d_x},g_x).$ Moreover, if $g_x>1,$ then
\begin{gather*}
\tag{ii} a^{d_x}<\frac{({g_x}')^2}{g_x-1} \cdot \frac{\log a}{\log b} \cdot x_j z_k;\\
\tag{iii} a^{d_x}<\left(\frac{g_x}{g_x-1}\right)^2 \cdot \frac{\log^2 a}{\log (a-1)\,\log b} \cdot t_{b,c} \cdot (x_j+d_x+{d_x}^2)\,z_k.
\end{gather*}
\end{lem}

Note that Lemma \ref{Lem_ineq_xi<xj} holds with $(a,x,i,j,k)$ replaced as $(b,y,l,m,n)$ by symmetry of $a$ and $b$.

\begin{proof}[Proof of Lemma \ref{Lem_ineq_xi<xj}]
Apply Proposition \ref{improvedgap}\,(II) for $(A,B,C;\lambda)=(c,b,a;-1)$ with $(X_r,Y_r,Z_r)=(z_t,y_t,x_t)$ with $(r,t) \in \{(1,i),(2,j),(3,k)\}$.
Then
\begin{equation}\label{ineq1_gx>1}
a^{x_j-x_i} \mid ({g_x}')^2 \cdot \frac{y_j z_k-y_k z_j}{g_x}.
\end{equation}
Moreover, if $g_x>1$, then
\begin{equation}\label{ineq2_gx>1}
a^{x_j-x_i}< \frac{\log a}{\log (a-1)} \cdot \frac{g_x}{g_x-1} \cdot t_{b,c} \cdot \frac{x_j}{x_i} \cdot |y_j z_k-y_k z_j|.
\end{equation}

From \eqref{basic} for $t \in \{j,k\}$,
\[
|y_j z_k-y_k z_j |<\max\{y_j z_k,y_k z_j\} <\frac{\log c}{\log b}\,z_j z_k.
\]
Since $a^{x_j-x_i} \le \frac{({g_x}')^2}{g_x} \cdot |y_j z_k-y_k z_j |$ by \eqref{ineq1_gx>1}, the above inequality yields (i).

By Lemma \ref{g2_upper_a} for $(X,Y,Z)=(z_j,y_j,x_j)$,
\[
y_j<\frac{g_x}{g_x-1}\,\frac{\log c}{\log b}\,x_j, \quad z_j<\frac{g_x}{g_x-1}\,\frac{\log a}{\log c}\,x_j.
\]
Since $y_k<\frac{\log c}{\log b}\,z_k$, we have
\[
|y_j z_k-y_k z_j|<\max\{y_j z_k,y_k z_j\}<\frac{g_x}{g_x-1}\frac{\log a}{\log b}\,x_j z_k.
\]
This together with \eqref{ineq1_gx>1} yields (ii).
Similarly, (iii) follows from \eqref{ineq2_gx>1} and the inequality ${x_j}^2/x_i \le x_j+d_x+{d_x}^2$.
\end{proof}

%%%%%%%%%%%%%%%%%%%%%%%%%%%%%%
\section{Restrictions on Common divisors among solutions} \label{Sec-gfe}
%%%%%%%%%%%%%%%%%%%%%%%%%%%%%%

The following is a well-known conjecture as a generalization of Fermat's last theorem, known as the generalized Fermat conjecture.

\begin{conj}\label{gfe}
Let $p,q$ and $r$ be any positive integers satisfying $1/p+1/q+1/r<1.$
Then all solutions $(X,Y,Z)$ with $XYZ \ne 0$ and $\gcd(X,Y)=1$ of the Diophantine equation
\begin{equation}\label{pqr}
X^p+Y^q=Z^r
\end{equation}
come from the following ten identities$:$
\begin{gather*}
1^p+2^3=3^2, \ 7^2+2^5=3^4, \ 13^2+7^3=2^9, \\
17^3+2^7=71^2, \ 11^4+3^5=122^2, \ 1549034^2+33^8=15613^3,\\
96222^3+43^8=30042907^2, \ 2213459^2+1414^3=65^7,\\
15312283^2+9262^3=113^7, \ 76271^3+17^7=21063928^2.
\end{gather*}
\end{conj}

The following is just a collection, needed for our purpose, from the existing results on Conjecture \ref{gfe} (cf.~\cite{BeChDaYa}, \cite{BeMiSi}, \cite[Ch.14]{Co}).

\begin{lem}\label{ourcase}
Conjecture \ref{gfe} is true for any $(p,q,r)$ in the following table{\rm :}
\begin{table}[H]
\begin{tabular}{| c | c |} \hline
$(p,q,r)$ & {\rm reference(s)}  \\ \hline% \hline
$(N,N,N), \ N \ge 3$ & {\rm \cite{Wi}, \cite{TaWi}}  \\
$(N,N,2), \ N \ge 4$ & {\rm \cite{DaMe}, \cite{Po}} \\
$(N,N,3), \ N \ge 3$ & {\rm \cite{DaMe}, \cite{Po}}  \\
$(2,4,N), \ N \ge 4$ & {\rm \cite{El}, \cite{BeElNg}, \cite{Br} } \\
$(2,N,4), \ N \ge 4$ & {\rm \cite{BeSk}, \cite{Br}} \\
$(2,N,6), \ N \ge 3$ & {\rm \cite{BeChDaYa}, \cite{Br}} \\
$(2,6,N), \ N \ge 3$ & {\rm \cite{BeCh}, \cite{Br}} \\
$(3,3,N), \ 3 \le N \le 10^9$ & {\rm \cite{Kr}, \cite{Br2}, \cite{Da}, \cite{DaSik2}}  \\
$(2,3,N), \ N \in \{7,8,9,10,15\}$ & {\rm \cite{PoShSt}, \cite{Br}, \cite{Zu}, \cite{Sik}, \cite{SikSt}} \\
$(3,4,5)$ & {\rm \cite{SikSt} } \\
$(5,5,7),(7,7,5)$ & {\rm \cite{DaSik} } \\
$(5,5,N), \ N \ge 2, \ 5 \mid Z$ & {\rm \cite{DaSik} }  \\  \hline
\end{tabular}
\end{table}
The result of \cite{DaSik} in the last line of the above table indicates that equation \eqref{pqr} with $(p,q,r)=(5,5,N)$ and $N \ge 2$ has no solutions satisfying $5 \mid Z$.
\end{lem}

As almost direct consequences of Lemma \ref{ourcase}, together with Lemma \ref{firstbounds}, we can show the following lemmas which are useful to restrict the values of $g_2,g_x$ and $g_y$ appearing in the stated inequalities in the lemmas in the previous section.

\begin{lem}\label{g2}
The following hold.
\begin{itemize}
\item[\rm (i)] If $2 \mid g_2,$ then $(\beta,z_1,z_2)=(1,1,1).$
\item[\rm (ii)] If $3 \mid g_2,$ then $z_2 \le 2.$
\item[\rm (iii)] If $5 \mid g_2,$ then $5 \nmid c$ or $z_2=1.$
\item[\rm (iv)] Suppose that $2 \mid z_2.$ Then $g_2 \in \{1,3\}.$ Moreover, $g_2=1$ if $z_2>2.$
\item[\rm (v)] If $3 \mid z_2,$ then $g_2=1.$
\end{itemize}
\end{lem}

\if0 \begin{proof}
(i) If $2 \mid g_2$, then taking 2nd equation modulo 4 implies that $2 \equiv c^{z_2} \pmod{4}$.
This yields the assertion.

(ii) This is a simple consequence of the combination of Lemma \ref{firstbounds},
\cite[Theorem 10\,(a)]{Kr} and \cite[Theorem 1;$(p,q,r)=(3,3,2n)$]{BeChDaYa}.

(iii) This is a direct consequence of  \cite[Proposition 3.3]{DaSi}.

(iv) \cite[Main Theorem 2]{DaMe} tells us that $g_2 \le 3$, and so $g_2 \in \{1,3\}$ by (i).
The second assertion follows from (ii).

(v) \cite[Main Theorem 3]{DaMe} tells us that $g_2 \le 2$, and so $g_2=1$ by (i).
\end{proof} \fi

\begin{lem}\label{gx}
The following hold.
\begin{itemize}
\item[\rm (i)] If $3 \mid g_x$ and $x_j \le 10^9,$ then $x_j \le 2.$
\item[\rm (i)] If $4 \mid g_x$ or $6 \mid g_x,$ then $x_j=1.$
\item[\rm (ii)] If $5 \mid g_x,$ then $5 \nmid a$ or $x_j=1.$
\item[\rm (iii)] Suppose that $2 \mid x_j.$ Then $g_x \in \{1,3\}.$
                         Moreover, if $4 \mid x_j,$ then $g_x=1.$
\item[\rm (iv)]  If $3 \mid x_j,$ then $g_x \le 2.$
\item[\rm (v)]   If $x_j \ge 3,$ then $3 \nmid g_x.$
\end{itemize}
\end{lem}

\if0 \begin{proof}
(i) By \cite{Da}.

(ii) By \cite{DaSi}.

(iii) The assertion in the case where $2 \nmid g_x$ is proved almost similarly to Lemma \ref{g2}\,(iv).
If $2 \mid g_x$, then $2 \nmid x_j$, otherwise $\{a^{x_j/2},b^{y_j/2},c^{z_j/2}\}$ would form a primitive Pythagorean triple with $2 \mid c^{z_j/2}$.

(iv) The assertion in the case where $4 \nmid g_x$ is proved almost similarly to Lemma \ref{g2}\,(v).
In addition, \cite[Proposition 14.6.6]{Co} tells us that $4 \nmid g_x$ if $3 \mid x_j$.

(v) This is similar to Lemma \ref{g2}\,(ii).
\end{proof} \fi

Note that Lemma \ref{gx} holds with $(a,x,i,j,k)$ replaced as $(b,y,l,m,n)$ by symmetry of $a$ and $b$.

%%%%%%%%%%%%%%%%%%%%%%%%%%
\section{Diophantine equation $A^x+B^y=A^X+B^Y$} \label{Sec-AB}
%%%%%%%%%%%%%%%%%%%%%%%%%%

Let $A$ and $B$ be coprime integers greater than 1.
Here we study the following purely exponential equation:
\begin{equation}\label{AB}%\tag{AB}
A^x+B^y=A^X+B^Y,
\end{equation}
where $x,y,X,Y$ are unknown positive integers with $x \ne X$ and $y \ne Y$.
It is easy to see that $x<X$ if and only if $y>Y$, and also that the case where $X/x=y/Y \in \N$ does not hold.

Below, we give two results on equation \eqref{AB}.

\begin{lem}\label{AB-basic}
Let $(x,y,X,Y)$ be a solution of equation \eqref{AB} with $x<X$ and $y>Y.$
Then the following hold.
\begin{itemize}
\item[\rm (i)] $B^Y \mid (A^{X-x}-1), \ A^x \mid (B^{y-Y}-1).$
\item[\rm (ii)] $x/X+Y/y<1.$
\item[\rm (iii)] If $A>B,$ then $y>X$ and $y \ge 4.$
\end{itemize}
\end{lem}

\begin{proof}
(i) The assertions readily follow from the equation:
\begin{equation}\label{AB-2}
A^{x}(A^{X-x}-1)=B^{Y}(B^{y-Y}-1)
\end{equation}
with $\gcd(A,B)=1$.

(ii) From (i), $A^{x}<B^{y-Y}$ and $B^{Y}<A^{X-x}$, that is, $\frac{x}{y-Y}<\frac{\log B}{\log A}<\frac{X-x}{Y}$, so $\frac{xY}{Xy}<(1-\frac{x}{X})(1-\frac{Y}{y})$.
This yields the assertion.

(iii) We follow an argument in \cite[p.213]{Lu} to see that if $y \le X$ then
%\begin{align*}
\[
A^{X-1}+B^{X-1}
%& \le (A-B)(A^{X-1}+A^{X-2}B+\cdots +AB^{X-2}+B^{X-1})\\ &
-1 \le A^X-B^X \le A^X-B^y=A^x-B^Y<A^x \le A^{X-1}.
\]
%\end{align*}
This contradiction shows that $y>X$.

Suppose that $y \le 3$. Then $(X,x,y)=(2,1,3)$ as $y>X>x$.
Thus, $Y=1$ by (ii), so $A+B^3=A^2+B$.
This yields an integral point $(\mathcal X,\mathcal Y)=(B,A)$ of the elliptic curve $\mathcal Y^2-\mathcal Y=\mathcal X^{3}-\mathcal X$.
However, none of those points gives a proper pair $(A,B)$. Thus, $y \ge 4$.
\end{proof}

\begin{lem}\label{X=x+1}
Let $(x,y,X,Y)$ be a solution of equation \eqref{AB} with $x<X$ and $y>Y.$
Suppose that $X-x=1.$
Then $A>B,$ and the following hold.
\begin{itemize}
\item[\rm (i)]
$A \equiv -x B^{2Y}-B^Y+1 \pmod{B^{3Y}}.$
In particular, $A \ge B^{3Y}-xB^{2Y} -B^Y+1.$
\item[\rm (ii)]
Assume that $B>2.$
Then one of the following cases holds.
\begin{itemize}
\item[\rm (ii-1)]
$y>(3Y-1)X,$ and
\[
A \ge \begin{cases}
\,\frac{1}{2}B^{3Y}+\frac{1}{2}B^{2Y}-B^Y+1, & \text{if $B$ is odd,}\\
\,\frac{1}{2}B^{3Y}-B^Y+1, & \text{if $B$ is even.}
\end{cases}
\]
\item[\rm (ii-2)]
It holds that
\[
A=r B^{2Y}-B^Y+1,\quad x \ge B^Y+r,
\]
where $r$ is some positive integer satisfying $r \equiv -x \pmod{B^Y}$ with $r \le \lfloor \frac{B^Y}{2} \rfloor.$
\end{itemize}
In particular, case {\rm (ii-1)} holds if $x \le B^Y.$
\end{itemize}
\end{lem}

\begin{proof}
First, under the assumption that $(X-x) \mid x$, we show the following congruence:
\begin{equation}\label{precisecong}
A^{X-x} \equiv -B^{Y}+1 \mod{B^{2Y}}
\end{equation}
By Lemma \ref{AB-basic}\,(i), $A^{X-x}=1+KB^{Y}$ with some $K \in \N$.
Substituting this into \eqref{AB-2} yields
\[
A^{x} \cdot K=B^{y-Y}-1.
\]
Suppose that $(X-x) \mid x$.
Then $A^{x} \equiv 1 \pmod{B^{Y}}$.
We take the above displayed equality modulo $B^{\,\min\{y-Y,Y\}}$ to see that
\[
K \equiv -1 \mod{B^{\,\min\{y-Y,Y\}}}.
\]
Thus, for obtaining \eqref{precisecong}, it suffices to show that $y \ge 2Y$.
Since $B^{y-Y}>A^{x}, \ A^{X-x}>B^{Y}$ by Lemma \ref{AB-basic}\,(i), and $\frac{x}{X-x} \ge 1$ as $(X-x) \mid x$, it follows that
\[
y-Y>\frac{\log A}{\log B} \cdot x>\frac{Y}{X-x} \cdot x =\frac{x}{X-x} \cdot Y \ge Y.
\]
In what follows, suppose that $X-x=1$.

(i) By congruence \eqref{precisecong}, $A=LB^{2Y}-B^{Y}+1$ with some $L \in \N$.
Substituting this into \eqref{AB-2} yields
\[
(LB^{2Y}-B^{Y}+1)^{x}(LB^{Y}-1)=B^{y-Y}-1.
\]
Observe that
\begin{align*}
(LB^{2Y}-B^{Y}+1)^{x}(LB^{Y}-1)
&\equiv (-xB^{Y}+1)(LB^{Y}-1) \mod{B^{2Y}} \\
&\equiv (x+L)B^{Y}-1 \mod{B^{2Y}}.
\end{align*}
We take the previous equality modulo $B^{2Y}$ to see that $(x+L)B^{Y} \equiv B^{y-Y} \pmod{B^{2Y}}$, and so
\[
x+L \equiv B^{y-2Y} \mod{B^{Y}}.
\]
It suffices to show that $y \ge 3Y$.
This follows from the inequalities $A \ge B^{2Y}-B^{Y}+1>B^{2Y-1}$, and $A \le A^{x}<B^{y-Y}$.

(ii) Set $r$ be the integer satisfying $r \equiv x \pmod{B^Y}$ with $|r| \le \lfloor \frac{B^Y}{2} \rfloor$.
By (i), $A=(TB^{Y}-r) B^{2Y}-B^Y+1$ with some $T \in \Z$.
It is clear that $TB^{Y}-r \ge 1$.
If $T<0$, then $-r \ge 1-TB^{Y} \ge 1+B^Y>\lfloor \frac{B^Y}{2} \rfloor$, which is absurd.
If $T=0$, then $A=-r B^{2Y}-B^Y+1$ with $r<0$, and that $x \ge B^Y-r \,(>B^Y)$ since $x \equiv r \pmod{B^Y}$, so case (ii-2) holds.

Finally, suppose that $T>0$.
Then $A \ge B^{3Y}-r B^{2Y}-B^Y+1$.
This together with $r \le \lfloor \frac{B^Y}{2} \rfloor$ easily gives the asserted lower bounds for $A$ in case (ii-1).
Also, from those observe that $A \ge (1/k) \cdot B^{3Y}$, where $k=2$ if $2 \nmid B$, and $k=64/31$ if $2 \mid B$ with $B>2$.
On the other hand, %from \eqref{AB-3},
\[
B^y=A^X \cdot \frac{1-\frac{1}{A}}{1-\frac{1}{B^{y-Y}}}>A^X \cdot (1-{\textstyle \frac{1}{A}}).
\]
Since $B/k \ge 3/2$ as $B>2$, these inequalities together show that
\begin{align*}
y>\frac{\log A}{\log B}\,X+\frac{\log (1-\frac{1}{A})}{\log B}
>(3Y -{\textstyle \frac{\log k}{\log B}})\,X-{\textstyle \frac{\log (3/2)}{\log B}}
\ge (3Y-1)X.
\end{align*}
Thus, case (ii-1) holds.
\end{proof}

In the forthcoming two sections, we apply results in Sections \ref{Sec-improve-HuLe}, \ref{Sec-ineq-z1<z2}, \ref{Sec-gfe} and \ref{Sec-AB} to find all possible values of $a,b,c$ (together with those of $\alpha,\beta$) and $x_1,y_1,z_1,x_2,y_2,z_2$.
Moreover, in the next section, we sieve all those remaining cases with $z_1=z_2$ by using the system formed of the first two equations, that is,
\begin{equation}\label{sys-12}
\begin{cases}
\,a^{x_1}+b^{y_1}=c^{z_1}, \\
\,a^{x_2}+b^{y_2}=c^{z_2}.
\end{cases}
\end{equation}
For these purposes, we prepare several notation and give a few remarks as follows.
In each of any forthcoming situations, let $\mathcal M_a, \mathcal M_b, \mathcal M_c$ denote any uniform upper bound for $a,b,c$, respectively.
Then any of programs using computer appearing depends on the sizes of these numbers, and it proceeds faster for smaller values of them.
Thus, through those programs, we always replace $\mathcal M_a, \mathcal M_b, \mathcal M_c$ by any smaller ones whenever those are found.
The details on the iterations coming from these are omitted in most cases in the text.
The situation is similar to lower bounds for $a,b$ and $c$.
In what follows, in each of the situations, let $a_0,b_0,c_0$ denote any uniform lower bound for $a,b,c$, respectively.
These numbers may be chosen appropriately according to each case together with \eqref{cond} and \eqref{la,lb,lc}.
For example, in any case with $a>b$, we may choose those numbers as follows:
\[
\begin{array}{ll}
a_0=\max\{19,2^\alpha+1,3 \cdot 2^\beta+1\}, \, c_0=3 \cdot 2^\beta,  &\text{if $a>\max\{b,c\}$,}\\
a_0=\max\{11,2^\alpha+1\}, \, c_0=\max\{18,3 \cdot 2^\beta,2^\alpha+2\},  &\text{if $c>a>b$,}
\end{array}
\]
with $b_0=2^\alpha-1$.

To treat system \eqref{sys-12}, it is very efficient to rely upon the existing results on ternary Diophantine equations, which are summarized in Lemma \ref{ourcase}.
Indeed, those restrict the divisibility properties of the exponential unknowns (as already seen in Lemmas \ref{g2} and \ref{gx}), and reduce considerably the computational time for showing results.
However, we omit the details on those applications for simplicity of the presentation.

%%%%%%%%%%%%%%%%%%%%%%%%%%
\section{Case where $z_1=z_2$}\label{Sec-z1=z2}
%%%%%%%%%%%%%%%%%%%%%%%%%

Here we examine the case of $z_1=z_2$, where system \eqref{sys-12} is
\begin{equation}\label{ab12} %\tag{ab12}
a^{x_1}+b^{y_1}=a^{x_2}+b^{y_2} =c^{z_1}.
\end{equation}
Without loss of generality, we may assume that $x_1<x_2$ and $y_1>y_2$.
Applying Lemma \ref{AB-basic} for $(A,B)=(a,b)$ and $(x,y,X,Y)=(x_1,y_1,x_2,y_2)$ gives
\begin{equation} \label{cond_z1=z2}
\begin{cases}
\ a^{x_1} \mid (b^{D_y}-1), \quad b^{y_2} \mid (a^{D_x}-1); \\
\ \max\{x_2,y_1\} \ge 4 \quad x_1/x_2+y_2/y_1<1;\\
\ y_1>x_2, \ \ x_1<D_y, \quad \text{if $a>b$;} \\
\ x_2>y_1, \ \ y_2<D_x, \quad \text{if $b>a$,}
\end{cases}
\end{equation}
where $D_x:=x_2-x_1, \, D_y:=y_1-y_2$.
Also, recall from Lemma \ref{dagger} that either $(\beta,z_1)=(1,1)$ or $z_1=\alpha/\beta$ if one of $d_x,d_y,D_x,D_y$ is odd.
In case $(\dagger)$ with $c<\max\{a,b\}$, we have $z_1=\alpha/\beta$ as $z_1>1$.
We often use these conditions implicitly below.

Let us begin with the following lemma.

\begin{lem}\label{z1=z2_first}
If $d_z=0,$ then
\[
c^{z_1/2}<\frac{\log c}{\log \min\{a,b\}}\,z_1^2 z_3, \quad \min\{a,b\}<c^{z_1/4}.
\]
\end{lem}

\begin{proof}
The second inequality immediately follows from system \eqref{ab12} with $\max\{x_2,y_1\} \ge 4$ by \eqref{cond_z1=z2}.
Suppose that $z_1=z_2$.
Apply Proposition \ref{=<lam1} with $(A,B,C)=(a,b,c)$ and $(X_r,Y_r,Z_r)=(x_r,y_r,z_r)$ for $r=1,2,3$.
Then $c^{z_2/2}<\frac{2}{\log \min\{a,b\}}\,z_3$, or
\begin{align*}
\frac{c^{z_2/2}}{z_2}
&<\max_{t \in \{1,2\}}\{  \,|x_3 z_2 - x_t z_3|, \ |y_3 z_2 - y_t z_3| \, \}\\
&<\max_{t \in \{1,2\}}\{  x_3 z_2,x_t z_3, y_3 z_2, y_t z_3 \}\\
&<\max_{t \in \{1,2\}}\left\{  \frac{\log c}{\log a}z_3 z_2,\frac{\log c}{\log a}z_t z_3, \frac{\log c}{\log b}z_3 z_2, \frac{\log c}{\log b}z_t z_3 \right\}
= \frac{\log c}{\log \min\{a,b\}}z_2 z_3.
\end{align*}
These together show the lemma as $z_2=z_1$.
\end{proof}

Using this lemma, we first deal with the case where $c>\max\{a,b\}$.

\begin{prop}\label{c=max_imply_z1<z2}
If $c>\max\{a,b\},$ then $d_z>0.$ %\marginpar{761sec}
\end{prop}

\begin{proof}
It suffices to consider the case where $a>b$.
Suppose on the contrary that $c>a>b$ and $z_1=z_2$.
Since $\frac{\log c}{\log b} z_3<\h(c;a,c)<\h(c)$ by \eqref{z3}, Lemma \ref{z1=z2_first} yields
\begin{equation} \label{ineq_c=max_z1=z2}
c^{z_1/2}<z_1^2 \cdot \h_{\alpha,\beta,\frac{\log 8}{\log 3}}(c).
\end{equation}
We use this inequality %with $m_2=\frac{\log 8}{\log 3}$
to find all possible values of the letters in system \eqref{ab12}, and we sieve them as follows.

By \eqref{la,lb,lc},
\begin{equation}\label{beta-alp_region}
\beta \le \left\lfloor \frac{\log (\mathcal M_c/3)}{\log 2} \right\rfloor, \ \ 2 \le \alpha \le \left\lfloor \frac{\log \mathcal M_{\min\{a,b\}}}{\log 2} \right\rfloor.
\end{equation}
Also, from \eqref{z1},
\begin{equation} \label{range_z1_c=max_z1=z2}
\beta=z_1=1 \quad \text{or} \quad \left\lceil\frac{\alpha}{\beta} \right\rceil \le z_1 \le \mathcal U_1(\alpha,\beta,\textstyle{\frac{\log 8}{\log 3}},a_0,b_0,\mathcal M_c,1).
\end{equation}

First, let us find a smaller upper bound for $c$.
Since $c^{1/2} \le c^{z_1/2}/z_1^2$, inequality \eqref{ineq_c=max_z1=z2} yields
\[
c^{1/2}<\max\bigr\{ \h_{2,1,\frac{\log 8}{\log 3}}(c), \h_{3,1}(c)\bigr\}.
\]
This implies that $c<1.4 \cdot 10^{13}$.
Thus we set $\mathcal M_c$ as this upper bound.

Next, for each of the values of $\beta,\alpha$ and $z_1$ satisfying \eqref{beta-alp_region} and \eqref{range_z1_c=max_z1=z2}, we use inequality \eqref{ineq_c=max_z1=z2} to find an upper bound for $c$, say $c_u$.
At the same time, an upper bound for $b\,(=\min\{a,b\})$ is found, say $b_u$, by the second inequality of Lemma \ref{z1=z2_first}, that is, $b_u:=\lfloor {c_u}^{z_1/4} \rfloor$.
This provides smaller uniform upper bounds for $b$ and $c$.
The same procedure can be continued by replacing $\mathcal M_b, \mathcal M_c$ by smaller ones until their decreases end.
Let LIST be the list of all tuples $(\beta,\alpha,z_1,b_u,c_u)$ found in the final iteration.

Third, for each tuple in LIST, we find all possible values of $b,y_1,y_2,a,x_2$ and $x_1$ in turn by using the following relations:
\begin{gather*}
b \le b_u, \ \ 4 \le y_1 \le  \left\lfloor \frac{\log c_u}{\log b}\,z_1 \right\rfloor, \ \ y_2 < y_1-1, \ \  a \mid (b^{D_y}-1), \\
1 < x_2 \le \min\biggl\{ \left\lfloor \frac{\log c_u}{\log a}\,z_1 \right\rfloor,y_1-1\biggl\}, \ x_1 < \min\{x_2,D_y\}.
\end{gather*}

Finally, for each of the found tuples $(a,b,x_1,y_1,x_2,y_2)$, we check whether system \eqref{ab12} holds or not, with the consideration to divisibility in \eqref{cond_z1=z2}.
As a result, the only remaining tuple is $(a,b,x_1,y_1,x_2,y_2)=(13,3,1,7,3,1)$ with $c=2200$, where a short modular arithmetic shows that there is no other triple $(x,y,z)$ satisfying $13^x+3^y=2200^z$.
The proof is completed.
\end{proof}

\begin{rem}\rm
The information on $\beta$ and $z_1$ can be also used to check whether system \eqref{ab12} holds or not.
\end{rem}

In what follows, we keep the notation in the proof of Proposition \ref{c=max_imply_z1<z2} and set $m_2=\frac{\log 8}{\log 3}$ uniformly.

For dealing with the case where $c<\max\{a,b\}$, we show two lemmas.

\begin{lem}\label{z1=z2 then dx>0 and dy>0}%182+748= 930sec
If $d_z=0,$ then $d_x>0$ and $d_y>0.$
\end{lem}

\begin{proof}
By symmetry of $a$ and $b$, it suffices to show that $d_x>0$.
Suppose on the contrary that $z_1=z_2$ and $x_i=x_j \, (<x_k)$.
Then $\{i,j\} \ni 3$.
Let $(I,J)$ be the permutation of $\{i,j\}$ such that $J=3$.
Note that $\{I,k\}=\{1,2\}, z_I=z_k=z_1$ and $x_I=x_J=x_3$.

Since $z_I=z_k$ and $x_I=x_j<x_k$, it follows that $y_k<y_I$.
Also, observe that $c^{z_I}-b^{y_I}=a^{x_I}=a^{x_J}=c^{z_J}-b^{y_J}$ and $z_I=z_2 \le z_3=z_J$.
Thus, $y_I \le y_J$.
To sum up, $y_k<y_I \le y_J$, so $d_y=y_I-y_k>0$ with $g_y=\gcd (x_I,z_1)$.
Now Lemma \ref{Lem_ineq_xi<xj}\,(i) with the base $b$ gives
\[
b^{d_y}<g_{y} \cdot \frac{\log c}{\log a} \cdot z_I z_J=g_{y} \cdot z_1 \cdot \frac{\log c}{\log a}\,z_3.
\]
By \eqref{z3},
\begin{align}\label{ineq_c<max_z1=z2_xi=xj}
b^{d_y}<g_y \cdot z_1\cdot \h(c;b,c).
\end{align}
Similarly to the use of \eqref{ineq_c=max_z1=z2}, we use inequality \eqref{ineq_c<max_z1=z2_xi=xj} to find all possible values of the letters in system \eqref{ab12}, and we sieve them.
We proceed in two cases according to whether $a>\max\{b,c\}$ or $b>\max\{a,c\}$.
Note that the conditions in \eqref{cond_z1=z2} correspond to
\begin{gather*}
a^{x_I} \mid (b^{d_y}-1), \ b^{y_k} \mid (a^{D_x}-1), \\
\max\{x_k,y_I\} \ge 4, \quad
\begin{cases}
\,y_I>x_k, & \text{if $a>b$,} \\ \,x_k>y_I, & \text{if $b>a$,}
\end{cases}
\quad x_I/x_k+y_k/y_I<1
\end{gather*}
with $D_x=x_k-x_I$.

\vspace{0.3cm}{\it Case where $a>\max\{b,c\}$.}

\vspace{0.2cm}
Since $a^{x_I}<b^{d_y}$ and $g_y \le x_I$, we see from \eqref{ineq_c<max_z1=z2_xi=xj} that $a \le a^{x_I}/x_I<b^{d_y}/g_y <z_1  \h(c;b,c)<z_1 \h(a)$, so
\begin{equation}\label{firstbounds_dx=0_c<max}
a<z_1 \h(a).
\end{equation}
Also, since $c<a<b^{d_y/x_I}$,
\begin{align}\label{ineq_a=max}
b^{d_y}<g_y \cdot z_1\cdot d_y/x_I \cdot \h\big(b^{d_y/x_I};b,b\big).
\end{align}

First, we use inequality \eqref{firstbounds_dx=0_c<max} with $\beta=1$ to find that $a<5.3 \cdot 10^{22}$.
Thus we set $\mathcal M_a:=5.3 \cdot 10^{22}$.
Note that
\[
x_I \le z_1-2, \quad x_I<d_y<\left\lfloor \frac{\log \mathcal M_c}{\log b_0}\,z_1 \right\rfloor.
\]
Next, for each of possible tuples $(\beta,\alpha,x_I,d_y)$, we use inequality \eqref{ineq_a=max} to find an upper bound for $b$, say $b_u$, thereby an upper bound for $a$ is also obtained, say $a_u$, from the divisibility $a^{x_I} \mid (b^{d_y}-1)$.
At the same time, we find an upper bound for $c$ and another upper bound for $b$, say $c_u,{b_u}'$ respectively, by using the following inequalities from Lemma \ref{z1=z2_first}:
\[
c^{z_1/2}<z_1^2 \,\h(c;a_u,c), \quad b<{c_u}^{z_1/4}.
\]
Third, for each of the found tuples $(\beta,\alpha,z_1,x_I,d_y,a_u,\min\{b_u,{b_u}'\},c_u)$, we find all possible values of $b,a,y_I,y_k$ and $x_k$ in turn by using the following relations:
\begin{gather*}
b \le \min\{b_u,{b_u}'\}, \ a \le a_u, \ a^{x_I} \mid (b^{d_y}-1), \ \ d_y+2 \le y_I \le \left\lfloor \frac{\log c_u}{\log b}\,z_1 \right\rfloor, \\
y_k=y_I-d_y, \ \ x_I < x_k \le \min\biggl\{ \left\lfloor \frac{\log c_u}{\log a}\,z_1 \right\rfloor,y_I-1\biggl\}.
\end{gather*}
Finally, we verify that system \eqref{ab12} does not hold for any found tuple $(a,b,x_1,y_1,x_2,y_2)$.

\vspace{0.3cm}{\it Case where $b>\max\{a,c\}$.}

\vspace{0.2cm}
In case $(\dagger)$, we have $z_1=\alpha /\beta$, and $y_k<y_I<z_I=z_1 \le \alpha<2^\alpha-1 \le a$.
Thus, if $d_y=1$, by Lemma \ref{X=x+1}, we can use the following relations:
\begin{gather*}
b \equiv -y_k a^{2x_I}-a^{x_I}+1 \mod{a^{3x_I}}, \\
b \ge {\textstyle \frac{1}{2}a^{3x_I}+\frac{1}{2}a^{2x_I}}-a^{x_I}+1,\quad
x_k>(3x_I-1) y_I.
\end{gather*}

Since $c<b$, inequality \eqref{ineq_c<max_z1=z2_xi=xj} yields
\begin{align} \label{ineq_b=max_z1=z2_xi=xj}
b^{d_y}<\gcd(x_I,z_1) \cdot z_1 \cdot \h(b).
\end{align}
First, we use the inequality $b<z_1^2 \cdot \h(b)$ with $\beta=1$ to see that $b=\max\{a,b,c\}<3.4 \cdot 10^{15}$, and set $\mathcal M_b:=3.4 \cdot 10^{15}$.
Note that
\begin{gather*}
x_I \le x_k-2 \le \left\lfloor \frac{\log \mathcal M_c}{\log a_0}\,z_1 \right\rfloor-2, \quad d_y \le z_1-2.
\end{gather*}
Next, for each possible tuple $(\beta,\alpha,z_1,x_I,d_y)$, we use inequality \eqref{ineq_b=max_z1=z2_xi=xj} to find an upper bound for $b$, say $b_u$.
Similarly to the previous case, we use the inequalities $a<{b_u}^{d_y/x_I},\,c^{z_1/2}<z_1^2 \,\h(c;b_u,c)$ and $a<c^{z_1/4}$, to find upper bounds for $a$ and $c$, say $a_u,c_u$ respectively.
Finally, we verify that system \eqref{ab12} does not hold for any possible tuples $(a,b,x_1,y_1,x_2,y_2)$ coming from all possible tuples $(\beta,\alpha,z_1,x_I,d_y,a_u,b_u,c_u)$ found similarly to the case where $a>\max\{b,c\}$.

To sum up, the lemma is proved by Proposition \ref{c=max_imply_z1<z2}.
\end{proof}

\begin{lem}\label{k ne 3}% 116+12+110+15=253 sec
Suppose that $d_z=0.$ Then
\[ \begin{cases}
\,k=3, & \text{if $a>\max\{b,c\};$}\\
\,n=3, & \text{if $b>\max\{a,c\}.$}
\end{cases} \]
\end{lem}

\begin{proof}
By symmetry of $a$ and $b$, it suffices to consider the case where $a>b$.
Suppose on the contrary that $a>\max\{b,c\}$ and $k \ne 3$.
Then $\{i,j\} \ni 3$.
Let $(I,J)$ be the permutation of $\{i,j\}$ such that $J=3$.
Note that $\{I,k\}=\{1,2\}$ and $z_I=z_k=z_1$.

Firstly, let us observe that
\begin{equation}\label{U3}
z_3<\mathcal U_3:=\max \biggl \{ (1+\varepsilon)z_1 + \frac{(1+\varepsilon)\log a}{\log c} \cdot d_x + 1, \,2523 \log b \biggl \}
\end{equation}
with $\varepsilon=999$.
First, suppose that $a^{(1+\varepsilon)x_J}>b^{y_J}$.
Then $c^{z_J}=a^{x_J}+b^{y_J}<2a^{(1+\varepsilon)x_J}$, and so
\[
z_J< 1+\frac{(1+\varepsilon)\log a}{\log c}\,x_J.
\]
Since $x_J=x_I-(x_I-x_J)$ and $x_I<\frac{\log c}{\log a}\,z_I$, it follows that
\[
z_3 < %1+(1+\varepsilon)z_I-\frac{(1+\varepsilon)\log a}{\log c}\,(x_I-x_J) \le
 1+(1+\varepsilon)z_1+\frac{(1+\varepsilon)\log a}{\log c}\,|x_I-x_J|.
\]
Next, suppose that $a^{(1+\varepsilon)x_J}<b^{y_J}$.
From $J$-th equation,
\[
\frac{c^{z_J}}{b^{y_J}}=\frac{a^{x_J}}{b^{y_J}}+1 <\frac{1}{b^{\frac{\varepsilon}{1+\varepsilon}y_J}}+1.
\]
Put $\lambda:=z_J\log c- y_J \log b\,(>0)$. We take the logarithms of these to find that
\[
\log \lambda < -\frac{\varepsilon}{1+\varepsilon}\,y_J \log b.
\]
On the other hand, from \cite[Corollary 2;$(m,C_2)=(10,25.2)$]{La},
\[
\log \lambda > -25.2 \,\Bigl(  \max \Bigl \{  \log \Bigr(  \frac{z_J}{\log b}+\frac{y_J}{\log c}  \Bigr) +0.38, 10\Bigl\} \Bigl)^2 (\log b)\log c.
\]
These inequalities together yield
\[
\frac{y_J}{\log c} <25.2\,(1+1/\varepsilon)\,\Bigr( \max\Bigr\{ \log \Bigr( \frac{2y_J}{\log c}+1 \Bigr) +0.38, 10 \Bigr\}  \Bigr)^2.
\]
This implies that $y_J/\log c<2520\,(1+1/\varepsilon)$.
Inequality \eqref{U3} for this case follows from the fact that $\lambda$ is small.

Secondly, in several cases according to the value of $g_x$, we apply Lemma \ref{Lem_ineq_xi<xj} together with inequality \eqref{U3} in \eqref{ab12} to find all possible values of the letters and sieve them.
We proceed basically along similar lines to the proof of the previous lemma.

\vspace{0.3cm}{\it Case where $g_x=1$.}

\vspace{0.2cm}
Lemma \ref{Lem_ineq_xi<xj}\,(i) gives
\[
a^{d_x}<\frac{\log c}{\log b} \cdot z_j z_k<z_1\,{\mathcal U_3}', \quad a^{d_x}<z_1\h(a).
\]
where ${\mathcal U_3}'=\max\Bigl \{ \frac{(1+\varepsilon)z_1\log c}{\log b_0} + \frac{(1+\varepsilon)\log a}{\log b_0} \cdot d_x + \log c, \,2523\log c\Bigl \}$.
The above second inequality gives a smaller bound for $a$, that is, we can set $\m_a:=2.7 \cdot 10^{11}$.
Note that $d_x \le z_1-2$.
For each possible tuple $(\beta,\alpha,d_x,z_1)$, we use the first inequality above to find an upper bound for $a$, say $a_u$.
Also, upper bounds for $c$ and $b$, say $c_u,b_u$ respectively, are found by using the following inequalities from Lemma \ref{z1=z2_first}:
\[
c^{z_1/2}<z_1^2\,{\mathcal U_3}', \quad b<{c_u}^{z_1/4}.
\]
Finally, we check that system \eqref{ab12} does not hold for any tuple $(a,b,x_1,y_1,x_2,y_2)$ coming from all possible tuples $(\beta,\alpha,z_1,a_u,b_u,c_u)$.

\vspace{0.3cm}{\it Case where $g_x \in \{2,5\}$.}

\vspace{0.2cm}
Since $\gcd(g_{x}',a)=1$ by Lemma \ref{gx}\,(ii), and $x_j \le x_k<\frac{\log c}{\log a}z_k$, Lemma \ref{Lem_ineq_xi<xj}\,(ii) gives
\[
a^{d_x}<\frac{\log a}{\log b} \cdot x_j z_k \le \frac{\log c}{\log b} \cdot z_k^2 < \frac{\log a}{\log b_0} \cdot z_1^2.
\]
For each possible tuple $(\beta,\alpha,d_x,z_1)$, we use the above inequality to find an upper bound for $a$.
The remaining part is handled similarly to the previous case.

\vspace{0.3cm}{\it Case where $g_x=3$.}

\vspace{0.2cm}
Since $2 \le x_j<\min\{z_j,z_k\} \le z_2$, it follows from Lemma \ref{firstbounds} and \ref{gx}\,(i) that $x_j=2$.
Thus, $x_i=1,d_x=1$, and so Lemma \ref{Lem_ineq_xi<xj}\,(ii-1) gives
\[
a<\frac{3^2}{3-1} \cdot \frac{\log a}{\log b} \cdot 2 \cdot z_k=\frac{9\log a}{\log b_0} \cdot z_1.
\]
For each possible pair $(\beta,\alpha,z_1)$, we use this inequality to find an upper bound for $a$.
The remaining part is handled similarly to the previous cases.

\vspace{0.3cm}{\it Case where $g_x \not\in \{1,2,3,5\}$.}

\vspace{0.2cm}
Note that $g_x \ge 7$ as $4 \nmid g_x$ and $6 \nmid g_x$ by Lemma \ref{gx}\,(ii).
Lemma \ref{Lem_ineq_xi<xj}\,(ii-2) gives
\begin{align*}
a^{d_x}
&< \frac{(7/6)^2\log^2 a}{\log (a-1)\,\log b} \cdot t_{b,c} \cdot \left( x_j+d_x+{d_x}^2 \right) z_k\\
& \le \frac{(49/36)\log a}{\log (a-1)} \cdot \frac{\log a}{\log \max\{b_0,c_0\}} \cdot \left(z_1-1+d_x+{d_x}^2\right)\,z_1.
\end{align*}
The remaining part is handled similarly to the previous case by using the above inequality.
\end{proof}

\begin{prop}\label{c<max_implies_z1 ne z2}%935+169+198=1302 sec
If $c<\max\{a,b\},$ then $d_z>0.$
\end{prop}

\begin{proof}
It suffices to consider the case where $a>b$.
Suppose on the contrary that $a>\max\{b,c\}$ and $z_1=z_2$.
Then $k=3$ and $d_x=|x_2-x_1|>0$ by the combination of Lemmas \ref{z1=z2 then dx>0 and dy>0} and \ref{k ne 3}.
Since the argument is almost similar to that of Lemma \ref{k ne 3}, we omit the details in the following.
The main difference is that $\mathcal U_3$ is replaced by $\h(c;a,b)$ from \eqref{z3}.

\vspace{0.3cm}{\it Case where $g_x \in \{1,2,5\}$.}

\vspace{0.2cm}
Lemma \ref{Lem_ineq_xi<xj}\,(i) together with Lemma \ref{gx}\,(ii) gives
\[
a^{d_x}<\frac{\log c}{\log b} \cdot z_1 z_3 < z_1\h(c;a,c).
\]
Note that $a^{d_x}<z_1\h(a)$ and this implies small upper bounds for $a$ and $d_x$.
The remaining part is handled almost similarly to the proof of the previous lemma.
However, remark that we can efficiently use the additional condition that $D_x=d_x$ and $\gcd(y_j,z_1) \in \{1,2,5\}$ with some $j \in \{1,2\}$ in checking system \eqref{ab12}.

\vspace{0.3cm}{\it Case where $g_x=3$.}

\vspace{0.2cm}
This case is also handled almost similarly to the previous lemma, in particular, $d_x=1$, and it is easy to verify that Lemma \ref{X=x+1} can be used so that the additional condition that $a \equiv -x_i b^{2y_j}-b^{y_j}+1 \pmod{b^{3y_j}}$ and $y_i>(3y_j-1) x_j$ can be efficiently used.

\vspace{0.3cm}{\it Case where $g_x \not\in \{1,2,3,5\}$.}

\vspace{0.2cm}
First, Lemma \ref{Lem_ineq_xi<xj}\,(i) gives
\[
a^{d_x} <\frac{\log c}{\log b} \cdot z_1 z_3^2<z_1 \h(a)^2.
\]
This implies small upper bounds for $a$ and $d_x$.
Next, together with $g_x \ge 7$, Lemma \ref{Lem_ineq_xi<xj}\,(ii-2) implies
\begin{align*}
a^{d_x}
%&< \frac{49/36\log^2 a}{\log (a-1)\,\log b} \cdot t_{b,c} \cdot \left( \frac{\log c}{\log a}z_j+d_x+{d_x}^2 \right) z_k\\
%& \le \frac{49/36\log a}{\log (a-1)} \cdot \frac{\log a}{\log \max\{b,c\}} \cdot \left( \frac{\log c}{\log a}\,z_1+d_x+{d_x}^2 \right)\,z_3\\
%&
<\frac{49}{36} \cdot \left( \frac{\log c}{\log a}\,z_1+d_x+{d_x}^2 \right)\,\h(c;a,a).
\end{align*}
The remaining part is handled similarly to the previous lemma, where the additional condition $\gcd(y_j,z_1) \ge 7$ is efficiently used.
\end{proof}

In view of Propositions \ref{c=max_imply_z1<z2} and \ref{c<max_implies_z1 ne z2}, the conclusion of this section is:

\begin{prop}\label{z1 ne z2}
$z_1 \ne z_2.$
\end{prop}

The total computational time for this proposition did not exceed 1 hour.

By Proposition \ref{z1 ne z2}, it remains to consider the case where $z_1<z_2$, where the upper bound for $\max\{a,b,c\}$ in \eqref{cond} can be replaced by $5 \cdot 10^{27}$ by \cite{HuLe_jnt18}.

%%%%%%%%%%%%%%%%%%%%%%%%%%
\section{Case where $z_1<z_2$ with $c>\max\{a,b\}$: finding bounds} \label{Sec-bounds-z1<z2_c=max}
%%%%%%%%%%%%%%%%%%%%%%%%%%

The aim of this section is to provide a list of possible values or upper bounds of some letters in system \eqref{sys-12} satisfying $z_1<z_2$ with $c>\max\{a,b\}$.
We proceed in two cases according to whether $c^{z_1}$ is divisible by $4$ or not.

Let us begin with the following lemma to give an upper bound for $g_2$ in terms of $z_2$.

\begin{lem}\label{U_g2}
$g_2<\frac{\log c}{\log \max\{a,b\}}\,z_2$.
\end{lem}

\begin{proof}
From 2nd equation, $\min\{x_2,y_2\}<\frac{\log c}{\log \max\{a,b\}} z_2$.
On the other hand, $g_2 \mid \min\{x_2,y_2\}$ by the definition of $g_2$.
These relations together readily yield the assertion.
\end{proof}

In what follows, for any numbers $P_1,P_2,\ldots,P_k$ and $Q_1,Q_2,\ldots,Q_k$, the notation $[P_1,P_2,\ldots,P_k] \le [Q_1,Q_2,\ldots,Q_k]$ means that $P_i \le Q_i$ for any $i$.

\begin{prop}\label{bounds_z1<z2_c^z1=2(4)}
Suppose that
\[
d_z>0, \quad c^{z_1} \equiv 2 \mod{4}, \quad c>\max\{a,b\}.
\]
Then $\beta=1, z_1=1,$ and the following hold.
\begin{itemize}
\item[\rm (i)]
Suppose that $g_2=1.$ Then
\[
\alpha \le 17, \quad z_2 \le 18, \quad c<2.2 \cdot 10^6.
\]
More exactly, one of the following cases holds.
\begin{itemize}
\item[$\bullet$] $z_2 \le 17, \ c<8.7 \cdot 10^5, \ (x_1,y_1)=(1,1), \ \min\{a,b\} \le 7;$
\item[$\bullet$] $[\alpha,z_2] \le [17,18], \ c<1.1\cdot10^6, \ (x_1,y_1)=(1,1), \ \min\{a,b\}>7;$
\item[$\bullet$] $z_2 \le 16, \ c<2.2\cdot10^6, \ (x_1,y_1)\ne(1,1), \ \min\{a,b\} \le 7;$
\item[$\bullet$] $[\alpha,z_2] \le [9,16], \ c<1.1\cdot10^6, \ (x_1,y_1)\ne(1,1), \ \min\{a,b\}>7.$
\end{itemize}
\if0
\begin{align*}
z_2 \le 17, \ c<8.7 \cdot 10^5, \ (x_1,y_1)=(1,1), \ \min\{a,b\} \le 7;\\
[\alpha,z_2] \le [17,18], \ c<1.1\cdot10^6, \ (x_1,y_1)=(1,1), \ \min\{a,b\}>7;\\
z_2 \le 16, \ c<2.2\cdot10^6, \ (x_1,y_1)\ne(1,1), \ \min\{a,b\} \le 7;\\
[\alpha,z_2] \le [9,16], \ c<1.1\cdot10^6, \ (x_1,y_1)\ne(1,1), \ \min\{a,b\}>7.
\end{align*}
\fi
\item[\rm (ii)]
Suppose that $g_2>1.$ Then
\[
\alpha \le 22, \quad z_2 \le 23, \quad c<2.2 \cdot 10^7.
\]
More exactly, one of the following cases holds.
\begin{itemize}
\item[$\bullet$] $z_2=2, \ c<2 \cdot 10^5, \ g_2=3, \ (x_1,y_1)=(1,1), \ \min\{a,b\} \le 7;$
\item[$\bullet$] $(\alpha,z_2)=(2,2), \ c<1.5 \cdot 10^6, \ g_2=3, \ (x_1,y_1)=(1,1), \ \min\{a,b\}>7;$
\item[$\bullet$] $z_2=2, \ c<1.3 \cdot 10^6, \ g_2=3, \ (x_1,y_1)\ne(1,1), \ \min\{a,b\} \le 7;$
%\alpha=2, \ z_2=2, \ c<1.5 \cdot 10^6, \ g_2=3, \ (x_1,y_1)\ne(1,1), \ \min\{a,b\}>7;$
\item[$\bullet$] $z_2 \le 13, \ c<1600, \ g_2=5, \ (x_1,y_1)=(1,1), \ \min\{a,b\} \le 7;$
\item[$\bullet$] $[\alpha,z_2] \le [13,13], \ c<75000, \ g_2=5, \ (x_1,y_1)=(1,1), \ \min\{a,b\}>7;$
\item[$\bullet$] $z_2 \le 13, \ c<1600, \ g_2=5, \ (x_1,y_1)\ne(1,1), \ \min\{a,b\} \le 7;$
\item[$\bullet$] $[\alpha,z_2] \le [6,13], \ c<4400, \ g_2=5, \ (x_1,y_1)\ne(1,1), \ \min\{a,b\}>7;$
\item[$\bullet$] $z_2 \le 19, \ c<1.5\cdot10^5, \ g_2=7, \ (x_1,y_1)=(1,1), \ \min\{a,b\} \le 7;$
\item[$\bullet$] $[\alpha,z_2] \le [19,19], \ c<5\cdot10^6, \ g_2=7, \ (x_1,y_1)=(1,1), \ \min\{a,b\}>7;$
\item[$\bullet$] $z_2 \le 19, \ c<1.5\cdot10^5, \ g_2=7, \ (x_1,y_1)\ne(1,1), \ \min\{a,b\} \le 7;$
\item[$\bullet$] $[\alpha,z_2] \le [11,19], \ c<5\cdot10^6, \ g_2=7, \ (x_1,y_1)\ne(1,1), \ \min\{a,b\}>7;$
\item[$\bullet$] $z_2 \le 19, \ c<2.4\cdot10^6, \ g_2 \ge 11, \ (x_1,y_1)=(1,1), \ \min\{a,b\} \le 7;$
\item[$\bullet$] $[\alpha,z_2] \le [22,23], \ c<1.9\cdot10^7, \ g_2 \ge 11, \ (x_1,y_1)=(1,1), \ \min\{a,b\}>7;$
\item[$\bullet$] $z_2 \le 19, \ c<2.2\cdot10^7, \ g_2 \ge 11, \ (x_1,y_1)\ne(1,1), \ \min\{a,b\} \le 7;$
\item[$\bullet$] $[\alpha,z_2] \le [11,19], \ c<1.2\cdot10^7, \ g_2 \ge 11, \ (x_1,y_1)\ne(1,1), \ \min\{a,b\}>7.$
\end{itemize}
\if0
\begin{align*}
z_2=2, \ c<2 \cdot 10^5, \ g_2=3, \ (x_1,y_1)=(1,1), \ \min\{a,b\} \le 7;\\
(\alpha,z_2)=(2,2), \ c<1.5 \cdot 10^6, \ g_2=3, \ (x_1,y_1)=(1,1), \ \min\{a,b\}>7;\\
z_2=2, \ c<1.3 \cdot 10^6, \ g_2=3, \ (x_1,y_1)\ne(1,1), \ \min\{a,b\} \le 7;\\
%\alpha=2, \ z_2=2, \ c<1.5 \cdot 10^6, \ g_2=3, \ (x_1,y_1)\ne(1,1), \ \min\{a,b\}>7;\\
z_2 \le 13, \ c<1600, \ g_2=5, \ (x_1,y_1)=(1,1), \ \min\{a,b\} \le 7;\\
[\alpha,z_2] \le [13,13], \ c<75000, \ g_2=5, \ (x_1,y_1)=(1,1), \ \min\{a,b\}>7;\\
z_2 \le 13, \ c<1600, \ g_2=5, \ (x_1,y_1)\ne(1,1), \ \min\{a,b\} \le 7;\\
[\alpha,z_2] \le [6,13], \ c<4400, \ g_2=5, \ (x_1,y_1)\ne(1,1), \ \min\{a,b\}>7;\\
z_2 \le 19, \ c<1.5\cdot10^5, \ g_2=7, \ (x_1,y_1)=(1,1), \ \min\{a,b\} \le 7;\\
[\alpha,z_2] \le [19,19], \ c<5\cdot10^6, \ g_2=7, \ (x_1,y_1)=(1,1), \ \min\{a,b\}>7;\\
z_2 \le 19, \ c<1.5\cdot10^5, \ g_2=7, \ (x_1,y_1)\ne(1,1), \ \min\{a,b\} \le 7;\\
[\alpha,z_2] \le [11,19], \ c<5\cdot10^6, \ g_2=7, \ (x_1,y_1)\ne(1,1), \ \min\{a,b\}>7;\\
z_2 \le 19, \ c<2.4\cdot10^6, \ g_2 \ge 11, \ (x_1,y_1)=(1,1), \ \min\{a,b\} \le 7;\\
[\alpha,z_2] \le [22,23], \ c<1.9\cdot10^7, \ g_2 \ge 11, \ (x_1,y_1)=(1,1), \ \min\{a,b\}>7;\\
z_2 \le 19, \ c<2.2\cdot10^7, \ g_2 \ge 11, \ (x_1,y_1)\ne(1,1), \ \min\{a,b\} \le 7;\\
[\alpha,z_2] \le [11,19], \ c<1.2\cdot10^7, \ g_2 \ge 11, \ (x_1,y_1)\ne(1,1), \ \min\{a,b\}>7.
\end{align*}
\fi
\end{itemize}
\end{prop}

\begin{proof}
Here we just indicate how we find a list of all possible pairs $(\alpha,z_2)$ with the corresponding upper bound for $c$.
It suffices to consider the case where $a>b$.

From $c=a^{x_1}+b^{y_1}$ by 1st equation, observe that $c \ge \max\{18,2^{\alpha+1}+2^{\alpha}-2\}$, and
\begin{equation*}\label{bounds_z1<z2__c^z1=2(4)_a0c0}
\begin{cases}
\,a=\max\{a,b\} \ge c/2+2, & \text{if $(x_1,y_1)=(1,1);$}\\
\,c \ge ( 2^\alpha-1 )^2 + (2^\alpha+1)=2^{2\alpha}-2^{\alpha}+2, & \text{if $(x_1,y_1) \ne (1,1).$}
\end{cases}
\end{equation*}
This affects the choice of the values of $a_0$ and $c_0$.
By Lemma \ref{Lem_ineq_z1<z2_c^{z1}=2(4)},
\begin{equation} \label{ineq_z1<z2_c^z1=2(4)}
c<\min\!\left\{ 2^{\alpha+1-z_2} \frac{({g_2}')^2}{g_2},\, \frac{({g_2}')^2}{g_2},\, \frac{T\log c}{\log (c-1)}\,z_2\right\}\cdot z_2\,\mathcal H_{\alpha,1,m_2}(c)
\end{equation}
where ${g_2}'=\gcd(c,g_2)$, and
\[ T= \begin{cases}
\,1, & \text{if $b>7$,}\\
\,\frac{\log b}{\log a_0}, & \text{if $b \in \{3,5,7\}$}.
\end{cases} \]
Note that $g_2$ is odd by Lemma \ref{g2}\,(i).

Similarly to Section \ref{Sec-z1=z2}, firstly setting $\m_c=5 \cdot 10^{27}$, we use inequality \eqref{ineq_z1<z2_c^z1=2(4)} to find an upper bound for $c$ for each possible pair $(\alpha,z_2)$ satisfying \eqref{beta-alp_region} and $\alpha \le z_2 \le \mathcal U_2(\alpha,1,m_2,\m_c,g_2)$ with $\m_c=5 \cdot 10^{27}$ by \eqref{z1}, where each of the procedures is implemented in two cases according to whether $b>7$ or not, and to $(x_1,y_1)=(1,1)$ or not.
Moreover, we proceed in several cases according to the value of $g_2$ as below, where we only indicate the required inequality from \eqref{ineq_z1<z2_c^z1=2(4)} with additional remarks.

\vspace{0.2cm}{\it Case where $g_2=1$.}

\vspace{0.2cm}
Inequality \eqref{ineq_z1<z2_c^z1=2(4)} is
\[
c<\min\!\left\{ 2^{\alpha+1-z_2},\, 1,\, \frac{T z_2\log c}{\log (c-1)} \right\} \cdot z_2\,\mathcal H_{\alpha,1,m_2}(c).
\]

\vspace{0.2cm}{\it Case where $g_2>1$.}

\vspace{0.2cm}Lemma \ref{g2}\,(iv,v) tells us that $\gcd(z_2,6)=1$ if $g_2>3$.
Moreover, $z_2 \not\equiv 0 \pmod{g_2}$ by Lemma \ref{ourcase}.
We proceed in several subcases.

\vspace{0.2cm}{\it (i) Case where $g_2 \equiv 0 \pmod{3}$.}

\vspace{0.2cm}
Lemma \ref{g2}\,(ii,iv) tells us that $g_2=3, z_2=2$, and so $\alpha=2$.
Inequality \eqref{ineq_z1<z2_c^z1=2(4)} is
\[
c<2\min\!\left\{3,\, \frac{2T\log c}{\log (c-1)}\right\}\cdot \mathcal H_{2,1,m_2}(c).
\]

\vspace{0.2cm}{\it (ii) Case where $g_2=5$.}

\vspace{0.2cm}
By Lemma \ref{g2}\,(iii), ${g_2}'=1$, and also $\gcd(z_2,7)=1$ by Lemma \ref{ourcase}.
Inequality \eqref{ineq_z1<z2_c^z1=2(4)} is
\[
c<\frac{1}{5}\, \min\!\left\{ 2^{\alpha+1-z_2},\,  \frac{Tz_2\log c}{\log (c-1)} \right\}\cdot z_2\,\mathcal H_{\alpha,1,m_2}(c).
\]

\vspace{0.2cm}{\it (iii) Case where $g_2=7$.}

\vspace{0.2cm}
Inequality \eqref{ineq_z1<z2_c^z1=2(4)} is used with $g_2,{g_2}'$ replaced by $7$.
Note that $\gcd(z_2,5)=1$.

\vspace{0.2cm}{\it (iv) Case where $g_2>7$.}

\vspace{0.2cm}
Note that $g_2 \ge 11$.
From Lemma \ref{U_g2}, inequality \eqref{ineq_z1<z2_c^z1=2(4)} yields
\[
c<\min\!\left\{ \frac{2^{\alpha+1-z_2} \log c}{\log a_0},\, \frac{\log c}{\log a_0},\, \frac{T\log c}{\log (c-1)}\right\}\cdot z_2^2\,\mathcal H_{\alpha,1,m_2}(c).
\]

By these observations, we find a list of finitely many possible pairs $(\alpha,z_2)$ with the corresponding upper bound for $c$, and those satisfy the stated conditions.
\end{proof}

The following lemma is a supplement to Proposition \ref{bounds_z1<z2_c^z1=2(4)}, which is efficient to reduce the computational time to sieve the given cases with $(x_1,y_1)=(1,1)$ (see Section \ref{Sec-sieve-z1<z2_c=max_c^z1=2(4)}).

\begin{lem}\label{bounds_z1<z2_c^z1=2(4)_x1=y1=1}
Under the hypothesis of Proposition \ref{bounds_z1<z2_c^z1=2(4)}, assume that $\min\{a,b\}>7$ and $(x_1,y_1)=(1,1).$
Then $\min\{x_2,y_2\} \le 7,$ and the following holds.
\[
\min\{a,b\} \le \begin{cases}
\, 13, & \text{if $\min\{x_2,y_2\}=7$},\\
\, 21, & \text{if $\min\{x_2,y_2\}=6$},\\
\, 45, & \text{if $\min\{x_2,y_2\}=5$},\\
\, 181, & \text{if $\min\{x_2,y_2\}=4$}.
\end{cases}
\]
\end{lem}

\begin{proof}
Since $x_2 \le x_3$ or $y_2 \le y_3$, one of the following cases holds.
\begin{alignat*}{2}
&(i,j,k)=(1,2,3),& \ \ & d_x=x_2-x_1 \ge \min\{x_2,y_2\}-1;\\
&(l,m,n)=(1,2,3), & & d_y=y_2-y_1 \ge \min\{x_2,y_2\}-1.
\end{alignat*}
Let us consider only the former case as the latter one is similarly handled.
From the assumption that $\min\{x_2,y_2\} \ge 4$, we have $d_x \ge 3$, and that $3 \nmid g_x$ by Lemmas \ref{firstbounds} and \ref{gx}\,(i).
Similarly to Proposition \ref{c<max_implies_z1 ne z2}, we apply Lemma \ref{Lem_ineq_xi<xj} together with the inequality $c>\max\{a,b\}$ to see that
\begin{align*}
a^{d_x}
%&<\frac{49/36 \log^2 a}{\log (a-1)\,\log c} \cdot (x_2+d_x+{d_x}^2)\,z_3\\
&<\frac{49}{36}\cdot \frac{\log a}{\log (a-1)} \cdot  \left( \frac{\log \m_c}{\log a}\,z_2+d_x+{d_x}^2 \right)\,\h(\m_c;a,a).%\\
%&<\frac{(7/6)^2 \log b}{\log (b-1)} \cdot \left( z_2+d_y+{d_y}^2 \right)\,\h(b).
\end{align*}
Note that $m_2=1$ since $\min\{a,b\}>7$, and we can set $\m_c:=2.2 \cdot 10^7$ by Proposition \ref{bounds_z1<z2_c^z1=2(4)}.
Finally, for each $d_x \ge 3$, similarly to Proposition \ref{bounds_z1<z2_c^z1=2(4)}, we use the above inequality to find an upper bound for $a$ for each possible pair $(\alpha,z_2)$.
The result implies the assertion.
\end{proof}

\begin{prop}\label{bounds_z1<z2_c^z1=0(4)}
Suppose that
\[
d_z>0, \quad c^{z_1} \equiv 0 \mod{4}, \quad c>\max\{a,b\}.
\]
\begin{itemize}
\item[\rm (i)] Suppose that $g_2=1.$ Then
\[
[\beta,\alpha,z_2,d_z] \le [10,18,19,4], \quad c<1.5 \cdot 10^6.
\]
More exactly, one of the following cases holds.
\begin{itemize}
\item[$\bullet$] $[\beta,\alpha,z_2] \le [1,2,6], \ c<20, \ d_z=4;$
\item[$\bullet$] $[\beta,\alpha,z_2] \le [3,5,10], \ c<70, \ d_z=3;$
\item[$\bullet$] $[\beta,\alpha,z_2] \le [5,8,14], \ c<750, \ d_z=2;$
\item[$\bullet$] $[\beta,z_2] \le [8,17], \ c<1.5 \cdot 10^6,\ \min\{a,b\} \le 7, \ d_z=1;$
\item[$\bullet$] $[\beta,\alpha,z_2] \le [10,18,19], \ c<1.1 \cdot 10^6, \ \min\{a,b\}>7, \ d_z=1.$
\end{itemize}
\if0
\begin{align*}
&[\beta,\alpha,z_2] \le [1,2,6], \ c<20, \ d_z=4;\\
&[\beta,\alpha,z_2] \le [3,5,10], \ c<70, \ d_z=3;\\
&[\beta,\alpha,z_2] \le [5,8,14], \ c<750, \ d_z=2;\\
&[\beta,z_2] \le [8,17], \ c<1.5 \cdot 10^6,\ \min\{a,b\} \le 7, \ d_z=1;\\
&[\beta,\alpha,z_2] \le [10,18,19], \ c<1.1 \cdot 10^6, \ \min\{a,b\}>7, \ d_z=1.
\end{align*}
\fi
\item[\rm (ii)] Suppose that $g_2>1.$
Then
\[
[\beta,\alpha,z_2,d_z] \le [10,19,23,4], \quad c<3.4\cdot 10^{6}.
\]
More exactly, one of the following cases holds.
\begin{itemize}
\item[$\bullet$] $[\beta,\alpha,z_2] \le [2,3,11], \ c<30, \ g_2 \ge 5, \ d_z=4;$
\item[$\bullet$] $[\beta,\alpha,z_2] \le [3,5,11], \ c<80, \ g_2 \ge 5, \ d_z=3;$
\item[$\bullet$] $[\beta,\alpha,z_2] \le [3,8,13], \ c<910, \ g_2 \ge 5, \ d_z=2;$
\item[$\bullet$] $\beta \le 9, \ c<2.4 \cdot 10^5, \ \min\{a,b\} \le 7,  \ (z_2, z_1, g_2)=(2,1,3);$
\item[$\bullet$] $[\beta,\alpha] \le [10,11], \ c<2.9 \cdot 10^5, \ \min\{a,b\}>7, \ (z_2, z_1, g_2)=(2,1,3);$
\item[$\bullet$] $[\beta,z_2]=[1,13], \ c<1600, \ \min\{a,b\} \le 7, \ (g_2,d_z)=(5,1);$
\item[$\bullet$] $[\beta,\alpha,z_2] \le [1,16,17], \ c<8\cdot10^4, \ \min\{a,b\}>7, \ (g_2,d_z)=(5,1);$
\item[$\bullet$] $[\beta,z_2]=[1,19], \ c<1.5\cdot10^5, \ \min\{a,b\} \le 7, \ (g_2,d_z)=(7,1);$
\item[$\bullet$] $[\beta,\alpha,z_2] \le [1,18,19], \ c<7.5\cdot10^5, \ \min\{a,b\}>7, \ (g_2,d_z)=(7,1);$
\item[$\bullet$] $[\beta,z_2] \le [3,19], \ c<3\cdot10^6, \ \min\{a,b\} \le 7, \ g_2 \ge 11, \ d_z=1;$
\item[$\bullet$] $[\beta,\alpha,z_2] \le [4,19,23], \ c<3.4\cdot10^6, \ \min\{a,b\}>7, \ g_2 \ge 11, \ d_z=1.$
\end{itemize}
\if0
\begin{align*}
&[\beta,\alpha,z_2] \le [2,3,11], \ c<30, \ g_2 \ge 5, \ d_z=4;\\
&[\beta,\alpha,z_2] \le [3,5,11], \ c<80, \ g_2 \ge 5, \ d_z=3;\\
&[\beta,\alpha,z_2] \le [3,8,13], \ c<910, \ g_2 \ge 5, \ d_z=2;\\
&\beta \le 9, \ c<2.4 \cdot 10^5, \ \min\{a,b\} \le 7,  \ (z_2, z_1, g_2)=(2,1,3);\\
&[\beta,\alpha] \le [10,11], \ c<2.9 \cdot 10^5, \ \min\{a,b\}>7, \ (z_2, z_1, g_2)=(2,1,3);\\
&[\beta,z_2]=[1,13], \ c<1600, \ \min\{a,b\} \le 7, \ g_2=5, \ d_z=1;\\
&[\beta,\alpha,z_2] \le [1,16,17], \ c<8\cdot10^4, \ \min\{a,b\}>7, \ g_2=5, \ d_z=1;\\
&[\beta,z_2]=[1,19], \ c<1.5\cdot10^5, \ \min\{a,b\} \le 7, \ g_2=7, \ d_z=1;\\
&[\beta,\alpha,z_2] \le [1,18,19], \ c<7.5\cdot10^5, \ \min\{a,b\}>7, \ g_2=7, \ d_z=1;\\
&[\beta,z_2] \le [3,19], \ c<3\cdot10^6, \ \min\{a,b\} \le 7, \ g_2 \ge 11, \ d_z=1;\\
&[\beta,\alpha,z_2] \le [4,19,23], \ c<3.4\cdot10^6, \ \min\{a,b\}>7, \ g_2 \ge 11, \ d_z=1.
\end{align*}
\fi
\end{itemize}
\end{prop}

\begin{proof}
Here we just indicate how we find a list of all possible triples $(\beta,\alpha,z_2)$ with the corresponding upper bound for $c$.
It suffices to consider the case where $a>b$.
We proceed along similar lines to that of Proposition \ref{bounds_z1<z2_c^z1=2(4)}.

By Lemma \ref{Lem_ineq_z1<z2_c^z1=0(4)},
\begin{equation}\label{ineq_z1<z2_c^z1=0(4)}
c^{d_z}< \min \! \left \{ 2^{\alpha -\beta z_1} \frac{({g_2}')^2}{g_2}, \, \frac{T\log c}{\log (c-1)}\,\frac{z_2}{z_1} \right \}\cdot \,z_2\h(c),
\end{equation}
where ${g_2}'=\gcd (c^{z_2-z_1},g_2)$ and $T$ is the same as in Proposition \ref{bounds_z1<z2_c^z1=2(4)}.
Note that $g_2$ is odd.
We use inequality \eqref{ineq_z1<z2_c^z1=0(4)} to find an upper bound for $c$ for each $d_z$ and for each triple $(\beta,\alpha,z_2)$ satisfying \eqref{beta-alp_region} and $\lceil \alpha/\beta \rceil \le z_1 \le \mathcal U_1(\alpha,\beta,m_2,a_0,b_0,\m_c,g_2)$ with $z_2=z_1+d_z$ and $\m_c=5 \cdot 10^{27}$, where each of the procedures are implemented in two cases according to whether $b>7$ or not.
Moreover, we proceed in several cases according to the value of $g_2$ as below.

\vspace{0.2cm}{\it Case where $g_2=1$.}

\vspace{0.2cm}
Inequality \eqref{ineq_z1<z2_c^z1=0(4)} is
\[
c^{d_z}<\min \! \left \{ 2^{\alpha -\beta z_1}, \, \frac{T\log c}{\log (c-1)}\,\frac{z_2}{z_1} \right \}\cdot \,z_2\h(c).
\]

\vspace{0.2cm}{\it Case where $g_2>1$.}

\vspace{0.2cm}
From Lemma \ref{U_g2}, inequality \eqref{ineq_z1<z2_c^z1=0(4)} yields
\begin{equation}\label{ineq_z1<z2_c^z1=0(4)_g2>1}
c^{d_z}<\min \! \left \{ \frac{\log c}{\log a_0}\,2^{\alpha -\beta z_1}, \, \frac{T\log c}{\log (c-1)}\,\frac{1}{z_1} \right \}\cdot \,z_2^2\h(c).
\end{equation}
Also, Lemma \ref{g2}\,(iv,v) tells us that $\gcd(z_2,6)=1$ if $g_2>3$.
Moreover, $z_2 \not\equiv 0 \pmod{g_2}$ by Lemma \ref{ourcase}.
We proceed in several subcases.

\vspace{0.2cm}{\it (i) Case where $g_2 \equiv 0 \pmod{3}$.}

\vspace{0.2cm}
Since $g_2=3$ and $z_2=2$, it follows that $z_1=1,d_z=1$, and so inequality \eqref{ineq_z1<z2_c^z1=0(4)} is
\[
c<\min \! \left \{ 6 \cdot 2^{\alpha -\beta }, \, \frac{4T\log c}{\log (c-1)}\, \right \}\cdot \h(c).
\]

\vspace{0.2cm}{\it (ii) Case where $d_z \ge 2$.}

\vspace{0.2cm}
Note that $g_2 \ge 5$ by a previous case.
Inequality \eqref{ineq_z1<z2_c^z1=0(4)_g2>1} implies the asserted bounds for each $d_z \ge 2$.

\vspace{0.2cm}{\it (iii) Case where $g_2=5, d_z=1$.}

\vspace{0.2cm}
Since ${g_2}'=1$, inequality \eqref{ineq_z1<z2_c^z1=0(4)} is
\[
c^{d_z}<\min \! \left \{  \frac{2^{\alpha -\beta z_1}}{5}, \, \frac{T\log c}{\log (c-1)}\,\frac{z_2}{z_1}\right \}\cdot \,z_2\h(c)
\]
with $\gcd(z_2,7)=1$.

\vspace{0.2cm}{\it (iv) Case where $g_2=7, d_z=1$.}

\vspace{0.2cm}
Inequality \eqref{ineq_z1<z2_c^z1=0(4)} is used with both $g_2$ and ${g_2}'$ replaced as $7$ and with $\gcd(z_2,5)=1$.

\vspace{0.2cm}{\it (v) Case where $g_2>7, d_z=1$.}

\vspace{0.2cm}
Inequality \eqref{ineq_z1<z2_c^z1=0(4)_g2>1} is used with $g_2 \ge 11$.

By these observations, we find a list of finitely many possible tuples $(d_z,\beta,\alpha,z_2)$ with the corresponding upper bound for $c$, and those satisfy the stated conditions.
\end{proof}

%%%%%%%%%%%%%%%%%%%%%%%%%%%%
\section{ Case where $z_1<z_2$ with $c<\max\{a,b\}$: finding bounds} \label{Sec-bounds-z1<z2_a=max}
%%%%%%%%%%%%%%%%%%%%%%%%%%%%

The aim of this section is to provide a list of possible values or upper bounds of some letters in system \eqref{sys-12} satisfying $z_1<z_2$ with $c<\max\{a,b\}$.
It suffices for us to do this when $a>b$.
We proceed basically in two cases according to whether $d_x=0$ or not.

We begin with showing the technical lemma to reduce the computational time for establishing the forthcoming propositions, which gives relatively small upper bounds for $d_z, d_x$ and $d_y$.

\begin{lem}\label{bounds_acdzdxdy}
Suppose that $d_z>0$ and $a>\max\{b,c\}.$
Then $d_z \le 6, d_x \le 5$ and $d_y \le 16.$
\end{lem}

\begin{proof}
Clearly we may assume that both $d_x,d_y$ are positive.
Observe that $z_1>1$ and $g_2<z_2$ as $a>c$, in particular, $c^{z_1} \equiv 0 \pmod{4}$.
First, Lemma \ref{Lem_ineq_z1<z2_c^z1=0(4)} yields $c^{d_z}<\min \! \left \{ 2^{\alpha -\beta z_1}, \, 1/z_1\right \}\cdot z_2^2\h(c)$.
Similarly to the previous section, we use this inequality, for all possible triples $(\beta,\alpha,z_2)$, to restrict the values of $c$ and $d_z$.
The result gives the asserted bound for $d_z$ and $c<2.5 \cdot 10^6$.
Second, we apply Lemma \ref{Lem_ineq_xi<xj} similarly as used in the proof of Proposition \ref{c<max_implies_z1 ne z2}.
It reveals that
$a^{d_x}<9/4 \cdot (z_2-1+d_x+{d_x}^2)\,\h(\m_c;a,a)$ with $\m_c=2.5 \cdot 10^6$.
For all possible triples $(\beta,\alpha,z_2)$, we use this inequality to restrict the values of $a$ and $d_x$, which gives the asserted bound for $d_x$.
Finally, using Lemma \ref{Lem_ineq_xi<xj} with the base $b$, similarly to the arguments in the proof of Proposition \ref{c<max_implies_z1 ne z2}, we proceed to find that
\begin{align*}
b^{d_y}
&<\frac{(9/4)\log^2 b\,\log c}{\log (b-1)\,\log^2 a} \left(\frac{\log c}{\log b}\,z_2+d_y+{d_y}^2\right)z_3\\
&<\frac{9\log b}{4\log (b-1)}\,\bigr(z_2+d_y+{d_y}^2\bigr)\,\h(\m_c;b,\m_c).
\end{align*}
It is easy to see that this inequality does not hold if $d_y>16$ for any possible triples $(\beta,\alpha,z_2)$.
\end{proof}

Lemma \ref{bounds_acdzdxdy} is implicitly used in the sequel.

\begin{prop}\label{bounds_z1<z2_a=max_xi<xj}
Suppose that
\[
d_z>0, \quad a>\max\{b,c\}, \quad d_x>0.
\]
Then the following hold.
\begin{itemize}
\item[\rm (i)] Suppose that $g_x=1.$ Then
\[
[\beta,\alpha,z_1,z_2] \le [6,18,18,21], \quad a<1.9 \cdot 10^6.
\]
More exactly, one of the following cases holds.
\begin{itemize}
%[\beta,\alpha,z_2] \le [6,8,20], \ a<1600, \ d_x \ge 2;\\
\item[$\bullet$] $[\beta,z_1,z_2] \le [1,3,9], \ a<1.9 \cdot 10^6, \ c<2.1\cdot10^5, \ b \le 7;$
\item[$\bullet$] $[\beta,\alpha,z_1,z_2] \le [6,18,18,21], \ a<1.1\cdot10^6, \ b>7.$
\end{itemize}
\if0
\begin{align*}
%[\beta,\alpha,z_2] \le [6,8,20], \ a<1600, \ d_x \ge 2;\\
&[\beta,z_1,z_2] \le [1,3,9], \ a<1.9 \cdot 10^6, \ c<2.1\cdot10^5, \ b \le 7;\\
&[\beta,\alpha,z_1,z_2] \le [6,18,18,21], \ a<1.1\cdot10^6, \ b>7.
\end{align*}
\fi
\item[\rm (ii)] Suppose that $g_x>1.$ Then
\[
a<4.5\cdot10^6, \quad z_2 \le 22.
\]
More exactly, one of the following cases holds.
\begin{itemize}
\item[$\bullet$] $[\beta,\alpha,z_1,z_2] \le [2,9,11,17], \ a<520, \ k \ne 3;$
\item[$\bullet$] $[\beta,\alpha,z_1,z_2] \le [4,10,18,21], \ a<3100, \ d_x \ge 2, \ k=3;$
%\beta=1, \ z_1={\textstyle \frac{\alpha}{\beta}} \in \{2,3\}, \ z_2 \le 8, \ a<5.8\cdot10^6, \ c<20, \ d_x=1, \ k=3, \ g_x=2, b \le 7;\\
\item[$\bullet$] $[\beta,\alpha,z_1,z_2] \le [3,18,18,21], \ a<4\cdot10^5, \ c<2.7\cdot10^5, \ (g_x,d_x,k)=(2,1,3), \ b>7;$
\item[$\bullet$] $[\beta,z_1,z_2] \le [1,3,9], \ a<4.5\cdot10^6, \ c<30, \ (g_x,d_x,k)=(3,1,3), \ b \le 7;$
\item[$\bullet$] $[\beta,\alpha,z_1,z_2] \le [7,17,17,21], \ a<3.8\cdot10^6, \ c<1.4\cdot10^5,  \ (g_x,d_x,k)=(3,1,3), \ b>7;$
\item[$\bullet$] $[\beta,\alpha,z_1,z_2] \le [1,15,15,17], \ a<8.2\cdot10^4, \ (g_x,d_x,k)=(5,1,3), \ b>7;$
\item[$\bullet$] $[\beta,\alpha,z_1,z_2] \le [2,19,19,22], \ a<8.9\cdot10^5, \ c<6.4\cdot10^5, \ g_x \ge 7, \ (d_x,k)=(1,3), \ b>7,$
\end{itemize}
\if0
\begin{align*}
&[\beta,\alpha,z_1,z_2] \le [2,9,11,17], \ a<520, \ k \ne 3;\\
&[\beta,\alpha,z_1,z_2] \le [4,10,18,21], \ a<3100, \ d_x \ge 2, \ k=3;\\
%\beta=1, \ z_1={\textstyle \frac{\alpha}{\beta}} \in \{2,3\}, \ z_2 \le 8, \ a<5.8\cdot10^6, \ c<20, \ d_x=1, \ k=3, \ g_x=2, b \le 7;\\
&[\beta,\alpha,z_1,z_2] \le [3,18,18,21], \ a<4\cdot10^5, \ c<2.7\cdot10^5, \   (g_x,d_x,k)=(2,1,3), \ b>7;\\
&[\beta,z_1,z_2] \le [1,3,9], \ a<4.5\cdot10^6, \ c<30, \ (g_x,d_x,k)=(3,1,3), \ b \le 7;\\
&[\beta,\alpha,z_1,z_2] \le [7,17,17,21], \ a<3.8\cdot10^6, \ c<1.4\cdot10^5,  \ (g_x,d_x,k)=(3,1,3), \ b>7;\\
&[\beta,\alpha,z_1,z_2] \le [1,15,15,17], \ a<8.2\cdot10^4, \ (g_x,d_x,k)=(5,1,3), \ b>7;\\
&[\beta,\alpha,z_1,z_2] \le [2,19,19,22], \ a<8.9\cdot10^5, \ c<6.4\cdot10^5, \ g_x \ge 7, \ (d_x,k)=(1,3), \ b>7,
\end{align*}
\fi
where $z_1=\alpha/\beta$ if $d_x=1.$
\end{itemize}
\end{prop}

\begin{proof}
Here we just indicate how we find a list of all possible tuples $(\beta,\alpha,z_1,z_2)$ with the corresponding upper bounds for $a$ and $c$.
We proceed basically similarly to Propositions \ref{bounds_z1<z2_c^z1=2(4)} and \ref{bounds_z1<z2_c^z1=0(4)}.

First, we use Lemma \ref{Lem_ineq_z1<z2_c^z1=0(4)}.
Since $a>c, z_1>1$ and $g_2<z_2$, it follows that
\begin{equation}\label{ineq_c_a=max_dx>0}
c^{d_z}<\min \! \left \{ 2^{\alpha -\beta z_1}(z_2-1), \,\frac{T z_2}{z_1}\right \}\cdot z_2\h(c),
\end{equation}
where $T=1$ if $b>7$ and $T=\frac{\log b}{\log (c+1)}$ if $b \le 7$.
For each $d_z$ and for each possible tuple $(\beta,\alpha,z_1)$ satisfying \eqref{beta-alp_region} and $\lceil \alpha/\beta \rceil \le z_1 \le \mathcal U_1(\alpha,\beta,m_2,a_0,b_0,\m_c,1)$ with $\m_c=5 \cdot 10^{27}$, we use inequality \eqref{ineq_c_a=max_dx>0} to find an upper bound for $c$, say $c_u$, where each of the procedures are implemented in two cases according to whether $b>7$ or not.

Next, for each of the found tuples $(d_z,\beta,\alpha,z_2,c_u)$, we use an inequality from Lemma \ref{Lem_ineq_xi<xj} to find an upper bound for $a$, where the used inequality depends on the size of $c$.
For this we proceed in several cases according to the value of $g_x$.
Below, we just indicate the used inequality from \ref{Lem_ineq_xi<xj} with additional remarks.

\vspace{0.3cm} {\it Case where $g_x=1$.}

\vspace{0.2cm}
By Lemma \ref{Lem_ineq_xi<xj}\,(i) with \eqref{z3},
\[
a^{d_x}<\frac{\log c}{\log b} \cdot z_j z_k \le \frac{\log c}{\log b} \cdot z_2 z_3 \le z_2 \h(c_u;a,c_u).
\]

\vspace{0.3cm} {\it Case where $g_x>1, k \ne 3$.}

\vspace{0.2cm}
Since $x_j \le x_k<z_k \le z_2$, Lemma \ref{Lem_ineq_xi<xj}\,(iii) yields
\[
a^{d_x}<\frac{4\log^2 a}{\log (a-1)\,\log \max\{b_0,c_0\}} \cdot (z_2-1+d_x+{d_x}^2)\,z_2.
\]

\vspace{0.3cm} {\it Case where $g_x>1, k=3, d_x \ge 2$.}

\vspace{0.2cm}
Since $x_j<\frac{\log c}{\log a}z_2$, Lemma \ref{Lem_ineq_xi<xj}\,(iii) implies
\[
a^{d_x}<%\frac{4 t_{b,c} \log^2 a}{\log (a-1)\,\log b}\cdot (x_j+d_x+{d_x}^2) z_3
%\le
4 \left( \frac{\log c_u}{\log a}\,z_2+d_x+{d_x}^2 \right) \h(c_u;a,a).
\]

\vspace{0.3cm} {\it Case where $g_x>1, k=3, d_x=1$.}

\vspace{0.2cm}
By Lemma \ref{dagger}, $z_1=\alpha/\beta$.
Note that $j \in \{1,2\}$ and $j$-th equation is $a^{x_j}+b^{y_j}=c^{z_j}$ with both $y_j,z_j$ divisible by $g_x$.
Then $z_j \ge \max\{g_x,x_j+1\}$.
We proceed in several subcases.

\vspace{0.2cm}{\it (i) Case where $g_x=2$.}

\vspace{0.2cm}
Note that $2 \nmid x_j$, so $x_j \ge 3$.
Thus, $z_j \ge 4$ and $2 \mid z_j$.
From Lemma \ref{ourcase}, $3 \nmid z_j$.
Since ${g_x}'=1$, Lemma \ref{Lem_ineq_xi<xj}\,(ii) yields
\[
a<\frac{\log a}{\log b} \cdot x_j z_3< \frac{\log c}{\log b} \cdot z_2 z_3<z_2 \,\h(c_u;a,c_u).
\]

\vspace{0.2cm}{\it (ii) Case where $g_x \equiv 0 \pmod{3}$.}

\vspace{0.2cm}
Note that $(x_j,x_i)=(2,1), g_x=3$ and $z_j$ is not divisible by any of $6,7,8,9,10$ and $15$.
The combination of (ii) and (iii) of Lemma \ref{Lem_ineq_xi<xj} yields that $a<9 \min\{1, T\}\,\h(c_u;a,a)$, where $T=1$ if $b>7$, and  $T=\frac{\log a}{\log (a-1)}\,\frac{\log b}{\log c_0}$ if $b<7$.

\vspace{0.2cm}{\it (iii) Case where $g_x=5$.}

\vspace{0.2cm}
Since $\gcd(x_j,2\cdot3\cdot5\cdot7)=1$, we have $x_j \ge 11$, so $z_j \ge 12$.
Since ${g_x}'=1$, Lemma \ref{Lem_ineq_xi<xj}\,(ii) yields that $a<\frac{1}{4}z_2 \h(c_u;a,c_u)$.

\vspace{0.2cm}{\it (iv) Case where $g_x \not\in\{1,2,3,5\}$.}

\vspace{0.2cm}
Note that $g_x \ge 7$, and Lemma \ref{Lem_ineq_xi<xj}\,(iii) implies that
\[
a<
\frac{49}{36} \left( \min \left\{\frac{\log c_u}{\log a}z_2,z_2-1\right\}+2 \right)\,\h(c_u;a,a).
\]

By these observations, we find a list of finitely many possible tuples $(d_z,\beta,\alpha,z_2,c_u)$ with the corresponding upper bound for $a$, and those satisfy the stated conditions.
\end{proof}

\begin{prop}\label{bounds_z1<z2_a=max_xi=xj_k=3}
Suppose that
\[
d_z>0, \quad a>\max\{b,c\}, \quad d_x=0, \quad k=3.
\]
Then one of the following cases holds.
\begin{itemize}
\item[$\bullet$] $[\beta,\alpha,z_1,z_2] \le [5,16,16,18], \ c<1.1 \cdot 10^6, \ g_2=1;$
\item[$\bullet$] $[\beta,\alpha,z_1,z_2] \le [3,7,15,17], \ a<400, \ g_2 \ge 5.$
\end{itemize}
\if0
\begin{align*}
&[\beta,\alpha,z_1,z_2] \le [5,16,16,18], \ c<1.1 \cdot 10^6, \ g_2=1;\\
&[\beta,\alpha,z_1,z_2] \le [3,7,15,17], \ a<400, \ g_2 \ge 5.
\end{align*}
\fi
\end{prop}

\begin{proof}
Since the method is similar to that of Proposition \ref{bounds_z1<z2_a=max_xi<xj}, we just indicate the inequality used to find an upper bound for $c$ or $a$ for each possible tuple $(\beta,\alpha,z_1,z_2)$.
We proceed in two cases according to whether $g_2=1$ or not.
Note that since $c^{z_1}+b^{y_2}=c^{z_2}+b^{y_1}$ we can apply the restrictions from Lemmas \ref{AB} and \ref{X=x+1}. % with $(A,B)=(c,b)$ and $(x,y,X,Y)=(z_1,y_2,z_2,y_1)$.

\vspace{0.2cm} {\it Case where $g_2=1$.}

\vspace{0.2cm}
By Lemma \ref{Lem_ineq_z1<z2_c^z1=0(4)},
\[
c^{d_z}<\min \biggl\{ 2^{\alpha -\beta z_1}, \, \frac{Tz_2}{z_1} \biggl\} \cdot z_2\h(c).
\]
where $T$ is the same as in inequality \eqref{ineq_c_a=max_dx>0}.
This gives an upper bound for $c$, and the obtained tuples satisfy the stated conditions.

\vspace{0.2cm} {\it Case where $g_2>1$.}

\vspace{0.2cm}
By Lemma \ref{g2}\,(i,ii), we have $\gcd(g_2,6)=1$, so $g_2 \ge 5$.
Thus, $x_j=x_2 \ge g_2 \ge 5$.
First, for each possible tuple $(\beta,\alpha,z_1,z_2)$, we can use inequality \eqref{ineq_c_a=max_dx>0} to find an upper bound for $c$, say $c_u$.
Next, for each of the found tuples $(\beta,\alpha,z_1,z_2,c_u)$, we apply Proposition \ref{=<lam-1}\,(i) for $(A,B,C)=(c,b,a)$ and
$(X_r,Y_r,Z_r)=(z_t,y_t,x_t)$ with $(r,t) \in \{(1,i),(2,j),(3,k)\}$.
Then one of the following inequalities is found.
\begin{gather}
a^{x_j}<\max_{t \in \{1,2\}} \bigr\{\gcd(z_t,x_j) \cdot |z_t x_3 - z_3 x_j| \bigr\} \label{ineq1_xi=xj_k=3},\\
a^{x_j/2}<\max_{t \in \{1,2\}} \bigr\{\gcd(y_t,x_j) \cdot |y_t x_3- y_3 x_j| \bigr\} \label{ineq2_xi=xj_k=3},\\
a^{x_j/2}<\frac{2}{\log a}\,z_3 \label{ineq3_xi=xj_k=3}.
\end{gather}
In cases \eqref{ineq1_xi=xj_k=3} and \eqref{ineq2_xi=xj_k=3}, respectively, we see that
\begin{align*}
a^{x_j}&<x_j \cdot \max_{t \in \{1,2\}} \{ z_t x_3,z_3 x_j \}\\
& \le x_j \cdot \max_{t \in \{1,2\}} \left\{ z_t \cdot \frac{\log c}{\log a}z_3,\, z_3 x_j \right\}\\
& \le x_j \cdot z_3 \cdot \max_{t \in \{1,2\}} \left\{ \frac{\log c}{\log a}z_t,\, x_j \right\}
=x_j \cdot z_3 \cdot \frac{\log c}{\log a} z_2;\\
a^{x_j/2}&<x_j \cdot \max_{t \in \{1,2\}} \{ y_t x_3,y_3 x_j \}\\
& < x_j \cdot \max_{t \in \{1,2\}} \left\{ \frac{\log c}{\log b}z_t \cdot \frac{\log c}{\log a}z_3,\,
\frac{\log c}{\log b}z_3 \cdot x_j \right\}\\
&= x_j \cdot \frac{\log c}{\log b}z_3 \cdot \max_{t \in \{1,2\}} \left\{ \frac{\log c}{\log a}z_t , x_j \right\}
=x_j \cdot \frac{\log c}{\log b}z_3 \cdot \frac{\log c}{\log a}z_2.%;\\
\end{align*}
These together with \eqref{ineq3_xi=xj_k=3} and \eqref{z3} imply one of the following inequalities:
\[
a^{x_j}/x_j<z_2 \h(c_u;S,c_u); \quad a^{x_j/2}/x_j<z_2 \h(c_u); \quad a^{x_j/2}<2 \h (c_u;S),
\]
where $S=a-1$ if $b>7$ and $S=b$ if $b \le 7$.
Each of these inequalities together with $x_j \ge 5$ gives an upper bound for $a$, and the obtained tuples satisfy the stated conditions.
\end{proof}

\begin{prop}\label{bounds_xi<xj_z1<z2_a=max_xi=xj_k<>3}
Suppose that
\[
d_z>0, \quad a>\max\{b,c\}, \quad d_x=0, \quad k \ne 3.
\]
Then $y_3 \ge \max\{y_1,y_2\},$ and
\[
[\beta,\alpha,z_1,z_2] \le [7,19,21,23], \quad a<3.9 \cdot 10^{10}.
\]
More exactly, one of the following cases holds.
\begin{itemize}
\item[\rm (i)]
$y_1 \le y_2, \,d_z \ge 2, b^{d_y} \equiv c^{d_z} \pmod{a^{\min\{x_1,x_2\}}}, \,b^{d_y}<c^{d_z},$ and one of the following cases holds.
\begin{itemize}
\item[$\bullet$] $[\beta,\alpha,z_1,z_2] \le [4,16,16,18], \ c^{d_z}<1.2 \cdot 10^6, \ g_2=1;$
\item[$\bullet$] $[\beta,\alpha,z_1,z_2] \le [3,17,21,23], \ c^{d_z}<2.4 \cdot 10^5, \ g_2 \ge 5.$
\end{itemize}
\item[\rm (ii)]
$y_1 \le y_2, \,d_y \ge 2, \,b^{d_y} \equiv c^{d_z} \pmod{a^{\min\{x_1,x_2\}}}, \,b^{d_y}>c^{d_z},$ and one of the following cases holds.
\begin{itemize}
\item[$\bullet$] $[\beta,\alpha,z_1,z_2] \le [6,9,18,20], \ b^{d_y}<2.7 \cdot 10^5, \ g_y \in \{1,2,5\};$
\item[$\bullet$] $[\beta,\alpha,z_1,z_2] \le [2,9,18,21], \ b^{d_y}<5.4 \cdot 10^5, \ g_y \ge 7.$
\end{itemize}
\item[\rm (iii)]
$y_1>y_2, \,x_1<x_2, \,a^{x_1} \mid (b^{d_y}c^{d_z}-1), \,g_y \in \{1,2,5\},$ and one of the following cases holds.
\begin{itemize}
\item[$\bullet$] $x_1=1, \ [\beta,\alpha,z_1,z_2] \le [7,19,19,21], \ a<3.9 \cdot 10^{10}, \ b<5.3 \cdot 10^5, \ c<1.3 \cdot 10^5, \ g_2=1;$
\item[$\bullet$] $x_1 \ge 2, \ [\beta,\alpha,x_1,z_1,z_2] \le [4,15,6,18,19], \ a<7.7 \cdot 10^4, \ b<6.9 \cdot 10^4, \ c<6.6 \cdot 10^4, \ g_2=1;$
\item[$\bullet$] $x_1=1, \ [\beta,\alpha,z_1,z_2] \le [4,18,18,19], \ a<3 \cdot 10^{10}, \ b<4 \cdot 10^5, \ c<1.2 \cdot 10^5, \ g_2 \ge 5;$
\item[$\bullet$] $x_1 \ge 2, \ [\beta,\alpha,x_1,z_1,z_2] \le [4,15,7,15,19], \ a<1.2 \cdot 10^5, \ b<6.5 \cdot 10^4, \ c<9.2 \cdot 10^4, \ g_2 \ge 5.$
\end{itemize}
\end{itemize}
\end{prop}

\begin{proof}
First, we rewrite the three equations as
\[
a^{x}+b^{y_I}=c^{z_I}, \ \ a^{x_J}+b^{y_J}=c^{z_J}, \ \ a^{x}+b^{y_3}=c^{z_3}
\]
with $\{I,J\}=\{1,2\}$ and $x=x_3$.

Note that $c^{z_3}<2\max\{a^x,b^{y_3}\}$.
If $c^{z_3}<2a^x$, then $c^{z_3}<2a^x<2c^I$, so $z_3 \le z_I$.
This implies that $z_I=z_3$, which is absurd as $(y_I,z_I) \ne (y_3,z_3)$.
Thus $c^{z_3}<2b^{y_3}$, so $2b^{y_3}>c^{z_2}>b^{\max\{y_1,y_2\}}$, which shows the first assertion.
Therefore, $d_y=|y_2-y_1|$ with $n=3$.

Next, we show the following:
\begin{gather}
b^{y_2-y_1} \equiv c^{d_z} \mod a^{\min\{x_1,x_2\}}, \label{last-cong}\\
z_3<\mathcal U_3:=\max\{ 1000 z_2,\,2523 \log b\} \label{U3again}.
\end{gather}
Congruence \eqref{last-cong} follows from taking 1st and 2nd equations modulo $a^{\min\{x_1,x_2\}}$.

Let $\varepsilon=998$. If $a^{(1+\varepsilon)x} \ge b^{y_3}$, then $c^{z_3}=a^{x}+b^{y_3}<2a^{(1+\varepsilon)x}<2c^{(1+\varepsilon)z_I}$, so $z_3 \le \frac{\log 2}{\log c}+(1+\varepsilon) z_I<1000 z_I$.
If $a^{(1+\varepsilon)x}<b^{y_3}$, then inequality $z_3<2523 \log b$ is deduced almost similarly to the proof of inequality \eqref{U3}.
To sum up, \eqref{U3again} holds.

Third, we combine Lemma \ref{Lem_ineq_z1<z2_c^z1=0(4)} with inequality \eqref{U3again}.
Since $\frac{({g_2}')^2}{g_2} \le g_2<z_2$, we have
\begin{equation}\label{ineq-c-final}
c^{d_z}<\min \! \left \{ 2^{\alpha -\beta z_1} Z, \,T z_2/z_1 \right \}\cdot \frac{(\log^2 c)\,z_2\,{\mathcal U_3}'}{\log \max\{a_0,c+1\}},
\end{equation}
where $Z=1$ if $g_2=1$, and $Z=z_2-1$ if $g_2>1$, and $T$ is the same as in \eqref{ineq_c_a=max_dx>0},
and ${\mathcal U_3}'=\max\{ 1000 z_2/\log b_0,\,2523\}$.
Note that $g_2 \ge 5$ if $g_2>1$.

Fourth, we apply Lemma \ref{Lem_ineq_xi<xj} with the base $b$ together with \eqref{U3again} to see that
\begin{equation}\label{ineq-b-final}
\begin{split}
&b^{d_y}<\frac{\log c_u}{\log C} \cdot z_2 \,\mathcal U_3, \quad {\text{if $g_y \in \{1,2,5\},$}}\\
&b^{d_y}<\frac{(49/36)\log b}{\log (b-1)} \cdot \left(z_2+d_y+{d_y}^2\right)\, \frac{\log c_u}{\log C} \,\mathcal U_3, \ \ \text{if $g_y \not\in \{1,2,5\},$}
\end{split}
\end{equation}
where $c_u$ is any upper bound for $c$, and $C=\max\{a_0,b+2,c_u+1\}$.
Note that $g_y \ge 7$ if $g_y \not\in \{1,2,5\}$.

In the remaining cases, we proceed in three cases separately.
In each of those cases, similarly to previous propositions, we just indicate how we find a list of all possible tuples composed of $\beta,\alpha,x_1,z_2,d_z,d_y$ and the corresponding upper bounds for some of $a,b,c$.

\vspace{0.3cm} {\it Case where $y_2 \ge y_1$ and $b^{d_y}<c^{d_z}$.}

\vspace{0.2cm}
By congruence \eqref{last-cong}, we have $a^{\min\{x_1,x_2\}}<c^{d_z}$, in particular, $d_z>1$ as $a>c$.
Taking these restrictions into consideration, we use inequality \eqref{ineq-c-final} to find an upper bound for $c$ for each possible tuple $(\beta,\alpha,d_z,z_2)$, where these procedures are implemented in the cases according to whether $b>7$ or not, and whether $g_2=1$ or not.
The obtained tuples satisfy the stated conditions of (i).

\vspace{0.3cm} {\it Case where $y_2 \ge y_1$ and $b^{d_y}>c^{d_z}$.}

\vspace{0.2cm}
Similarly to the previous case, we have $a^{\min\{x_1,x_2\}}<b^{d_y}$ and $d_y>1$.
Note that $z_2 \ge 2g_y$ as $g_y=\gcd(x_2,z_2)$ with $x_2<z_2$.
Taking these restrictions into consideration, similarly to the previous case, we find a list of all possible tuples $(\beta,\alpha,d_z,z_2,c_u)$ with $c_u$ the corresponding upper bound for $c$ derived from inequalities \eqref{ineq-c-final}, where these procedures are implemented in two cases according to whether $b>7$ or not.
Finally, for each of the found tuples and for each $d_y$, we use inequality \eqref{ineq-b-final} to find an upper bound for $b$ in two cases according to whether $g_y \in \{1,2,5\}$ or not.
The obtained tuples satisfy the stated conditions of (ii).

\vspace{0.3cm} {\it Case where $y_1>y_2$.}

\vspace{0.2cm}
Note that $d_y=y_1-y_2$ with $g_y=\gcd(x_1,z_1)$, and $z_1 \ge 2g_y$.
We proceed almost similarly to the previous case.
There are two main differences.
In this case we have $a^{x_1}<b^{d_y}c^{d_z}$ from congruence \eqref{last-cong}, in particular, $d_y+d_z>x_1$ as $a>\max\{b,c\}$.
Further, inequality \eqref{ineq-b-final} is changed, its factor $z_2$ replaced by $z_1$.
Taking these into consideration, for each $x_1$ we have a list of all possible tuples $(\alpha,\beta,d_y,d_z,z_2,a_u,b_u,c_u)$ where $b_u$ is the corresponding upper bound for $b$ and $a_u:=\lfloor (b_u \cdot c_u)^{1/x_1}\rfloor$ is the corresponding upper bound for $a$.
The obtained tuples satisfy the stated conditions of (iii).
\end{proof}

Under the assumption that equation \eqref{abc} has three solutions $(x_t,y_t,z_t)$ with $t \in \{1,2,3\}$ satisfying $z_1<z_2 \le z_3$, the propositions established in the previous two sections provide us middle-sized bounds on the base numbers $a,b,c$ and on the exponential unknowns $x_t,y_t,z_t$ for $t \in \{1,2\}$.
Although those bounds are relatively sharp, a direct enumeration of the possible solutions of system \eqref{sys-12} is impossible.
In order to find efficient methods for reducing the obtained bounds, we need to be more sophisticated than in the case where $z_1=z_2$.
In the next two sections, we investigate system \eqref{sys-12} with $z_1<z_2$ and explicitly present our reduction algorithms for the cases $c>\max\{a,b\}$ and $c<\max\{a,b\}$ respectively, where it suffices to consider the case where $a>b$.

%%%%%%%%%%%%%%%%%%%%%%%%%%%%%%%%%%%%%%%%%%
\section{Case where $z_1<z_2$ and $c>\max\{a,b\}$: Sieving}
\label{Sec-sieve-z1<z2_c=max}
%%%%%%%%%%%%%%%%%%%%%%%%%%%%%%%%%%%%%%%%%%

The aim of this section is to show that there is no solution of system \eqref{sys-12} fulfilling the statements of Propositions \ref{bounds_z1<z2_c^z1=2(4)} and \ref{bounds_z1<z2_c^z1=0(4)}, respectively.
It suffices to consider the case where $c>a>b$, and we put
\[
a_0=\max\{11,2^\alpha+1\}, \ b_0=2^\alpha-1, \ c_0=\max\{18,3 \cdot 2^\beta,2^\alpha+2\}.
\]
Recall that these numbers are uniform lower bounds for $a,b$ and $c$, respectively.

We proceed in two cases according to whether $c^{z_1} \equiv 2 \pmod{4}$ or $c^{z_1} \equiv 0 \pmod{4}$.

%%%%%%%%%%%%%%%%%%%%%%%%%%%%%%%%%%%%%%
\subsection{Case where $c^{z_1} \equiv 2 \pmod{4}$}
\label{Sec-sieve-z1<z2_c=max_c^z1=2(4)}
%%%%%%%%%%%%%%%%%%%%%%%%%%%%%%%%%%%%%%

System \eqref{sys-12} is
\begin{equation} \label{sys2}
\begin{cases}
\,a^{x_1}+b^{y_1}=c, \\
\,a^{x_2}+b^{y_2}=c^{z_2}
\end{cases}
\end{equation}
with $\beta=1$.

First, we show two lemmas to give several restrictions on the solutions of system \eqref{sys2}.

\begin{lem} \label{red1}
Let $(x_1,y_1,x_2,y_2,z_2)$ be a solution of system \eqref{sys2}.
Then the following hold.
\begin{itemize}
\item[(i)] $x_2$ or $y_2$ is odd.
\item[(ii)] If both $x_1$ and $x_2$ are odd, then $y_1$ or $y_2$ is odd.
\item[(iii)] If both $x_1$ and $x_2$ are even, then $y_1$ or $y_2$ is even.
\item[(iv)] One of $x_1,x_2,y_1$ and $y_2$ is even.
\item[(v)] $x_2>x_1$ or $y_2>y_1.$
\item[(vi)] $x_2 \ge z_2$ or $y_2 \ge z_2.$
\item[(vii)] $x_1y_2 \ne x_2y_1, \ y_1z_2 \ne y_2, \ x_1z_2 \ne x_2.$
\item[(viii)] $\min\{x_1,x_2\}<|y_1z_2-y_2|.$
\item[(ix)] Assume that $b<11.$
Then $(2 \nmid x_2$ or $3 \nmid z_2)$ and $(3 \nmid x_2$ or $2 \nmid z_2).$
\end{itemize}
\end{lem}

\begin{proof}
(i) This is a direct consequence of Lemma \ref{g2}\,(i).

(ii) Suppose that both $x_1,x_2$ are odd and both $y_1,y_2$ are even.
Then $a^{x_i} \equiv a \pmod 4$ and $b^{y_i} \equiv 1 \pmod 4$ for $i \in \{1,2\}$.
Since $c \equiv 2 \pmod 4$ and $z_2>1$, 1st equation leads to $a \equiv c-1 \equiv 1 \pmod{4}$, while 2nd one leads to $a \equiv c^{z_2}-1 \equiv -1 \pmod{4}$.
These are incompatible.

(iii,iv) These are shown similarly to (ii).

(v,vi) These easily follow from the inequality $a^{x_1}+b^{y_1}<a^{x_2}+b^{y_2}=c^{z_2}$ with $c>\max\{a,b\}$.

(vii) This is a direct consequence of applying Lemma \ref{two} to the equations in \eqref{sys2}.

(viii) We take the equations in \eqref{sys2} modulo $a^{\min\{x_1,x_2\}}$ to see that $b^{|y_1z_2-y_2|} \equiv 1 \pmod{a^{\min\{x_1,x_2\}}}$.
Since $a>b$, and $y_1z_2-y_2 \ne 0$ by (vii), the obtained congruence leads to the assertion.

(ix) If $2 \mid x_2$ and $3 \mid z_2$, then 2nd equation is of the form $A^2+b^{y_2}=C^3$ with $b \in \{3,5,7\}$.
For $S=\{b\}$, we compute the $S$-integral points $(A,C)$ on this elliptic curve.
None of the found points leads to a solution of the system.
The remaining case is similarly handled.
\end{proof}

\begin{lem} \label{red2}
Let $(x_1,y_1,x_2,y_2,z_2)$ be a solution of system \eqref{sys2}. Then
\begin{gather}
\begin{cases}
\,a^{|x_1z_2-x_2|} \equiv 1 \mod{2b^{\,\min\{y_1,y_2\}}},\\
\,b^{|y_1z_2-y_2|} \equiv 1 \mod{2a^{\,\min\{x_1,x_2\}}}, \\
\end{cases}\label{kongr}  \\
a \ge a_1, \ \ b \le b_1, \notag
\end{gather}
where
\[
a_1:=\max\{a_0,b+2\}, \ \ b_1:=\min\bigr\{\lfloor{c^{1/x_1}}\rfloor, \lfloor{c^{1/y_1}}\rfloor, \lfloor{c^{z_2/x_2}}\rfloor, \bigl\lfloor{c^{z_2/y_2}}\rfloor \bigr\}.
\]
Moreover, the following hold.
\begin{itemize}
\item[(i)]
Suppose that $a^{x_2}>b^{y_2}$ and $a^{x_1}>b^{y_1}.$ Then
\[
0<x_1z_2-x_2 \le t_1, \quad a \le a_2,
\]
where $t_1=\bigl \lfloor \frac{\log 2}{\log a_0}z_2 \bigl \rfloor,$ and
\[
a_2:=\min \Bigr\{\bigl \lfloor 2^{z_2/(x_2-x_1z_2)} \bigl \rfloor, \lfloor{c^{1/x_1}}\rfloor, \lfloor{c^{z_2/x_2}}\rfloor \Bigr\}.
\]
\item[(ii)]
Suppose that $a^{x_2}>b^{y_2}$ and $a^{x_1}<b^{y_1}.$
Then the following hold.
\[ \tag{a}
\begin{cases}
\,y_1>x_1, \ x_1y_2<x_2y_1, \ x_2-x_1z_2 \ge 1, \\
\,y_2-y_1z_2 \le t_2, \ x_2-y_1z_2 \le t_3,
\end{cases}
\]
where $t_2=\bigl \lfloor \frac{\log 2}{\log b_0}(z_2-1)\bigl \rfloor$ and $t_3=\bigl \lfloor \frac{\log 2}{\log a_0}z_2 \bigl \rfloor.$
\begin{gather*}
\tag{b}
a \ge a_4:=\max\Bigr \{ a_1, \bigl\lfloor b^{y_1z_2/x_2}/2^{1/x_2} \bigl\rfloor+1 \Bigr \},\\
\tag{c}
a \le a_5:=\min\Bigr\{ \bigl\lfloor 2^{z_2/x_2}\,b^{y_1z_2/x_2} \rfloor, \lfloor{c^{1/x_1}}\rfloor, \lfloor{c^{z_2/x_2}}\rfloor, \bigl\lfloor{b^{y_1/x_1}}\rfloor \Bigr\}.
\end{gather*}
\item[(iii)]
Suppose that $a^{x_2}<b^{y_2}$ and $a^{x_1}>b^{y_1}.$
Then the following hold.
\[ \tag{a}
\begin{cases}
\,y_2>x_2,  \ x_1y_2>x_2y_1,  \ y_2-y_1z_2 \ge 1, \ y_2 \ge x_1z_2, \\
\,x_2-x_1z_2 \le t_4,
\end{cases}
\]
where $t_4=\bigl \lfloor \frac{\log 2}{\log a_0}(z_2-1) \bigl \rfloor.$
\begin{gather*}
\tag{b}
a \ge a_6:=\max\Bigr\{a_1, \bigl\lfloor b^{y_2/(x_1z_2)} / 2^{1/x_1} \bigl\rfloor+1\Bigr\},\\
\tag{c}
a \le a_7:=\min\Bigr\{ \bigl\lfloor 2^{1/(x_1z_2)} b^{y_2/(x_1z_2)} \bigl\rfloor, \lfloor{c^{1/x_1}}\rfloor, \lfloor{c^{z_2/x_2}}\rfloor, \lfloor{b^{y_2/x_2}}\rfloor \Bigr\}.
\end{gather*}
\item[(iv)]
Suppose that $a^{x_2}<b^{y_2}$ and $a^{x_1}<b^{y_1}.$
Then the following hold.
\[ \tag{a}
\begin{cases}
\,y_2>x_2, \ y_1>x_1, \ y_2-x_1z_2 \ge 1, \\
\,x_2-y_1z_2 \le t_5, \ 2 \le y_2-y_1z_2 \le t_6,
\end{cases}
\]
where $t_5=\bigl \lfloor \frac{\log 2}{\log a_0}(z_2-1) \bigl \rfloor$ and $t_6=\bigl \lfloor \frac{\log 2}{\log b_0}z_2 \bigl \rfloor.$
\[ \tag{b}
a \le a_8:=\min\bigr\{\lfloor{c^{1/x_1}}\rfloor, \lfloor{c^{z_2/x_2}}\rfloor, \lfloor{b^{y_1/x_1}}\rfloor, \lfloor{b^{y_2/x_2}}\rfloor \bigr\}.
\]
\end{itemize}
\end{lem}

\begin{proof}
From \eqref{sys2},
\begin{equation} \label{kongr2}
a^{x_2}+b^{y_2}=(a^{x_1}+b^{y_1})^{z_2}.
\end{equation}
The congruences in \eqref{kongr} follow by taking equation \eqref{kongr2} modulo $b^{\min\{y_1,y_2\}}$ and $a^{\min\{x_1,x_2\}}$.
The next asserted inequality giving an upper bound for $b$ follows easily from system \eqref{sys2}.

(i) From \eqref{kongr2} with $a^{x_2}<b^{y_2}$, observe that
\begin{align*}
a^{x_2}<a^{x_2}+b^{y_2}=&(a^{x_1}+b^{y_1})^{z_2}<(2a^{x_1})^{z_2}=2^{z_2}a^{x_1z_2},\\
a^{x_1z_2}<&(a^{x_1}+b^{y_1})^{z_2}=a^{x_2}+b^{y_2}<2a^{x_2}.
\end{align*}
These inequalities together imply
\begin{equation}\label{i1a}
\frac{1}{2}<a^{x_2-x_1z_2}<2^{z_2}.
\end{equation}
The left-hand inequality shows that $x_2-x_1z_2 \ge 0$, so $x_2-x_1z_2>0$ by Lemma \ref{red1}\,(vii),  while the right-hand one implies that $x_2-x_1z_2 < \frac{\log 2}{\log a}z_2 \le t_1$.

On the other hand, $a<2^{\frac{z_2}{x_2-x_1z_2}}$ by the right-hand inequality of \eqref{i1a}.
Also, by \eqref{sys2}, $a^{x_1}<c$ and $a^{x_2}<c^{z_2}$, leading to $a \le a_2$.

(ii) Since $a^{x_1}<b^{y_1}$ with $a>b$, and $a^{x_1}<b^{y_1}<(a^{x_2/y_2})^{y_1}$, we have $y_1>x_1$ and $x_1y_2<y_1x_2$.
The remaining three inequalities in (a) can be proven in exactly the same way as the corresponding results in (i)/(a).
It remains to show (b) and (c).

Observe from \eqref{kongr2} that
\[
b^{y_1z_2}<a^{x_2}+b^{y_2}<2a^{x_2}, \quad  a^{x_2}<(a^{x_1}+b^{y_1})^{z_2}<2^{z_2}b^{y_1}.
\]
The former inequality yields $a>b^{y_1z_2/x_2}/2^{1/x_2}$, and so (b) holds, while the latter inequality implies that $a<(2^{z_2}b^{y_1})^{1/x_2}$, leading to (c).

(iii)-(iv) These are shown similarly to (i) and (ii).
\end{proof}

In what follows, we proceed in two cases according to whether $(x_1,y_1)=(1,1)$ or not.

%%%%%%%%%%%%%%%%%%%%%%%%%%%%%%%%%%%%%%%%%%%%%%
\subsubsection{Case where $c^{z_1} \equiv 2 \pmod{4}$ with $(x_1,y_1) \ne (1,1)$}
%%%%%%%%%%%%%%%%%%%%%%%%%%%%%%%%%%%%%%%%%%%%%%

As already mentioned in the end of Section \ref{Sec-AB}, it is very efficient to rely upon the existing results on ternary Diophantine equations which are summarized in Lemma \ref{ourcase}.
In our algorithms we use Lemma \ref{ourcase} without any further special reference and combine it with Lemmas \ref{red1} and \ref{red2}.

Proposition \ref{bounds_z1<z2_c^z1=2(4)} gives us a list of all possible values of $\alpha$ and $z_2$ together with the corresponding upper bound for $c$, say $c_u=c_u(\alpha,z_2)$.
%Note that $2 \le \alpha \le 8$ and $2 \le z_2 \le 15$ in this subcase. Since $b>7$ we get by \eqref{aa} that $b_0=11$ and $a_0=13$.

We divide our algorithm in four parts according to (i)-(iv) of Lemma \ref{red2}.
The basic strategy is similar in each of these cases, where the cases $b>7$ and $b \in \{3,5,7\}$ are distinguished.
First we give the details of our reduction method for case (i) under the assumption that $b>7$.

\vspace{0.2cm}
(i) Case where $a^{x_2}>b^{y_2}$ and $a^{x_1}>b^{y_1}$ with $b>7$.

\vspace{0.2cm}{\it Step I. \ Initialization.\ }
We have an explicitly determined list of all possible triples $(\alpha, z_2, c_u)$ satisfying system \eqref{sys2}.
We put these data into the list named $clist$.

\vspace{0.2cm}{\it Step II.} \
We generate a list named $list1$ containing elements of the form $[x_1,y_1,x_2,y_2,\alpha, z_2, c_u]$, where the last three elements are the same as the elements of $clist$, while the first four elements are the possible solutions $(x_1,y_1,x_2,y_2)$ restricted by Lemma \ref{red1} and Lemma \ref{red2} (i)/(a).
The construction of $list1$ is given by the following program.

\vspace{0.2cm}{\tt
for each element of $clist$ do

\hskip.2cm for $x_1:=1$ to $\bigl\lfloor (\log c_u)/\log a_0 \bigl\rfloor$ do

\hskip.2cm for $x_2:=1$ to $\bigl\lfloor z_2(\log c_u)/\log a_0 \bigl\rfloor$ do

\hskip.2cm for $y_1:=1$ to $\bigl\lfloor (\log c_u)/\log b_0 \bigl\rfloor$ do

\hskip.2cm for $y_2:=1$ to $\bigl\lfloor z_2(\log c_u)/\log b_0 \bigl\rfloor$ do

\hskip.4cm sieve using  Lemma \ref{red1} and Lemma \ref{red2} (i)/(a)

end}

\vspace{0.2cm}\noindent
In the last line of the above program, we take into account the restrictions from Lemma \ref{red1} together with the fact that $0<x_2-x_1z_2 \le t_1$, where $t_1=\bigl \lfloor \frac{\log 2}{\log a_0}z_2 \bigl \rfloor$ by Lemma \ref{red2} (i)/(a).

\vspace{0.2cm}{\it Step III.} \
Using the elements of $list1$ and the bounds $a_1,a_2$ and $b_1$, we check for each of possible values of $a$ and $b$ whether congruences \eqref{kongr} hold or not.
It turned out that at least one of them does not hold in any case.
The details are given as follows.
First, from \eqref{cond}, we are in one of the cases:
\[
[a,b] \equiv [-1,1],[-1,-1],[1,-1] \mod{2^{\alpha}},
\]
where $[a,b] \equiv [u,v] \pmod{2^{\alpha}}$ means $a \equiv u \pmod{2^{\alpha}}$ and $b \equiv v \pmod{2^{\alpha}}$.
Define the list $sig=\bigr[[-1,1],[-1,-1],[1,-1]\bigr]$ as a possible list of signatures.
We proceed as follows.

\vspace{0.2cm}{\tt
begin

for each element of $list1$ do

for each element $s$ of $sig$ do

$d_a:=s[1]$ and $d_b:=s[2]$ and $T_b:=\lceil (b_0-d_b)/2^{\alpha} \rceil$

\hskip.2cm for $b:=T_b \cdot 2^{\alpha}+d_b$ to $b_2$ by $2^{\alpha}$ do

\hskip.4cm $T_a:=\lceil (a_1-d_a)/2^{\alpha} \rceil$

%\hskip.4cm if $x_2-x_1z_2 \ne 0$ then

\hskip.6cm for $a:=T_a \cdot 2^{\alpha}+d_a$ to $a_2$ by $2^{\alpha}$ do

\hskip.8cm if $(a^{x_2}>b^{y_2})$ and $(a^{x_1}>b^{y_1})$ then

\hskip1cm sieve using congruences \eqref{kongr}

end}

\vspace{0.2cm}\noindent
We implemented the above algorithms and
it turned out there is no solution of system \eqref{sys2}.
Case (i) with $b \le 7$ can be handled similarly, the only difference being that in {\it Step III} the range for $b$ is replaced by $b \in \{3,5,7\}$.

By using the same strategy as above and the bounds for $a$ and $b$ from Lemma \ref{red2}, we can handle the cases according to (ii)-(iv) of Lemma \ref{red2}, as well.

%%%%%%%%%%%%%%%%%%%%%%%%%%%%%%%%%%%%%%%%%%
\subsubsection{Case where $c^{z_1} \equiv 2 \pmod{4}$ with $(x_1,y_1)=(1,1)$}
%%%%%%%%%%%%%%%%%%%%%%%%%%%%%%%%%%%%%%%%%%

Note that only cases (i) and (iii) of Lemma \ref{red2} can occur.

\vspace{0.2cm}
{\it (i) Case where $a^{x_2}>b^{y_2}$.}

\vspace{0.2cm}
We can proceed exactly in the same way as in case (i) of the case $(x_1,y_1) \ne (1,1)$.

\vspace{0.2cm}
{\it (iii) Case where $a^{x_2}<b^{y_2}$.}

\vspace{0.2cm}
We basically follow the method described in (iii) of the case $(x_1,y_1) \ne (1,1)$ with one important modification for the case where $b>7$.
Namely, in order to increase the efficiency of our algorithm, we make use of Lemma \ref{bounds_z1<z2_c^z1=2(4)_x1=y1=1}, which says that $\min\{x_2,y_2\} \le 7$ and provides us the sharp upper bounds for $b$, that is, $181, 45, 21$ and $13$, according to the cases $\min\{x_2,y_2\}=4,5,6$ and $7$, respectively.
We built in these information in our program and it turned out there is no solution to the system.

\vspace{0.2cm}
The total computational time in Subsection \ref{Sec-sieve-z1<z2_c=max_c^z1=2(4)} did not exceed 1 hour.

%%%%%%%%%%%%%%%%%%%%%%%%%%%%%%%%%%%%%%%%
\subsection{Case where $c^{z_1} \equiv 0 \pmod{4}$}
\label{Sec-sieve-z1<z2_c=max_c^z1=2(4)_c^z1=0(4)}
%%%%%%%%%%%%%%%%%%%%%%%%%%%%%%%%%%%%%%%%%

Note that $\beta>1$ or $z_1>1$. Also, from \eqref{sys-12},
\begin{align}
(a^{x_1}+b^{y_1})^{z_2}=(a^{x_2}+b^{y_2})^{z_1}, \label{final-eq} \\
(c^{z_1}-b^{y_1})^{x_2}=(c^{z_2}-b^{y_2})^{x_1}. \label{final-eq2}
\end{align}
\if0 Also, according to Proposition \ref{bounds_z1<z2_c^z1=0(4)}, we know that
\begin{equation} \label{assu1}
\begin{cases}
\,z_2-z_1 \in \{1,2,3,4\}, \\
\,z_2 \le 19 \ \text{or} \ z_2=23, g_2 \ge 11.
\end{cases}
\end{equation}
\fi

First, we deal with two special cases.

\begin{lem} \label{trivi-lemma}
Under the hypothesis of Proposition \ref{bounds_z1<z2_c^z1=0(4)}, if $c<1000,$ then system \eqref{sys-12} has no solution.
\end{lem}

\begin{proof}
By Proposition \ref{bounds_z1<z2_c^z1=0(4)}, we know that $z_2-z_1 \in \{1,2,3,4\}$ and $z_2 \le 23$.
Under these restrictions, if $\max\{a,b,c\}\,(=c)$ is small, for example, $c<1000$, a brute force is enough to verify that equation \eqref{final-eq} does not hold for any possible tuples $(a,b,x_1,y_1,z_1,x_2,y_2,z_2)$.
\end{proof}

\begin{lem} \label{small-z2}
Under the hypothesis of Proposition \ref{bounds_z1<z2_c^z1=0(4)}, system \eqref{sys-12} has no solution $(x_1,y_1,z_1,x_2,y_2,z_2)$ satisfying $(z_1,z_2) \in \{(1,2),(2,3)\}.$
\end{lem}

\begin{proof}
We proceed in two cases according to whether $(z_1,z_2)=(1,2)$ or $(2,3)$.
By Proposition \ref{bounds_z1<z2_c^z1=0(4)}, we may assume that $c \le c_U$,
where $c_U=5.5 \cdot 10^5$ if $(z_1,z_2)=(1,2)$, and $c_U=1.5 \cdot 10^6$ if $(z_1,z_2)=(2,3)$.

\vspace{0.2cm}{\it I. Case where $(z_1,z_2)=(1,2)$.}

\vspace{0.2cm}
System \eqref{sys-12} is
\begin{equation} \label{sys5}
a^{x_1}+b^{y_1}=c, \ \ a^{x_2}+b^{y_2}=c^2.
\end{equation}
Note that $2x_1 \ne x_2$ by Lemma \ref{two}.
We further consider several subcases.

\vspace{0.3cm}{\it I/(i). Case where $x_1 \ge 2$ or $x_2 \ge 4$.}

\vspace{0.2cm}
From system \eqref{sys5}, observe that
\[
a<\min\{c^{1/x_1},c^{2/x_2} \}=c^{1/2} \le {c_U}^{1/2}.
\]
Then $a$ is small.
It is not hard to enumerate all possible tuples $(a,b,x_1,y_1,x_2,y_2)$, and to verify that none of those satisfies equation \eqref{final-eq}.

\vspace{0.2cm}{\it I/(ii). Case where $(x_1,x_2)=(1,3)$}.

\vspace{0.2cm}
It is not hard to see that $y_1>1$.
From \eqref{sys5}, $a<{c_U}^{2/3}$ and $b<{c_U}^{1/2}$, thereby both $a$ and $b$ are small enough to deal with this case similarly to case I/(i).

\vspace{0.2cm}
{\it I/(iii). Case where $(x_1,x_2)=(1,1)$.}

\vspace{0.2cm}
Actually, this case can be handled by the methods described in Section \ref{Sec-z1=z2}.
However, we present this with an important idea to find a good restriction on solutions, which will play an important role in other difficult cases.

From \eqref{sys5}, we have
\begin{equation} \label{sys9}
c+b^{y_2}=c^2+b^{y_1}, \quad c-b^{y_2/2}=\frac{a}{c+b^{y_2/2}}.
\end{equation}
We apply Lemma \ref{X=x+1} with $(A,B)=(c,b)$ and $(x,y,X,Y)=(1,y_2,2,y_1)$ to see that
\[
y_2>6y_1-2 \ge 4, \quad c \equiv -b^{2y_1}-b^{y_1}+1 \pmod{b^{3y_1}}.
\]
In particular, $b$ is small as $b<c^{z_2/y_2} \le {c_U}^{2/5}$.
On the other hand, since $a<c$, the second equation in \eqref{sys9} leads to $0<c-b^{y_2/2}<1$.
Therefore,
\[
c=\lfloor b^{y_2/2}\rfloor+1.
\]
These restrictions on the values of $b,c,y_1$ and $y_2$ are so strong that we can verify by a brute force that the first equation in \eqref{sys9} does not hold in any possible cases.
%a moment.

\vspace{0.2cm}
{\it II. Case where $(z_1,z_2)=(2,3)$.}

\vspace{0.2cm}
System \eqref{sys-12} is
\begin{equation} \label{syst1}
a^{x_1}+b^{y_1}=c^2, \ \ a^{x_2}+b^{y_2}=c^3.
\end{equation}
We proceed similarly to case I.

\vspace{0.2cm}
{\it II/(i). Case where $x_1 \ge 4$ or $x_2 \ge 5$.}

\vspace{0.2cm}
Since $a<{c_U}^{3/5}$ by \eqref{syst1}, $a$ is small enough to deal with this case similarly to I/(i).

\vspace{0.2cm}{\it II/(ii). Case where $x_2=4$ and $y_2 \ge 7$}.

\vspace{0.2cm}
Since $a<{c_U}^{3/4}$ and $b<{c_U}^{3/7}$ by \eqref{syst1}, $a$ and $b$ are small enough to deal with this case similarly to I/(ii).

\vspace{0.2cm}
{\it II/(iii). Case where $x_2=4$ and $y_2 \in \{1,2\}$}.

\vspace{0.2cm}
Since $a>b$ and $y_2 \le 2$, it follows from 2nd equation that $c^{3/2}-a^2=\frac{b^{y_2}}{c^{3/2}+a^2}<1$.
Using this inequality, we apply the algorithm described in Lemma \ref{red6} (see below) to deal with this case.

\vspace{0.2cm}
{\it II/(iv). Case where $(x_2=4$ and $y_2 \in \{3,4,5,6\})$ or $x_2=3$.}

\vspace{0.2cm}
This case is handled by applying Lemma \ref{ourcase} to 2nd equation.

\vspace{0.2cm}{\it II/(v). Case where $x_1=3$ and $x_2 \le 2$.}

\vspace{0.2cm}
Since $a$ is relatively small as $a<{c_U}^{2/3}$, and $x_2$ is very small, this case can be handled similarly to case II/(i).

\vspace{0.2cm}{\it II/(vi). Case where $(x_1,x_2) \in \{(1,1), (1,2), (2,1), (2,2)\}$.}

\vspace{0.2cm}
The case where $x_1=1$ or $x_2=1$ can be dealt with by the same algorithm as in case II/(iii).
Finally, assume that $x_1=x_2=2$.
Since $c^2+b^{y_2}=c^3+b^{y_1}$ from system \eqref{syst1}, this case is deal with by similar methods described in Section \ref{Sec-z1=z2}.
\end{proof}

By Proposition \ref{bounds_z1<z2_c^z1=0(4)} together with Lemmas \ref{trivi-lemma} and \ref{small-z2}, we may assume in system \eqref{sys-12} that
\[
c \ge c_1, \quad z_2-z_1=1,  \quad z_2 \ge 4,
\]
where $c_1=\max\{c_0,1000\}$.

The next lemma is an analogue to Lemma \ref{red1} in case where $c^{z_1} \equiv 2 \pmod 4$, and it can be proved almost similarly.

\begin{lem} \label{red5}
Let $(x_1,y_1,z_1,x_2,y_2,z_2)$ be a solution of system \eqref{sys-12}.
Then the following hold.
\begin{itemize}
  \item[$\bullet$]
   If $a \equiv 1 \pmod{4}$ and $b \equiv -1 \pmod{4}$ then both $y_1,y_2$ are odd.
  \item[$\bullet$]
   If $a \equiv -1 \pmod{4}$ and $b \equiv 1 \pmod{4}$ then both $x_1,x_2$ are odd.
  \item[$\bullet$]
   If $a \equiv b \equiv -1 \pmod{4}$, then $x_1 \not\equiv y_1 \pmod{2}$ and $x_2 \not\equiv y_2 \pmod{2}.$
  \item[$\bullet$]
   One of $x_1$ and $y_1$ is odd, and one of $x_2$ and $y_2$ is odd.
  \item[$\bullet$]
   $x_1 < x_2$ or $y_1 < y_2.$
  \item[$\bullet$]
   $(x_1 \ge z_1$ or $y_1 \ge z_1)$ and $(x_2 \ge z_2$ or $y_2 \ge z_2).$ % then \eqref{syst1} has no solutions.
 \item[$\bullet$]
  $x_1y_2 \ne x_2y_1, \ x_1z_2 \ne x_2z_1, \ y_1z_2 \ne y_2z_1.$
 \item[$\bullet$]
  $\min\{x_1,x_2\}<|y_1z_2-y_2z_1|.$
 \item[$\bullet$]
  $(x_1 \ne z_1$ or $y_1 \ge z_1)$ and $(y_1 \ne z_1$ or $x_1 \ge z_1)$ and $(x_2 \ne z_2$ or $y_2 \ge z_2)$ and $(y_2 \ne z_2$ or $x_2 \ge z_2).$
 \item[$\bullet$]
  If $b<11,$ then $(2 \nmid x_1$ or $3 \nmid z_1)$ and $(3 \nmid x_1$ or $2 \nmid z_1)$ and $(2 \nmid x_2$ or $3 \nmid z_2)$ and $(3 \nmid x_2$ or $2 \nmid z_2).$
\end{itemize}
\end{lem}

Finally, using the established lemmas, we further show three lemmas, where the latter two of them together show the contrary to the condition from Lemma \ref{red5} saying that $x_1<x_2$ or $y_1<y_2$ in \eqref{sys-12}.

\begin{lem} \label{red6}
Under the hypothesis of Proposition \ref{bounds_z1<z2_c^z1=0(4)}, if system \eqref{sys-12} has a solution $(x_1,y_1,z_1,x_2,y_2,z_2),$ then
\[
\min\{a^{x_1},b^{y_1}\} \ge c, \quad \min\{a^{x_2},b^{y_2}\} \ge c^2.
\]
\end{lem}

\begin{proof}
First, we illustrate the method to show that $a^{x_1} \ge c$.
Suppose on the contrary that $a^{x_1}<c$.
If $y_1 \le z_1$, then $c>a^{x_1}=c^{z_1}-b^{y_1} \ge c^{z_1}-b^{z_1}>c^{z_1-1}$, so $z_1<2$,  which is absurd as $z_1 \ge 3$.
Thus $y_1>z_1$.
On the other hand, from 1st equation, observe that
\[
c^{z_1/2}-b^{y_1/2}=\frac{a^{x_1}}{c^{z_1/2}+b^{y_1/2}}<\frac{c}{c^{z_1/2}}<1.
\]
Thus
\[
\lceil b^{y_1/z_1} \rceil=:c_2 \le c \le c_3:=\left\lfloor (1+b^{y_1/2})^{2/z_1} \right\rfloor.
\]
Since $y_1>z_1$, it is very often observed that $c_2>c_3$ for given $b,y_1$ and $z_1$.
By Proposition \ref{bounds_z1<z2_c^z1=0(4)}, we have a list of all possible tuples $(\alpha,\beta,z_1,c_u)$, where $c_u$ is the corresponding upper bound for $c$.
For each of those tuples, and for each possible tuple $(b,c,x_1,y_1,x_2,y_2)$ satisfying
\begin{gather*}
z_1<y_1 \le \left\lfloor \frac{\log c_u}{\log b_0}z_1 \right\rfloor, \ b_0 \le b \le \bigr\lfloor{c_u}^{z_1/y_1} \bigr\rfloor, \ \max\{c_1,c_2\} \le c \le \min\{c_3,c_u\},\\
y_2 \le \left\lfloor \frac{\log c}{\log b}z_2 \right\rfloor, \ \ x_1 \le \left\lfloor \frac{\log c}{\log a_0}z_1\right\rfloor, \ \ x_2 \le \left\lfloor \frac{\log c}{\log a_0}z_2 \right\rfloor
\end{gather*}
with $z_2=z_1+1$, we check equation \eqref{final-eq2} does not hold.
Thus the inequality $a^{x_1} \ge c$ holds.
The remaining inequalities can be shown exactly in the same way by changing the roles of $a,b$ and $z_1,z_2$, respectively.
\end{proof}

\begin{lem} \label{red7}
Under the hypothesis of Proposition \ref{bounds_z1<z2_c^z1=0(4)}, if system \eqref{sys-12} has a solution $(x_1,y_1,z_1,x_2,y_2,z_2),$ then $x_1 \ge x_2.$ % or $y_1 \ge y_2$.
\end{lem}

\begin{proof}
Suppose that $x_1<x_2$.
Recall that we may assume that $z_2=z_1+1$.

First, consider the case where $y_1<y_2$.
From \eqref{sys-12}, $a^{x_1}(c-a^{x_2-x_1})=-b^{y_1}(c-b^{y_2-y_1})$.
This implies that
\[
a^{x_1} \mid (c-b^{y_2-y_1}), \quad  b^{y_1} \mid (c-a^{x_2-x_1})
\]
with $(c-b^{y_2-y_1})\,(c-a^{x_2-x_1})<0$.
These together yield that $a^{x_1} \le c-b^{y_2-y_1}$ or $b^{y_1} \le c-a^{x_2-x_1}$, thereby $a^{x_1}<c$ or $b^{y_1}<c$.
However, this contradicts Lemma \ref{red6}.

Second, consider the case where $y_1 \ge y_2$.
From \eqref{sys-12}, $a^{x_1}(a^{x_2-x_1}-c)=b^{y_2}(cb^{y_1-y_2}-1)$ with $a^{x_2-x_1}-c>0$ and $cb^{y_1-y_2}-1>0$.
Since $a^{x_1} \mid (cb^{y_1-y_2}-1)$ and $b^{y_2} \mid (a^{x_2-x_1}-c)$, we have
\[
a^{x_1} < cb^{y_1-y_2}, \quad b^{y_2} < a^{x_2-x_1}.
\]
These together with equation \eqref{final-eq} yield
\begin{align*}
a^{x_2z_1}<(b^{y_1}+cb^{y_1-y_2})^{z_2}=b^{y_1z_2}(1+c/b^{y_2})^{z_2},\\
b^{y_1z_2}<(a^{x_2}+a^{x_2-x_1})^{z_1}=a^{x_2z_1}( 1+1/a^{x_1})^{z_1}.
\end{align*}
Thus
\[
\frac{1}{\bigr( 1+1/a^{x_1} \bigr)^{1/x_2}} \cdot b^{\frac{y_1z_2}{x_2z_1}}<a< (1+c/b^{y_2})^{z_2/(x_2z_1)} \cdot b^{\frac{y_1z_2}{x_2z_1}}.
\]
Since $b<a<c$, and $c^2<b^{y_2}$ by Lemma \ref{red6}, it follows that
\begin{equation}\label{assu8}
\frac{1}{\bigr( 1+1/a_1^{x_1} \bigr)^{1/x_2}} \cdot b^{\frac{y_1z_2}{x_2z_1}} <a<(1+1/c_2)^{z_2/(x_2z_1)} \cdot b^{\frac{y_1z_2}{x_2z_1}},
\end{equation}
where $a_1=\max\{a_0,b+2\}$ and $c_2=\max\{c_1,b+2\}$.

We are now in the position to give the details of our reduction algorithm.

\vspace{0.2cm}{\it Step I.} \
Proposition \ref{bounds_z1<z2_c^z1=0(4)} provides us a list of all possible tuples $[\alpha, \beta, z_2, c_u]$, where $c_u$ is the corresponding upper bound for $c$.
Some elements of the list are ruled out from inequality (i) of Lemma \ref{z1z2_upper} with $(a,b,c)=(a_0,b_0,c_u)$.
We denote this smaller list by $clist$.

In the sequel, we call a pair of integers $[u,v]$ with $u,v \in \{1,-1\}$ the {\it signature} of $[a,b]$ denoted by $s=s([a,b])$ if $[a,b] \equiv [u,v] \pmod{2^{\alpha}}$.
From \eqref{cond}, we know that $s \in \{[-1,-1],[1,-1],[-1,1]\}$.
On the other hand, Lemma \ref{red5}\,(i,ii,iii) shows that, if for instance, a tuple $[x_1,y_1,x_2,y_2]$ is a solution of \eqref{sys-12} with $[x_1,y_1,x_2,y_2] \equiv [0,1,1,0] \pmod{2}$, then in this case we necessarily have $s([a,b])=[-1,-1]$.
On distinguishing between the 16 possible cases of $[x_1,y_1,x_2,y_2]$ according to the parities of $x_1,y_1,x_2$ and $y_2$, between the $3$ cases of possible signatures of $[a,b]$ and using Lemma \ref{red5}\,(i,ii,iii), we can assign for each tuple $[x_1,y_1,x_2,y_2]$ the corresponding signatures of $[a,b]$.
This way we can rule out $36$ cases of the total of $16 \times 3=48$ cases, and we obtain a list of possible parities and signatures denoted by $parsig$.
The elements of $parsig$ are of the form $\bigr[px_1,py_1,px_2,py_2,[s_a,s_b]\bigr]$, where, for $i=1,2$, we write $px_i,py_i=1$ or $2$ according to whether $x_i,y_i$ are odd or even, respectively.
Moreover, $[s_a,s_b]$ denotes the corresponding signatures of $[a,b]$.
$parsig$ is explicitly given as follows:
\begin{align*}
parsig=\Bigr[ & \bigr[2,1,2,1,[-1,-1]\bigr],\bigr[2,1,1,2,[-1,-1]\bigr],\bigr[1,2,1,2,[-1,-1]\bigr],\\
&\bigr[1,2,2,1,[-1,-1]\bigr],\bigr[2,1,2,1,[1,-1]\bigr], \bigr[2,1,1,1,[1,-1]\bigr],\\
&\bigr[1,1,2,1,[1,-1]\bigr], \bigr[1,1,1,1,[1,-1]\bigr],\bigr[1,2,1,2,[-1,1]\bigr], \\
&\bigr[1,2,1,1,[-1,1]\bigr],\bigr[1,1,1,2,[-1,1]\bigr], \bigr[1,1,1,1,[-1,1]\bigr] \Bigr].
\end{align*}
Now, for each element in $clist$ and each element in $parsig$, we use Lemma \ref{ourcase} together with Lemmas \ref{red5} and \ref{red6} (see also Remark \ref{rem1} below) to sieve considerably the possible solutions $[x_1,y_1,z_1,x_2,y_2,z_2]$ of system \eqref{sys-12}.
This way we obtain a list named {\it list1} having elements of the form
$$[\alpha,\beta, x_1,y_1,z_1,x_2,y_2,z_2,c_u,b_{max},[s_a,s_b]\bigr],$$
where $[s_a,s_b]$ denotes the signature of $[a,b]$ and $b_{max}$ is defined as
\[
b_{max}:=\min\bigr\{c_u, \lfloor {c_u}^{z_1/x_1} \rfloor, \lfloor {c_u}^{z_1/y_1} \rfloor, \lfloor {c_u}^{z_2/x_2} \rfloor, \lfloor {c_u}^{z_2/y_2} \rfloor\bigr\}.
\]
The above algorithm for generating $list1$ is given by the following program.

\vspace{0.2cm}{\tt
begin

\hskip.2cm for each element of $clist$ do

\hskip.2cm $z_1:=z_2-1$

\hskip.2cm for each element of $parsig$ do

\hskip.2cm for $x_1:=px_1$ to $\left\lfloor z_1(\log c_u)/\log a_0 \right\rfloor$ by 2 do

\hskip.2cm for $y_1:=py_1$ to $\left\lfloor z_1(\log c_u)/\log b_0 \right\rfloor$ by 2 do

\hskip.2cm for $x_2:=px_2$ to $\left\lfloor z_2(\log c_u)/\log a_0 \right\rfloor$ by 2 do

\hskip.2cm for $y_2:=py_2$ to $\left\lfloor z_2(\log c_u)/\log b_0 \right\rfloor$ by 2 do

\hskip.2cm if $x_1y_2-x_2y_1$ mod $2^{\beta z_1-\alpha}=0$ then

\hskip.4cm sieve using Lemmas \ref{red5} and \ref{red6}

\hskip.4cm $b_{max}:=\min\bigr\{c_u, \lfloor {c_u}^{z_1/x_1} \rfloor, \lfloor {c_u}^{z_1/y_1} \rfloor, \lfloor {c_u}^{z_2/x_2} \rfloor, \lfloor {c_u}^{z_2/y_2} \rfloor\bigr\}$

\hskip.4cm put the result $\bigr[\alpha, \beta,x_1,y_1,z_1,x_2,y_2,z_2,c_u,b_{max},[s_a,s_b]\bigr]$ into $list1$

end}

\vspace{0.2cm}{\it Step II.} \
In order to create $list2$ composed of all possible tuples $[a,b,x_1,y_1,z_1,x_2,y_2,z_2]$, by using inequalities \eqref{assu8}, we proceed as follows.

\vspace{0.2cm}{\tt
begin

\hskip.2cm for each element of $list1$ do

\hskip.2cm $T_b:=\left\lceil (b_0-s_b)/2^\alpha \right\rceil$

\hskip.4cm for $b:=T_b \cdot 2^{\alpha}+s_b$ to $b_{max}$ by $2^{\alpha}$ do

\hskip.4cm $a_{min}:=\max\Bigl\{a_1,\Bigl\lceil \bigr(1+1/a_1^{x_1}\bigr)^{-1/x_2} \cdot b^{\frac{y_1z_2}{x_2z_1}} \Bigl\rceil \Bigl\}$

\hskip.4cm $a_{max}:=\min\Bigl\{c_u, \lfloor {c_u}^{z_1/x_1} \rfloor, \lfloor {c_u}^{z_2/x_2} \rfloor, \Bigl\lfloor \left(1+1/c_2\right)^{z_2/(x_2z_1)} \cdot b^{\frac{y_1z_2}{x_2z_1}} \Bigl\rfloor \Bigl\}$

\hskip.4cm $T_a:=\left\lceil (a_{min}-s_a)/2^\alpha \right\rceil$

\hskip.4cm for $a:=T_a \cdot 2^{\alpha}+s_a$ to $a_{max}$ by $2^{\alpha}$ do

\hskip.6cm test whether equation \eqref{final-eq} holds or not

\hskip.6cm put the result $[a,b,x_1,y_1,z_1,x_2,y_2,z_2]$ into the $list2$

end}

\vspace{0.2cm}\noindent
It turned out that $list2$ is empty.

Finally, we mention that the restriction from Lemma \ref{red5}\,(ix) was very efficient for the case where $b \le 7$.
\end{proof}

\begin{rem} \label{rem1}
\rm Through our program implemented in the proof of Lemma \ref{red7}, we may assume by Lemma \ref{red6} that $\min\{x_1,x_2\}=x_1 \ge 2$ and $\min\{y_1,y_2\}=y_2 \ge 3$.
On the one hand, for generating $list1$ in {\it Step I}, we combined that information with Lemma \ref{red5}.
This way we excluded a lot of candidates from our $list1$ since the number of tuples satisfying $x_1=1$ or $y_2 \le 2$ is high.
On the other hand, the second advantage is that $a_{min}$ in {\it Step II} becomes larger as $x_1$ increases.
\end{rem}

\begin{lem} \label{red9}
Under the hypothesis of Proposition \ref{bounds_z1<z2_c^z1=0(4)}, if system \eqref{sys-12} has a solution $(x_1,y_1,z_1,x_2,y_2,z_2),$ then $y_1 \ge y_2.$
\end{lem}

\begin{proof}
We may assume that $x_1 \ge x_2$ by Lemma \ref{red7}.
Suppose on the contrary that $y_1<y_2$.
Starting with these two inequalities, we can proceed similarly to Lemma \ref{red7}.
Thus we just indicate the key points on the implemented algorithms.
First, we generate the list $list1$ exactly in the same way as in {\it Step I} of Lemma \ref{red7}.
Second, we closely follow the method of {\it Step II} of Lemma \ref{red7}, where the only difference arises from the fact that $\min\{y_1,y_2\}=y_1$ and $\min\{x_1,x_2\}=x_2$.
Namely, system \eqref{sys-12} with $z_2-z_1=1$ implies that $b^{y_1}(b^{y_2-y_1}-c)=a^{x_2}(ca^{x_1-x_2}-1)$, whence $b^{y_1}<ca^{x_1-x_2}$ and $a^{x_2}<b^{y_2-y_1}$.
These together with Lemma \ref{red6} yield
\[
\frac{1}{\bigr(1+1/a_1^{x_2}\bigr)^{1/x_1}} \cdot b^{\frac{y_2z_1}{x_1z_2}} <a<(1+1/c_2)^{z_1/(x_1z_2)} \cdot b^{\frac{y_2z_1}{x_1z_2}}.
\]
We can proceed exactly in the same way as in Lemma \ref{red7} by using the corresponding parameters $a_{min}$ and $a_{max}$ indicated by the above inequalities.

\end{proof}
The total computational time in Subsection \ref{Sec-sieve-z1<z2_c=max_c^z1=2(4)_c^z1=0(4)} did not exceed 6 hours.

%%%%%%%%%%%%%%%%%%%%%%
\section{Case where $z_1<z_2$ and $c<\max\{a,b\}$: Sieving}
\label{Sec-sieve-z1<z2_a=max}
%%%%%%%%%%%%%%%%%%%%%%

The aim of this section is to show that there is no solution of system \eqref{sys-12} fulfilling the statements of Propositions \ref{bounds_z1<z2_a=max_xi<xj}, \ref{bounds_z1<z2_a=max_xi=xj_k=3} and \ref{bounds_xi<xj_z1<z2_a=max_xi=xj_k<>3}, respectively.
However, the case under the hypothesis of Proposition \ref{bounds_z1<z2_a=max_xi=xj_k=3} can be handled similarly to Section \ref{Sec-z1=z2}, since the system is reduced to the equation $c^{z_1}+b^{y_2}=c^{z_2}+b^{y_1}$ and both $b,c$ are relatively small.

It suffices to consider the case where $a>\max\{b,c\}$, and we put
\[
a_0=\max\{19,2^\alpha+1,3 \cdot 2^\beta+1\}, \ b_0=2^\alpha-1, \ c_0=3 \cdot 2^\beta.
\]
These numbers are lower bounds for $a,b$ and $c$, respectively.
We can use both equations \eqref{final-eq} and \eqref{final-eq2}.
Moreover, since $z_1 \ge 2$ as $a>c$, we have $c^{z_1} \equiv 0 \pmod{4}$, so Lemma \ref{red5} can be used.
The restrictions from Lemmas \ref{ourcase} and \ref{red5} will be used several times in our reduction procedure without any further explicit reference.

We proceed in two cases according to Proposition \ref{bounds_z1<z2_a=max_xi<xj} or Proposition \ref{bounds_xi<xj_z1<z2_a=max_xi=xj_k<>3}.

\subsection{On the system under the hypothesis of Proposition \ref{bounds_z1<z2_a=max_xi<xj}}\label{prop12.1}
%%%%%%%%%%%%%%%%%%%%%%%%%%%%%%%%%%%%%%%%%
Proposition \ref{bounds_z1<z2_a=max_xi<xj} provides us all possible tuples $(\alpha,\beta,  z_1,z_2,a_u,c_u)$, where $a_u$ and $c_u$ are the corresponding upper bounds for $c$ and $a$, respectively.
We put these data in the list named $clist$. %, i.e., $clist=[\alpha,\beta,z_1,z_2,c_u,a_u]$.
In the sequel, we proceed in two cases according to whether $d_z=1$ on not.

%%%%%%%%%%%%%%%%%%
\subsubsection{Case where $d_z>1$}
%%%%%%%%%%%%%%%%%%

Note that $2 \le d_z \le 6$ and each $c_u$ is very small.
The order of magnitude of $c_u$ is between $6$ and $802$, where smaller values occur for larger $d_z$'s.

We begin with the following lemma.

\begin{lem}\label{z1 ge 2dz}
Under the hypothesis of Proposition \ref{bounds_z1<z2_a=max_xi<xj}, if system \eqref{sys-12} has a solution $(x_1,y_1,z_1,x_2,y_2,z_2),$ then $z_1 \ge 2d_z.$
\end{lem}

\begin{proof}
Suppose on the contrary that $z_1<2d_z$.
Since $z_2-z_1=d_z \in \{2,3,4,5,6\}$, the possible pairs $[z_1,z_2]$ are given as follows:
\begin{align*}
[z_1,z_2] \in \bigr[ &[2,4],[3,5],[2,5],[3,6],[4,7],[5,8],[2,6],[3,7],[4,8],[5,9],\\
&[6,10],[7,11],[2,7],[3,8], [4,9],[5,10],[6,11],[7,12],[8,13],\\
&[9,14],[2,8], [3,9],[4,10],[5,11],[6,12],[7,13],[8,14],[9,15],\\
&[10,16],[11,17]\bigr].%:=zlist.
\end{align*}
We set $zlist$ as the list composed of these pairs.
If $b>c$, then $\max\{x_i,y_i\}<z_i$ for $i \in \{1,2\}$.
Now, for each element of $clist$ and for $x_i,y_i$ in that ranges, we use the restrictions from Lemmas \ref{ourcase} and \ref{red5} together with \eqref{z2} for $t=1$ to generate a list named $list1$ of the form $[\alpha,\beta,x_1,y_1,z_1,x_2,y_2,z_2,a_u,c_u]$, as follows:
%Namely, we do the following:

\vspace{0.2cm}{\tt
for each element of $clist$ do

\hskip.2cm for $x_1:=1$ to $z_1-1$ do

\hskip.2cm for $y_1:=1$ to $z_1-1$ do

\hskip.2cm for $x_2:=1$ to $z_2-1$ do

\hskip.2cm for $y_2:=1$ to $z_2-1$ do

\hskip.2cm if $x_1y_2-x_2y_1 \mod{2^{\beta z_1-\alpha}}=0$

\hskip.4cm sieve using Lemmas \ref{ourcase} and \ref{red5} and put the result into

\hskip.4cm $list1$

end}

\vspace{0.2cm}\noindent
Note that once $list1$ is generated, we have not only a list of possible solutions $x_i,y_i,z_i$ of \eqref{sys-12} but also upper bounds for $a,b,c$, as well (i.e., $b<a \le a_u$ and $c \le c_u$).
Using these bounds we basically check for each possible case whether equation \eqref{final-eq2} holds or not.
%By putting $si:=[-1,1]$,
We proceed as follows.

\vspace{0.2cm}{\tt
for each element of $list1$ do

\hskip.2cm for $c:=c_0$ to $c_u$ by $2^{\beta+1}$ do

\hskip.2cm for $s$ in $[-1,1]$ do

%\hskip.2cm $T_b:=\left\lceil(c+1-s)/2^{\alpha}\right\rceil$

\hskip.2cm for $b:=\left\lceil(\max\{b_0,c+1\}-s)/2^{\alpha}\right\rceil \cdot 2^{\alpha}+s$ to $a_u$ by $2^{\alpha}$ do

\hskip.2cm if $(c^{z_1}-b^{y_1}>0)$ and $(c^{z_2}-b^{y_2}>0)$ then

\hskip.2cm if equation \eqref{final-eq2} holds
% and $(c^{z_1}-b^{y_1})^{1/x_1}=(c^{z_2}-b^{y_2})^{1/x_2}$
then

\hskip.2cm $a':=(c^{z_1}-b^{y_1})^{1/x_1}$

\hskip.2cm if $a'$ is an integer and $\gcd(a',b,c)=1$ then put the values

\hskip.2cm of $[a',b,c,x_1,y_1,z_1,x_2,y_2,z_2]$ into $list2$

end}

\vspace{0.2cm}\noindent
We implemented the above algorithm and it turned out that $list2$ is empty.
If $b<c$, then we can use the inequality $b<c \le c_u$ to proceed exactly in the same way as above.
\end{proof}

The next lemma is an analogue to Lemma \ref{red6}.

\begin{lem} \label{red6dge2}
Under the hypothesis of Proposition \ref{bounds_z1<z2_a=max_xi<xj}, if system \eqref{sys-12} has a solution $(x_1,y_1,z_1,x_2,y_2,z_2),$ then
\[
\min\{a^{x_1},b^{y_1}\} \ge c^{d_z}, \quad \min\{a^{x_2},b^{y_2}\} \ge c^{d_z+1}.
\]
\end{lem}

\begin{proof}
The proof is similar to that of Lemma \ref{red6}.
We only show that $a^{x_1} \ge c^{d_z}$, since the treatment of the remaining inequalities is similar.
By Lemma \ref{z1 ge 2dz}, we may assume that $z_1 \ge 2d_z$.
Suppose on the contrary that $a^{x_1}<c^{d_z}$.
Since $a^{x_1}<c^{d_z}<c^{z_1/2}$, from 1st equation, we have $c^{z_1/2}-b^{y_1/2}=\frac{a^{x_1}}{c^{z_1/2}+b^{y_1/2}}<1$, thereby
\begin{equation}\label{bcy2}
\bigl\lceil {(c^{z_1/2}-1)}^{2/y_1} \bigl\rceil=:b_1 \le b \le b_2:=\bigl\lfloor c^{z_1/y_1} \bigl\rfloor.
\end{equation}
Recall that the bounds $c \le c_u$ in $clist$ are (very) sharp.
Moreover, on combining this information with inequalities \eqref{bcy2}, it very often holds that $b_1>b_2$ for given $c,y_1$ and $z_1$.
We construct a list named $list1$ consisting of elements of the form $[b,c,y_1,z_1,z_2]$.
We proceed as follows.

\vspace{0.2cm}{\tt
for each element of $clist$ do

\hskip.2cm for $c:=c_0$ to $c_u$ by $2^{\beta+1}$ do

%\hskip.2cm for $z_1:=2d_z$ to $22-d_z$ do

\hskip.2cm for $y_1:=1$ to $\left \lfloor z_1(\log c)/\log b_0 \right\rfloor$ do

\hskip.2cm $b_{min}:=\max\{b_0,b_1\}$ and $b_{max}:=\min\{\lfloor c^{z_1/y_1} \rfloor, b_2,a_u-2\}$

\hskip.4cm for $s$ in $[-1,1]$ do

\hskip.4cm $T_b:=\left\lceil(b_{min}-s) /2^{\alpha} \right\rceil$

\hskip.4cm for $b:=T_b \cdot 2^{\alpha}+s$ to $b_{max}$ by $2^{\alpha}$ do

\hskip.6cm put $[b,c,y_1,z_1,z_2]$ into $list1$

end}

\vspace{0.2cm}\noindent
Since $z_1 \ge 2d_z$ and $a^{x_1}<c^{d_z}$, we observe that $x_1<\min\{d_z,\lfloor z_1 /2\rfloor\}$.
Finally, using $list1$ and the above range for $x_1$, we basically check whether equation \eqref{final-eq2} holds or not.
All possible solutions will be put in the list named $list2$, as follows.

\vspace{0.2cm}{\tt
for each element of $list1$ do

\hskip.2cm for $x_1:=1$ to $\min\{d_z,\lfloor z_1 /2\rfloor\}-1$ do

\hskip.2cm for $x_2:=1$ to $z_2-1$ do

\hskip.2cm for $y_2:=1$ to $\left\lfloor z_2(\log c)/\log b \right\rfloor$ do

\hskip.2cm if $(c^{z_2}-b^{y_2}>0)$ and $(c^{z_1}-b^{y_1}>0)$ then

\hskip.2cm if equation \eqref{final-eq2} holds then

\hskip.2cm $a':=(c^{z_1}-b^{y_1})^{1/x_1}$% and $a_2:=(c^{z_2}-b^{y_2})^{1/x_2}$

\hskip.4cm if $(a'>\max\{b,c\})$ and $\gcd(a',b,c)=1$ and $({a'}^{x_1}<c^{d_z})$ then

\hskip.6cm put $[a',b,c,x_1,y_1,z_1,x_2,y_2,z_2]$ into $list2$

end}

\vspace{0.2cm}\noindent
It turned out that $list2$ is empty.
\end{proof}

In the following two lemmas together, we show the contrary to the condition from Lemma \ref{red5}\,(iv) saying that $x_1<x_2$ or $y_1<y_2$ in \eqref{sys-12}.

\begin{lem} \label{red14}
Under the hypothesis of Proposition \ref{bounds_z1<z2_a=max_xi<xj}, if system \eqref{sys-12} has a solution $(x_1, y_1, z_1, x_2, y_2, z_2),$ then $x_1 \ge x_2.$
\end{lem}

\begin{proof}
We closely follow the method described in Lemma \ref{red7}.
Suppose on the contrary that $x_1<x_2$.

If $y_1<y_2$, then $a^{x_1}(c^{d_z}-a^{x_2-x_1})=-b^{y_1}(c^{d_z}-b^{y_2-y_1})$.
This equation implies that $\min\{a^{x_1},b^{y_1}\}<c^{d_z}$, which however contradicts Lemma \ref{red6dge2}.
Suppose that $y_1 \ge y_2$.
Then $a^{x_1}(a^{x_2-x_1}-c^{d_z})=b^{y_2}(c^{d_z}b^{y_1-y_2}-1)$, and this implies that
\[
a^{x_1} \mid (c^{d_z} b^{y_1-y_2}-1), \quad b^{y_2} \mid (a^{x_2-x_1}-c^{d_z})
\]
with $c^{d_z} b^{y_1-y_2}-1>0$ and $a^{x_2-x_1}-c^{d_z}>0$.
Thus
\[
a^{x_1} < c^{d_z}b^{y_1-y_2}, \quad b^{y_2} < a^{x_2-x_1}.
\]
These inequalities together with equation \eqref{final-eq} yield
\begin{gather*}
a^{x_2z_1}<(b^{y_1}+c^{d_z}b^{y_1-y_2})^{z_2}=b^{y_1z_2}(1+c^{d_z}/b^{y_2})^{z_2},\\
b^{y_1z_2}<(a^{x_2}+a^{x_2-x_1})^{z_1}=a^{x_2z_1}(1+1/a^{x_1})^{z_1}.
\end{gather*}
Since $c^{d_z+1}<b^{y_2}$ by Lemma \ref{red6dge2}, it follows that
\begin{equation}\label{assu14}
\frac{1}{ \bigr(1+1/{a_1}^{x_1}\bigr)^{1/x_2} } \cdot b^{\frac{y_1z_2}{x_2z_1}}<a<(1+1/c_0)^{z_2/x_2z_1} \cdot b^{\frac{y_1z_2}{x_2z_1}}
\end{equation}
with $a_1=\max\{a_0,b+2\}$.

We are now in the position to give the details of our reduction algorithm in this case.
We proceed exactly in the same way as in Lemma \ref{red7} with some appropriate modifications.
Some elements in $clist$ can be ruled out by Lemma \ref{z1z2_upper}\,(i), and we denote this smaller list by $clist0$.

\vspace{0.2cm}{\it Step I.} \
We follow %the algorithm described in
{\it Step I} of Lemma \ref{red7} with the following modifications.
By using $clist0$, the list $parsig$ defined in Lemma \ref{red7} and the inequalities $x_i<z_i$ for $i \in \{1,2\}$, we generate a list named $list1$ containing elements of the form $\bigr[\alpha,\beta,x_1,y_1,z_1,x_2,y_2,z_2,a_u,b_{max},[s_a,s_b]\bigr]$, where $[s_a,s_b]$ denotes the signature of $[a,b]$, and $b_{max}$ is defined by
\[
b_{max}:=\min\bigr\{a_u, \lfloor {c_u}^{z_1/x_1} \rfloor, \lfloor {c_u}^{z_1/y_1} \rfloor, \lfloor {c_u}^{z_2/x_2} \rfloor, \lfloor {c_u}^{z_2/y_2} \rfloor\bigr\}.
\]

\vspace{0.2cm}{\it Step II.} \
We follow {\it Step II} of Lemma \ref{red7} with a single modification in the bound $a_{max}$ for $a$ according to \eqref{assu14}, namely:
\[
a_{max}:=\min\Bigr\{a_u, \lfloor {c_u}^{z_1/x_1} \rfloor, \lfloor {c_u}^{z_2/x_2} \rfloor, \bigl\lfloor (1+1/c_0)^{z_2/(x_2z_1)} \cdot b^{\frac{y_1z_2}{x_2z_1}} \bigl\rfloor \Bigr\}.
\]
Using the program occurring in {\it Step II} of Lemma \ref{red7} with the above bounds we check that equation \eqref{final-eq} holds in no case.
\end{proof}

\begin{lem} \label{red15}
Under the hypothesis of Proposition \ref{bounds_z1<z2_a=max_xi<xj}, if system \eqref{sys-12} has a solution $(x_1, y_1, z_1, x_2, y_2, z_2), $ then $y_1 \ge y_2.$
\end{lem}

\begin{proof}
Since the method of the proof and the resulting algorithm are similar to the ones presented in Lemma \ref{red14}, we only indicate the key points of the algorithm.
By Lemma \ref{red14} we may assume that $x_1 \ge x_2$, and suppose on the contrary that $y_1<y_2$.
In {\it Step I}, we generate $list1$ exactly in the same way as in Lemma \ref{red14}.
In {\it Step II}, we closely follow the method of Lemma \ref{red14}, where the only difference comes from the fact that $\min\{y_1,y_2\}=y_1$ and $\min\{x_1,x_2\}=x_2$.
Namely, in this case, we deduce from \eqref{sys-12} that
\if0
with $z_2-z_1=d_z$ implies that
\[
b^{y_1}(b^{y_2-y_1}-c^{d_z})=a^{x_2}(c^{d_z} a^{x_1-x_2}-1),
\]
whence $b^{y_1}<c^{d_z}a^{x_1-x_2}$ and $a^{x_2}<b^{y_2-y_1}$.
This way we see that the corresponding inequality of \eqref{assu13} in this case is
\[
\frac{b^{\frac{y_2z_1}{x_1z_2}}}{ \bigr(1+c^{d_z}/a^{x_2}\bigr)^{1/x_1} }<a<\bigr(1+1/b^{y_1}\bigr)^{z_1/(x_1z_2)} \cdot b^{\frac{y_2z_1}{x_1z_2}},
\]
This together with Lemma \ref{red6dge2} gives
\fi
\[
\frac{1}{\bigr(1+1/c_0\bigr)^{1/x_1}} \cdot b^{\frac{y_2z_1}{x_1z_2}}<a<(1+1/b^{y_1})^{z_1/(x_1z_2)} \cdot b^{\frac{y_2z_1}{x_1z_2}}.
\]
Using these inequalities, we can proceed exactly in the same way as in Lemma \ref{red14} with the corresponding changes of $a_{min}$ and $a_{max}$.
%The total running time of this case was $1598$ seconds which is about $26.64$ minutes.
\end{proof}

%In what follows, we deal with system \eqref{sys-12} satisfying the statements of Proposition \ref{bounds_z1<z2_a=max_xi<xj} with $d_z=z_2-z_1=1$.

\subsubsection{Case where $d_z=1$}
We begin with the following lemma, which can be regarded as an analogue of Lemma \ref{small-z2} in the case where $a>\max\{b,c\}$ with $(z_1,z_2)=(2,3)$.

\begin{lem} \label{small-z23}
Under the hypothesis of Proposition \ref{bounds_z1<z2_a=max_xi<xj}, system \eqref{sys-12} has no solution $(x_1, y_1, z_1, x_2, y_2, z_2)$ satisfying $(z_1,z_2)=(2,3).$
\end{lem}

\begin{proof}
Since $a>c$ and $a^{x_i}<c^{z_i}$ for $i \in \{1,2\}$, we have $x_1=1$ and $x_2 \le 2$.
Note that the case where $x_2=1$ is reduced to the equation $c^2+b^{y_2}=c^3+b^{y_1}$ with both $b,c$ suitably small, so we only consider the case where $x_2=2$.
Then system \eqref{sys-12} is
\begin{equation} \label{sys-Lem_small-z23}
a+b^{y_1}=c^2, \ \ a^2+b^{y_2}=c^3.
\end{equation}
Note that $a_u \le a_U:=4.5 \cdot 10^6$ in any element in $clist$.

In \eqref{sys-Lem_small-z23}, suppose that $y_1 \le y_2$.
We take the equations modulo $b^{y_1}$ to see that $b^{y_1} \mid (a-c)$, so $b^{y_1}<a$.
This together with the equations implies that $c^2<2a<2c^{3/2}$, yielding a contradiction.
Thus we may assume that $y_1>y_2$.
From \eqref{sys-Lem_small-z23}, $a(a-c)=b^{y_2}(cb^{y_1-y_2}-1)$, leading to $b^{y_2}<a$.
Since $a^2>c^3/2$ by the second equation in \eqref{sys-Lem_small-z23}, we have
\begin{equation}\label{syst-a4}
c<2^{1/3}{a_U}^{2/3}.
\end{equation}
If $y_1 \ge 4$, then the first equation in \eqref{sys-Lem_small-z23} together with \eqref{syst-a4} implies that $b<2^{1/6}{a_U}^{2/6}$.
By this estimate of $b$ and \eqref{syst-a4}, both $b$ and $c$ are small enough to check by a brute force that equation \eqref{final-eq2} does not hold in any case.
Finally, suppose that $y_1 \le 3$.
Since $y_1>y_2$, we use Lemma \ref{red5} to see that $(y_1,y_2)=(3,1)$.
In this case, $b \mid (c-1)$, in particular, $b<2^{1/3}{a_U}^{2/3}$ by \eqref{syst-a4}.
Then we can deal with this case by a brute force similarly to the previous case.
\end{proof}

By Lemma \ref{small-z23}, in what follows we may assume that $z_2 \ge 4$.

The next lemma can be regarded as a common analogue of Lemma \ref{red6} and of Lemma \ref{red6dge2} in the case $d_z=1$.

\begin{lem} \label{red16}
Under the hypothesis of Proposition \ref{bounds_z1<z2_a=max_xi<xj}, if system \eqref{sys-12} has a solution $(x_1,y_1,z_1,x_2,y_2,z_2)$ with $d_z=1,$ then $\min\{a^{x_2},b^{y_2}\} \ge c^2.$
\end{lem}

\begin{proof}
We only show that $a^{x_2} \ge c^2$ since the treatment of the inequality $b^{y_2} \ge c^2$ is similar.
Suppose on the contrary that $a^{x_2}<c^2$.
Similarly to the proof of Lemma \ref{red6}, we can use 2nd equation to see that
\[
y_2>z_2, \quad b<{a_u}^{z_2/y_2}, \quad \bigr\lceil b^{y_2/z_2}\bigr\rceil \le c \le \left\lfloor (b^{y_2/2}+1)^{2/z_2} \right\rfloor.
\]
The details of the algorithm to create the list named $list1$ including all possible tuples $[b,c,z_1,y_2,z_2]$ are given as follows.

\vspace{0.2cm}{\tt
for each element of $clist$ do

\hskip.2cm for $y_2:=1$ to $\left\lfloor z_2(\log a_u)/ \log b_0 \right\rfloor$ do

\if0 \hskip.2cm if $y_2<z_2$ then

\hskip.4cm if $\alpha<7$ then

\hskip.6cm $c_u:=\left\lfloor\frac{1}{\left(1-1/6^{z_2-2}\right)^{1/z_2}}{a_u}^{y_2/z_2}\right\rfloor$

\hskip.6cm for $c:=c_0$ to $c_u$ by $2^{\beta+1}$ do

\hskip.6cm $b_{min}:=\left\lceil (c^{z_2/2}-1)^{2/y_2}\right\rceil$ and $b_{max}:=\left\lfloor c^{z_2/y_2} \right\rfloor$

\hskip.8cm for each element $s$ of the list $si$ do

\hskip.8cm for $b:=\left\lceil(b_{min}-s)/{2^\alpha}\right\rceil \cdot 2^\alpha+s$ to $\min\{a_u,b_{max}\}$ by $2^\alpha$ do

\hskip.1cm  put the data $[b,c,z_1,y_2,z_2]$ into $list1$

\hskip.4cm if $\alpha \ge 7$ then

\hskip.6cm for each element $s$ of the list $si$ do

\hskip.6cm for $b:=\left\lceil(b_0-s)/{2^\alpha} \right\rceil \cdot 2^\alpha+s$ to $a_u$ by $2^\alpha$ do

\hskip.6cm $c_{min}:=\lceil b^{y_2/z_2}\rceil$ and $c_{max}:=\lfloor (b^{y_2/2}+1)^{2/z_2} \rfloor$

\hskip.8cm for $c:=2 \cdot \lceil c_{min}/2 \rceil$ to $c_{max}$ by $2$ do

\hskip1cm put the data $[b,c,z_1,y_2,z_2]$ into $list1$

%\hskip.8cm if $c_{min} \equiv 1 \pmod{2}$ then
%\hskip.8cm for $c:=c_{min}+1$ to $\min\{a_u,c_{max}\}$ by $2$ do
%\hskip1cm put the data $[b,y_2,c,z_2,z_1]$ into $list1$

\hskip.2cm if $y_2>z_2$ then
\fi

\hskip.4cm for each $s$ in $[-1,1]$ do

\hskip.4cm for $b:=\left\lceil(b_0-s)/{2^\alpha} \right\rceil \cdot 2^\alpha+s$ to $\bigr\lfloor {a_u}^{z_2/y_2} \bigr\rfloor$ by $2^\alpha$ do

\hskip.6cm $c_{min}:=\max\bigr\{ c_0, \lceil b^{y_2/z_2}\rceil \bigr\}$ and $c_{max}:=\max\bigr\{c_u, \lfloor (b^{y_2/2}+1)^{2/z_2} \rfloor \bigr\}$

\hskip.8cm for $c:=\lceil (c_{min}-2^\beta)/2^{\beta+1} \rceil \cdot 2^{\beta+1}  +2^\beta$ to $c_{max}$ by $2^{\beta+1}$ do

\hskip1cm put the data $[b,c,z_1,y_2,z_2]$ into $list1$

end}

\vspace{0.2cm}\noindent
Finally, for each element of $list1$ and each possible tuple $(x_1,y_1,x_2)$ satisfying $x_1<z_1, y_1<\frac{\log c}{\log b}z_1$ and $x_2<z_2$, we check equation \eqref{final-eq} does not.
\end{proof}

We finish this subsection by the following lemma giving the contrary to an assertion in Lemma \ref{red5}.

\begin{lem} \label{red20}
Under the hypothesis of Proposition \ref{bounds_z1<z2_a=max_xi<xj}, if system \eqref{sys-12} has a solution $(x_1,y_1,z_1,x_2,y_2,z_2)$ with $d_z=1,$ then $x_1 \ge x_2$ and $y_1 \ge y_2.$
\end{lem}

\begin{proof}
This can be proved on the lines of the proofs of Lemmas \ref{red7} and \ref{red9}.
We just note that the resulting algorithm is very similar to that of Lemma \ref{red7} with appropriate modifications on the parameters $a_{max}, a_{min}$ and $b_{max}$.
Namely, in this case we could only use $c_0$ as a uniform lower bound for $c$ (instead of the lower bound $c_1$ in the case where $c>a>b$ with $c^{z_1} \equiv 0 \pmod{4}$).
Therefore the running times of our programs were somewhat slower but they run safely.
\end{proof}

The total computational time for Subsection \ref{prop12.1} did not exceed 6 hours, where the most time consuming part was Lemma \ref{red20}.

\subsection{On the system under the hypothesis of Proposition \ref{bounds_xi<xj_z1<z2_a=max_xi=xj_k<>3}}
%%%%%%%%%%%%%%%%%%%%%%%%%%%%%%%%%%%%%%%%%%%%%%

Proposition \ref{bounds_xi<xj_z1<z2_a=max_xi=xj_k<>3} provides us with some upper bounds and possible solutions of system \eqref{sys-12} which are classified in three cases denoted by (i)-(iii).
We present our reduction algorithms in each of these cases.

\vspace{0.2cm} (i)
We are in the case where
\[
d_z \in \{2,3,4,5\}, \quad b^{d_y}<c^{d_z} \le c_U:=2.4\cdot10^5 %:=238328
\]
with $d_y=y_2-y_1$.
Further, we have a list of all possible tuples $[\alpha,\beta,z_1,z_2,c_u]$, where $c_u$ is the corresponding (sharp) upper bound for $c$.
By applying the same method as in {\it Step I} of Lemma \ref{red7}, we generate a list named $list1$ having elements of the form
\[
[\alpha,\beta,x_1,y_1,z_1,x_2,y_2,z_2,b_{max},c_{max},s_a,s_b],
\]
where $[s_a,s_b]$ is the signature of the pair $[a,b]$, $c_{max}=c_u$ and $b_{max}$ is an upper bound for $b$ defined as
\[
b_{max}:=\min\Bigr\{\bigl \lfloor {c_U}^{1/(y_2-y_1)} \bigl \rfloor, \lfloor {c_u}^{z_1/y_1}\rfloor, \left\lfloor {c_u}^{z_1/x_1}\right\rfloor, \lfloor {c_u}^{z_2/y_2}\rfloor, \lfloor {c_u}^{z_2/x_2}\rfloor \Bigr\}.
\]

\noindent Finally, for each element of $list1$, we loop through the values of $c:=c_0$ to $c_{max}$ by $2^{\beta+1}$ and the values of $b:=\left\lceil(b_0-s_b)/2^{\alpha}\right\rceil\cdot 2^{\alpha}+s_b$ to $b_{max}$ by $2^{\alpha}$, to verify that equation \eqref{final-eq2} does not hold in any case.

\vspace{0.2cm} (ii)
We are in the case where
\[
\max\{a^{\min\{x_1,x_2\}},c^{d_z}\}<b^{d_y} \le b_U:=5.4 \cdot 10^5 %:=531441
\]
with $d_y=y_2-y_1$.
Further, we have a list of all possible tuples $[\alpha,\beta,z_1,z_2,d_y,b_u]$, where $b_u$ is the corresponding (sharp) upper bound for $b$.
We proceed in two cases according to whether $d_z=1$ or not.

If $d_z>1$, then $c^2 \le c^{d_z}<b_U$, so $c \le \lfloor {b_U}^{1/2} \rfloor (<10^3)$. %( \le 729)$.
This together with $b \le b_u$ shows that both $b,c$ are so small that we can apply the same algorithm as in (i).

In the case where $d_z=1$, a short modular arithmetic leads to $x_1 \ge x_2$.
We further proceed in two cases according to whether $x_2=1$ or not.

If $x_2 \ge 2$, then $a^2 \le a^{x_2}=a^{\min\{x_1,x_2\}}<b_U$, so $a \le \lfloor {b_U}^{1/x_2} \rfloor \le \lfloor {b_U}^{1/2} \rfloor$.
Thus both $a,b$ are so small that we can apply the same method as in the case $d_z \ge 2$ and  verify that equation \eqref{final-eq} (with $d_z=1$) does not hold in any possible cases.

Finally, in the case where $d_z=1$ and $x_2=1$, we proceed as follows.
The case where $a>c^2$ is dealt with by a previous method since both $b,c$ are small enough as $c \le a^{1/2}<{b_U}^{1/2}$.
%$413$ seconds.
Thus suppose that $a<c^2$.
Since $z_2 \ge 4$, it follows from 2nd equation that $c^{z_2/2}-b^{y_2/2}=\frac{a}{c^{z_2/2}+b^{y_2/2}}<1$, which strictly restricts the value of $c$ in terms of $b,y_2$ and $z_2$.
Using this fact and $b$ is small, we can apply the algorithm described in Lemma \ref{red6} (see also Lemma \ref{red16}).

\vspace{0.2cm}(iii)
We are in the case where
\begin{equation} \label{red33}
x_1<x_2, \quad a^{x_1} \mid (b^{d_y}c^{d_z}-1)
\end{equation}
with $d_y=y_1-y_2$.
Further, we have a list named $clist$ containing all possible tuples $[\alpha,\beta,x_1,z_2,a_u,b_u,c_u,d_z,d_y]$, where $a_u, b_u, c_u$ are the corresponding upper bounds for $a,b,c$, respectively.
A quick check on $clist$ shows that if $d_z \ge 2$ or $d_y \ge 2$ then at least one of the bounds $b_u$ and $c_u$ is small (about $ 10^3$ or less) and the other is middle sized (about $5 \cdot 10^5$ or less).
Then this case can be handled similarly to (i) and (ii) with the parameters $c_{max}$ and $b_{max}$ are given as $c_{max}=c_u$ and
\[
b_{max}=\min\bigl\{b_u,  \lfloor {c_u}^{z_1/y_1}\rfloor, \lfloor {c_u}^{z_1/x_1}\rfloor, \lfloor {c_u}^{z_2/y_2}\rfloor, \lfloor {c_u}^{z_2/x_2}\rfloor \bigl\}.
\]
Moreover, if $d_z=d_y=1$, then both of these bounds are middle sized, while, unfortunately, the bound $a_u$ becomes large $(\approx 4\cdot10^{10})$.
Thus, we have to find another reduction procedure which avoids the use of $a_u$.
\if0
\vspace{0.2cm}{\it Case where $d_z \ge 2$ or $d_y \ge 2$.}
Since one of the bounds $b_u$ and $c_u$ is small, we can proceed in a similar way as in (i) or (ii) above.
We only indicate the corresponding value of $b_{max}$, as follows:
follow the method as in {\it Step I} of Lemma \ref{red7} to generate a list named $list1$ having elements of the form
\[
[\alpha,\beta,x_1,y_1,z_1,x_2,y_2,z_2,b_{max},c_{max},s_a,s_b],
\]
where $[s_a,s_b]$ is the signature of the pair $[a,b]$, $c_{max}=c_u$ and $b_{max}$ is defined as
\[
b_{max}:=\min\bigl\{b_u,  \lfloor {c_u}^{z_1/y_1}\rfloor, \lfloor {c_u}^{z_1/x_1}\rfloor, \lfloor {c_u}^{z_2/y_2}\rfloor, \lfloor {c_u}^{z_2/x_2}\rfloor \bigl\}.
\]
%The running time of this case was $146$ seconds.
\fi

Finally, we consider the case where $(d_z,d_y)=(1,1)$.
We follow the method applied in {\it Step I} of Lemma \ref{red7} to generate the corresponding list named $list1$ having elements satisfying $(d_z,d_y)=(1,1)$ of the form
\[
[\alpha,\beta,x_1,y_1,z_1,x_2,y_2,z_2,b_{max},c_{max},s_a,s_b].
\]
The pair $[s_a,s_b]$ is the signature of $[a,b]$ while the bounds $c_{max}$ and $b_{max}$ are defined by $c_{max}:=c_u$ and $b_{max}:=\min\bigl\{b_u, \lfloor {c_u}^{z_1/y_1}\rfloor\bigl\}$.
A quick look together with Lemma \ref{red5} shows that $list1$ does not include any element satisfying $(y_1,z_1)=(2,2)$.
Thus, if $y_1 \le z_1$, then $z_1>2$, which together with the inequality $bc>a^{x_1}$ by \eqref{red33} shows that $bc>a^{x_1}=c^{z_1}-b^{y_1} \ge c^{z_1}-b^{z_1}>c^{z_1-1}+b^{z_1-1} \ge c^2+b^2$, a contradiction.
Therefore, it remains to consider the case where $y_1>z_1$.
We note that since $y_1>z_1$ and the orders of magnitude of $b_u$ and $c_u$ are the same, the quantity $\lfloor {c_u}^{z_1/y_1}\rfloor$ is smaller than $b_u$, resulting in a sharper upper bound $b_{max}$ for $b$.
This observation is crucial in order to have a reasonable running time.
The remaining task can be dealt with similarly to cases (i) and (ii).

\vspace{0.2cm}
The total computational time of Section \ref{Sec-sieve-z1<z2_a=max} did not exceed 7 hours.

The conclusion of Sections \ref{Sec-sieve-z1<z2_c=max} and \ref{Sec-sieve-z1<z2_a=max} together is:
\begin{prop}\label{z1 eq z2}
$z_1=z_2.$
\end{prop}
In view of Propositions \ref{z1 ne z2} and \ref{z1 eq z2}, the proof of Theorem \ref{atmost2} is finally completed.

%%%%%%%%%%%%%%%%%%%%%%%%
\section{Concluding remarks}%%%%%%%%%
%%%%%%%%%%%%%%%%%%%%%%%%

Theorem \ref{atmost2} says that there is only one example which allows equation \eqref{abc} to have three solutions in positive integers.
On the other hand, a simple search in a suitably finite region using computer (cf.~\cite[Section 3]{ScSt_deb16}) finds a number of examples where there are two solutions to \eqref{abc} in positive integers $x,y$ and $z$, corresponding to the following set of equations:
\begin{gather}
%\begin{split}
5^{}+2^{2}=3^{2}, \ 5^{2}+2^{}=3^{3}; \nonumber\\
7^{}+2^{}=3^{2}, \ 7^{2}+2^{5}=3^{4}; \nonumber\\
3^{2}+2^{}=11^{}, \ 3^{}+2^{3}=11^{}; \nonumber\\
3^{3}+2^{3}=35^{}, \ 3^{}+2^{5}=35^{}; \nonumber\\
3^{5}+2^{4}=259^{}, \ 3^{}+2^{8}=259^{}; \nonumber\\
5^{3}+2^{3}=133^{}, \ 5^{}+2^{7}=133^{}; \nonumber\\
\label{atmost1:ex}3^{}+10^{}=13^{}, \ 3^{7}+10^{}=13^{3}; \\
89^{}+2^{}=91^{}, \ 89^{}+2^{13}=91^{2}; \nonumber\\
91^{2}+2^{}=8283^{}, \ 91^{}+2^{13}=8283^{}; \nonumber\\
3^{}+5^{}=2^{3}, \ 3^{3}+5^{}=2^{5}, \ 3^{}+5^{3}=2^{7};\nonumber\\
3^{}+13^{}=2^{4}, \ 3^{5}+13^{}=2^{8}; \nonumber\\
3^{}+13^{3}=2200^{}, \ 3^{7}+13^{}=2200^{}; \nonumber\\
2^{}+(2^k-1)^{}={2^k+1}^{}, \ 2^{k}+(2^k-1)^{2}=(2^k+1)^{2},\nonumber
%\end{split}
\end{gather}
where $k$ is any integer with $k \ge 2$.

While Theorem \ref{atmost2} is essentially sharp, as indicated by \eqref{atmost1:ex}, it is natural, in light of a lot of existing works on the solutions of equation \eqref{abc}, to believe that something rather stronger is true.
A formulation in this direction is posed by Scott and Styer \cite{ScSt_deb16}, which is regarded as a 3-variable generalization of \cite[Conjecture 1.3]{Be_Canad2001}, as follows:

\begin{conj}\label{atmost1}
For any fixed coprime positive integers $a,b$ and $c$ with $\min\{a,b,c\}>1,$ equation \eqref{abc} has at most one solution in positive integers $x,y$ and $z,$ except for those triples $(a,b,c)$ arising from \eqref{atmost1:ex}.
\end{conj}

There are many results in the literature which support this conjecture.
However, Conjecture \ref{atmost1} seems completely out of reach.
It is worth noting that it implies several unsolved problems to ask for the determination of the solutions of equation \eqref{abc} for some infinite families of $(a,b,c)$, including not only the conjecture of Sierpi\'nski and Je\'smanowicz on primitive Pythagorean triples, but also its generalization posed by Terai as mentioned in the first section.
Finally, we mention that Conjecture \ref{atmost1} seems not to be directly followed from some of well-known conjectures closely related to ternary Diophantine equations including generalized Fermat conjecture and any effective version of $abc$ conjecture.

\subsection*{Acknowledgements}
The first author would like to sincerely thank Nobuhiro Terai for not only his encouragement to the present work but also his continuous and kind support since the first author was a student.
The second author is pleased to express his gratitude to Clemens Fuchs and Volker Ziegler for the possibility to spend the period between February and June 2019 at the University of Salzburg, where a lot of progress on the present article has been made.
He is also very grateful to \'Akos Pint\'er, Lajos Hajdu and Attila B\'erczes for their continuous
encouragement and support.
We thank Mihai Cipu for his many remarks which improved an earlier draft.
We finally mention that the work of the present article would not have been possible without not only Noriko Hirata-Kohno's continuous support but also Lajos Hajdu's help and the hospitality of the University of Debrecen, thanks to which the first author had previously spent a month in 2014 at the University of Debrecen.

\end{document}